\documentclass[a4paper,12pt,oneside]{report}

\usepackage[left=5cm,right=3cm,top=2.5cm,bottom=3cm]{geometry}
\usepackage{amsmath}
\usepackage{amsthm}
\usepackage{amssymb}
\usepackage{enumerate}
\usepackage{graphicx}
\usepackage{fancyhdr}
\usepackage{setspace}
\newtheorem{definition}{Definition}
\newtheorem{theorem}[definition]{Theorem}
\newtheorem{corollary}[definition]{Corollary}
\newtheorem{lemma}[definition]{Lemma}
\newtheorem*{remark}{Remark}

\newtheorem{proposition}[definition]{Proposition}

\begin{document}

\pagenumbering{roman}

\thispagestyle{empty}
\begin{center}
\begin{minipage}{0.75\linewidth}
    \centering
    {\uppercase{University of Nottingham\par}}
    \vspace{0.5cm}
    \includegraphics[width=0.4\linewidth]{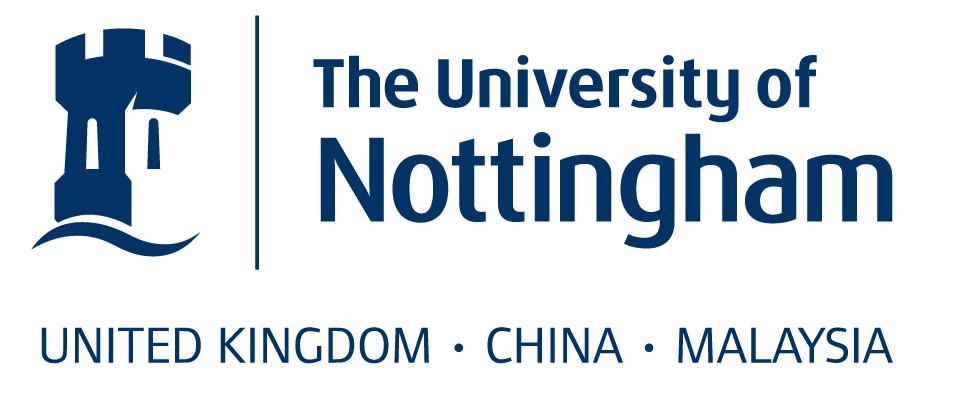}\par
    \vspace{0cm}
    {\uppercase{School of Mathematical Sciences\par}}
    \vspace{0.5cm}
    {\bf \Huge On the Right Nucleus of Petit Algebras\par}
    \vspace{3cm}
    {\Large Adam Owen\par}
    \vspace{3cm}
    {\Large A thesis submitted to the University of Nottingham for the degree of \\ DOCTOR OF PHILOSOPHY \par}
    \vspace{3cm}
    {\Large August 2021}
\end{minipage}
\end{center}
\clearpage

\begin{abstract}
Let $D$ be division algebra over its center $C$, let $\sigma$ be an endormorphism of $D$, let $\delta$ be a left $\sigma$-derivation of $D$, and let $R=D[t;\sigma,\delta]$ be a skew polynomial ring. We study the structure of a class of nonassociative algebras, denoted by $S_f$, whose construction canonically generalises that of the associative quotient algebras $R/Rf$ where $f\in R$ is right-invariant. \\
\\
We determine the structure of the right nucleus of $S_f$ when the polynomial $f$ is bounded and not right invariant and either $\delta = 0$, or $\sigma = {\rm id}_D$. As a by-product, we obtain a new proof on the size of the right nuclei of the cyclic (Petit) semifields $\mathbb{S}_f$.\\
\\
We look at subalgebras of the right nucleus of $S_f$, generalising several of Petit's results \cite{petit1966certains} and introduce the notion of semi-invariant elements of the coefficient ring $D$. The set of semi-invariant elements is shown to be equal to the nucleus of $S_f$ when $f$ is not right-invariant. Moreover, we compute the right nucleus of $S_f$ for certain $f$.\\
\\
In the final chapter of this thesis we introduce and study a special class of polynomials in $R$ called generalised A-polynomials. In a differential polynomial ring over a field of characteristic zero,  A-polynomials were originally introduced by Amitsur \cite{amitsur1954differential}. We find examples of polynomials whose eigenring  is a central simple algebra over the field $C \cap {\rm Fix}(\sigma) \cap {\rm Const}(\delta)$. 
\end{abstract}

\chapter*{Acknowledgements}
My full and sincere gratitude goes out to my supervisor Susanne Pumpl{\" u}n. You have been a wonderful mentor for the last three and a half years and your infinite patience and wisdom got me through the highs and lows of this process.\\
Thank you to Dan for being an excellent travel companion to conferences, coffee shops, and bars, and for always having time to listen to my ramblings.\\
I would also like to extend my thanks to John Sheekey and Jos{\' e} G{\' o}mez-Torrecillas whose correspondence proved to be extremely helpful.\\
Finally, I would like to say a special thank you to Holly and my parents for putting up with me and the maths and for your unconditional love and support.

\chapter*{Introduction}
Let $D$ be a unital associative ring, let $\sigma$ and $\delta$ be an injective endomorphism and a left $\sigma$-derivation of $D$, respectively, and let $R=D[t;\sigma,\delta]$ be a skew polynomial ring in the single indeterminate $t$. The ring $R$ was introduced by Ore in 1933 \cite{ore1933theory}. Since then it has been widely studied and its properties are well understood as a result (e.g. see \cite{mcconnell2001noncommutative} and \cite{goodearl2004introduction} for surveys). We exclusively consider the case that $D$ is a unital associative division ring, in which case any endomorphism of $D$ is necessarily injective.\\
\\
Ore showed that $R$ possesses a right division algorithm, that is for any given skew polynomials $f,g \in R$ (with $f \neq 0$), there exists a unique quotient $q \in R$, and a unique remainder $r \in R$, such that ${\rm deg}(r) < {\rm deg}(f)$ and $g=qf+r$. Consequently, every left ideal of $R$ is principal, i.e. has the form $Rg$ for some $g \in R$. Moreover, if $\sigma$ is an automorphism, then $R$ also possesses the analogous left division algorithm, hence every right ideal of $R$ is also principal, i.e. of the form $gR$ for some $g \in R$, and $R$ is a principal ideal domain \cite{jacobson2009finite}. We note that what we call a right division algorithm, Jacobson calls left and vice versa.\\
\\
For each skew polynomial $f \in R$ of degree $m \geq 1$, the eigenring of $f$ is defined by the set $$\mathcal{E}(f) = \{ g \in R : {\rm deg}(g) < m \text{ and } fg \in Rf \}$$ which is an associative algebra, and if $f$ is an irreducible polynomial, then $\mathcal{E}(f)$ is a division algebra. For bounded polynomials $f \in R$ the irreducible factors of $f$ in the ring $R$ are in one to one correspondence with the nontrivial zero divisors in $\mathcal{E}(f)$, in particular if $f$ is bounded and $\mathcal{E}(f)$ is a division algebra, then $f$ is irreducible \cite{gomez2012basic,gomez2013computing}. Therefore the eigenring of a skew polynomial $f$ often appears whenever a factorisation of $f$ in $R$ is investigated.\\
\\
In 1966 Petit introduced a new class of unital nonassociative algebras, whose construction canonically generalises the construction of the associative quotient algebras $R/Rf$ when factoring by a right-invariant (i.e. two-sided) polynomial $f \in R$ in \cite{petit1966certains}. Let $f \in R$ have degree $m \geq 1$ and consider the additive subgroup $$R_m = \{g \in R: {\rm deg}(g) < m \},$$ of $R$. When endowed with the multiplication $$g \circ h = gh \, {\rm mod}_r f,$$ where ${\rm mod}_r f$ denotes the remainder upon right division by $f$ in $R$, $R_m$ forms a unital nonassociative algebra over the field $F=C(D) \cap {\rm Fix}(\sigma) \cap {\rm Const}(\delta)$. These so-called {\it Petit algebras} are denoted by $S_f$ where $f$ is the polynomial used in the construction, or often by $R/Rf$ due to the $R$-module structure of $S_f$. Petit algebras were studied in detail in \cite{petit1966certains} and more recently in \cite{brown2018}, \cite{brown2018nonassociative} for example, and in \cite{lavrauw2013semifields}, \cite{sheekey2018new} over fields of characteristic $p>0$ (in which case $\delta = 0$ w.l.o.g.). Petit's method can be used to construct unital nonassociative division algebras, for example the Petit algebra $S_f$ with $f(t) = t^2 - i \in \mathbb{C}[t;\overline{\phantom{x}}]$, where $\overline{\phantom{x}}$ denotes complex conjugation, was described earlier by Dickson and forms the earliest example of a nonassociative division algebra \cite{dickson1906linear}.\\
\\
Constructions of codes via classical algebraic methods are well known. In recent years, the nonassociative algebras $S_f$ have been employed in the construction of space-time block codes, $(f,\sigma,\delta)$-codes, and wire tap codes, to name a few, when $S_f$ is an algebra over a number field (e.g.~ \cite{pumplun2015finite,pumpluen2015note,berhuy2013introduction,ducoat2015skew,ducoat2015lattice,oggier2012quotients}. In particular, we note their use in the construction of fast-decodable fully diverse space-time block codes in \cite{pumpluen2015fast,pumplun2015nonassociative,steele2012mido}, due to the fact that the right nucleus of $S_f$ is large. This motivates our investigation into the structure of ${\rm Nuc}_r(S_f)$ from a purely algebraic point of view, since a better understanding of the algebras $S_f$ and their nuclei may lead to new developments in coding theory.\\
\\
For $f$ of degree at least 2, Petit showed that the eigenring of $f$ is equal to the right nucleus of $S_f$ \cite{petit1966certains}. Thus we are able to analyse the structure of the nonassociative algebra $S_f$ and its right nucleus via classical methods from the theory of $R$-modules and the theory of associative algebras. Alternatively, one may say that this identification of the right nucleus of $S_f$ as the eigenring of $f$ allows us to take the novel approach to investigate the structure of the eigenring of $f$ in the context of nonassociative algebras. \\
\\
The structure of this thesis is as follows:\\
\\
Chapter 1 introduces the necessary notation, definitions and background theory used throughout. In particular, we recall some properties of the skew polynomial ring $R$, recount the theory of bounded polynomials in $R$, and describe Petit's construction of the algebras $S_f$. In Chapter 2, we focus on the case that $R=D[t;\sigma]$, where $\sigma$ is an automorphism of $D$ of finite inner order $n$. For any bounded skew polynomial $f \in R$ such that its bound $f^*$ is a maximal two-sided element of $R$, i.e. the left ideal $Rf^*$ is a two-sided maximal ideal, the ring $A=R/Rf^*$ is a simple Artinian ring, hence the left $A$-module $R/Rf$ is semisimple, i.e. is the direct sum of simple modules.  Motivated by this we employ methods from the Artin-Wedderburn theory of semisimple Artinian rings and algebras (e.g.~\cite{lang2004algebra}), and Jacobson's theory of noncommutative principal ideal domains \cite{jacobson1943theory}, to study both the $A$-module and $R$-module structure of $S_f$. Moreover, the endomorphism ring ${\rm End}_R(R/Rf)$ is known to be equal to the eigenring of $f$ (e.g.~\cite{gomez2013computing}), hence we obtain a description of the right nucleus of the Petit algebra $S_f$ via classical methods.\\
\\
Chapter 3 follows the same format as Chapter 2, however, this time we employ the classical methods described above in order to analyse the module structure of $S_f$ when $R=D[t;\delta]$ for $\delta$ an inner derivation of $D$ (when ${\rm Char}(D)=0$) or an algebraic derivation of $D$ of finite degree (when ${\rm Char}(D)=p>0$). Chapter 4 consists of a generalisation of the results of Chapters 2 and 3, under the assumption that $f^*$ completely factorises into a product of irreducibles in the center of $R$
, which are not equal to one another up to scalar multiplication by elements in $D^{\times}$. We focus here on the case $R=D[t;\sigma]$, however, the results of Chapter 4 also hold analogously in the case $R=D[t;\delta]$. We discard the case $R=D[t;\sigma,\delta]$ in Chapters 2 through 4, since we require that $D$ has finite degree as a central division algebra over $C$, hence $R$ is either a twisted polynomial ring or a differential polynomial ring \cite[Theorem 1.1.21]{jacobson2009finite}. \\
\\
Moving on, we study the algebraic properties of $S_f$ and its nuclei. We introduce the notion of (right) semi-invariant elements in the coefficient ring $D$ with respect to a skew polynomial $f \in R$. An element $c \in D$ is {\it (right) semi-invariant} with respect to $f \in R$ if $fc \in Rf$. We show that the set of such elements $L=L^{(\sigma,\delta,f)}$ is equal to ${\rm Nuc}(S_f)$ whenever $f$ is not right-invariant, and if $f$ is right-invariant, then $L = D$. Moreover, we determine sufficient conditions for the indeterminate $t$ and its powers to lie in the right nucleus of $S_f$, and we generalise Petit's result \cite[(16)]{petit1966certains} providing  necessary and sufficient conditions for $t$ to lie in ${\rm Nuc}_r(S_f)$ when $R=D[t;\sigma]$. We make use of these new structural results to generate subalgebras of the right nucleus. In the special case $f(t)=t^m-a \in D[t;\sigma]$ we show that $f(t)$ is right-invariant in $L[t;\sigma]$ and that ${\rm Nuc}_r(S_f)$ is equal to the associative quotient algebra $L[t;\sigma]/L[t;\sigma]f(t)$. \\
\\
The eigenring appears implicitly in classical papers by Amitsur \cite{amitsur1953noncommutative, amitsur1954differential, amitsur1955generic}, wherein Amitsur shows that any central simple algebra that is split by some field $K$ is isomorphic to the eigenring of some skew polynomial in $K[t;\delta]$, where $(K,\delta)$ is a suitable differential field. Any such polynomial is called an {\it A-polynomial}, however, the property of being an A-polynomial has been somewhat ignored since its introduction. Motivated by Amitsur's work we introduce the notion of {\it generalised A-polynomials} in $R=D[t;\sigma,\delta]$ in Chapter 7. These are non-constant polynomials in $R$ whose eigenrings are central simple algebras over $F$. We utilise the results of the previous chapters to provide necessary and sufficient conditions of $f \in R$ to be a generalised A-polynomial in a variety of special cases. 

\tableofcontents

\thispagestyle{empty}

\chapter{Preliminaries and Notation}
\pagenumbering{arabic}
Although we have tried to make this thesis as comprehensive as possible, we assume that the reader has some knowledge of the basic theory of rings, modules over rings, and algebras over fields, as well as some elementary linear algebra, and the Galois theory of fields. 

\section{Modules, Endomorphism Rings and Annihilators}
Let $R$ denote a unital, associative ring and $M$ and $M^{\prime}$ be left $R$-modules.
Throughout this thesis, all rings and modules are assumed to be unital and we take $R$-module to mean a left $R$-module, unless stated otherwise\footnote{All of the notions/results in this section apply analogously if we replace all left $R$-modules with right $R$-modules, via symmetric arguments.}. \\
\\
The set of $R$-module homomorphisms from $M$ to $M^{\prime}$ forms an abelian group under addition of maps, denoted by ${\rm Hom}_R(M,M^{\prime})$. In fact, the set of endomorphisms of $M$ 
is a ring with multiplication defined by the composition of maps, and is denoted by ${\rm End}_R(M)$. The structure of ${\rm End}_R(M)$ encodes many of the properties of the module $M$, for example, Schur's Lemma states that if $M$ is a simple $R$-module, then ${\rm End}_R(M)$ is a division ring, see for example \cite[pg.~396]{adkins2012algebra}.\\
\\
It will be useful to recall some of the basic properties of ${\rm End}_R(M)$, especially when $M$ is a direct sum of finitely many left $R$-modules, or when $M$ is a simple left $R$-module.
\begin{lemma}\cite[Exercise 6.7.2]{sullivan2004}\label{Intro: Lemma 1}
Suppose that $M \cong M^{\prime}$. Then ${\rm End}_R(M) \cong {\rm End}_R(M^{\prime})$.
\end{lemma}

\begin{lemma}\cite[Lemma 6.7.5]{sullivan2004}\label{Intro: Lemma 2}
Let $M_1, M_2, \dots, M_k, M_1^{\prime}, M_2^{\prime}, \dots, M_l^{\prime}$ be left $R$-modules, and let $M = M_1 \oplus M_2 \oplus \cdots \oplus M_k$ and $M^{\prime}= M_1^{\prime} \oplus M_2^{\prime} \oplus \cdots \oplus M_l^{\prime}$. 
\begin{enumerate}[(i)]
\item ${\rm Hom}_R(M,M^{\prime}) \cong \left(\begin{array}{ccc}
   {\rm Hom}_R(M_1,M_1^{\prime}) & \cdots & {\rm Hom}_R(M_k,M_1^{\prime}) \\
   \vdots & \ddots & \vdots \\
   {\rm Hom}_R(M_1,M_l^{\prime}) & \cdots & {\rm Hom}_r(M_k,M_l^{\prime})\\
\end{array}\right)$ as Abelian groups. 
\item If $M_1=M_2=\cdots =M_k$, then $${\rm End}_R(M) = {\rm End}_R(M_1^{\oplus k}) \cong M_k({\rm End}_R(M_1))$$ as rings.
\end{enumerate}
\end{lemma}

\begin{lemma}\label{Intro: Lemma 3}
Let $S_1,S_2,\dots,S_k$ be mutually non-isomorphic simple left $R$-modules. 
\begin{enumerate}[(i)]
\item If $i \neq j$, then $${\rm Hom}_R(S_i,S_j) = 0.$$
\item If $i \neq j$ and $m,n \in \mathbb{N}$, then $${\rm Hom}_R(S_i^{\oplus m},S_j^{\oplus n}) = 0.$$
\item Let $m_1,m_2,\dots,m_k \in \mathbb{N}$, then $${\rm End}_R(\bigoplus\limits_{i=1}^k S_i^{\oplus m_i}) \cong \bigoplus\limits_{i=1}^k {\rm End}_R(S_i^{\oplus m_i}).$$
\end{enumerate}   
\end{lemma}
\begin{proof}
\begin{enumerate}[(i)]
\item Let $\phi \in {\rm Hom}_R(S_i,S_j)$ be a nonzero map. Since ${\rm ker}(\phi)$ is a submodule of $S_i$, ${\rm ker}(\phi) = 0$ (as $S_i$ is simple). Similarly, ${\rm im}(\phi)$ is a submodule of $S_j$, thus we are forced to take ${\rm im}(\phi)=S_j$. Hence $\phi$ is a bijection, i.e. $S_i \cong S_j$. This is a contradiction, and so $\phi = 0$. Therefore ${\rm Hom}_R(S_i,S_j) = 0$ as claimed.
\item By Lemma \ref{Intro: Lemma 2}, $${\rm Hom}_R(S_i^{\oplus m},S_j^{\oplus n}) \cong \left(\begin{array}{ccc}
   {\rm Hom}_R(S_i,S_j) & \cdots & {\rm Hom}_R(S_i,S_j) \\
   \vdots & \ddots & \vdots \\
   {\rm Hom}_R(S_i,S_j) & \cdots & {\rm Hom}_R(S_i,S_j)\\
\end{array}\right),$$ is the ring of $m \times n$ matrices with entries in ${\rm Hom}_R(S_i,S_j)$; and by (i), ${\rm Hom}_R(S_i,S_j) = 0$. We conclude that ${\rm Hom}_R(S_i^{\oplus m},S_j^{\oplus n}) = 0.$
\item By Lemma \ref{Intro: Lemma 2}, $${\rm End}_R(\bigoplus\limits_{i=1}^k S_i^{\oplus m_i}) \cong 
\left(\begin{array}{ccc}
   {\rm Hom}_R(S_1^{\oplus m_1},S_1^{\oplus m_1}) & \cdots & {\rm Hom}_R(S_k^{\oplus m_k},S_1^{\oplus m_1}) \\
   \vdots & \ddots & \vdots \\
   {\rm Hom}_R(S_1^{\oplus m_1}, S_k^{\oplus m_k}) & \cdots & {\rm Hom}_R(S_k^{\oplus m_k},S_k^{\oplus m_k})\\
\end{array}\right).$$
By (ii) and the fact that ${\rm End}_R(S_i^{\oplus m_i}) = {\rm Hom}_R(S_i^{\oplus m_i},S_i^{\oplus m_i})$, this becomes
\begin{align*}
{\rm End}_R(\bigoplus\limits_{i=1}^k S_i^{\oplus m_i}) &\cong \left(\begin{array}{cccc}
   {\rm End}_R(S_1^{\oplus m_1}) & 0 & \cdots & 0 \\
   0 & {\rm End}_R(S_2^{\oplus m_2}) & \cdots & \vdots \\
   \vdots & \cdots & \ddots & 0 \\
   0 & \cdots & 0 & {\rm End}_R(S_k^{\oplus m_k})\\
\end{array}\right)
\\ &\cong \bigoplus\limits_{i=1}^k{\rm End}_R(S_i^{\oplus m_i}).
\end{align*} 
\end{enumerate}
\end{proof}

Let $N$ be a subset of $M$. The annihilator of $N$ is defined by ${\rm Ann}_R(N) = \{ r \in R : ra = 0,\forall a \in N\}$, and it is a left ideal of $R$. If $N$ is a submodule of $M$, then ${\rm Ann}_R(N)$ is a two-sided ideal of $R$. In particular, the annihilator $M^0 = {\rm Ann}_R(M)$ is a two-sided ideal of $R$. For lack of a proper reference, a proof of the following well known result is included.

\begin{lemma}\label{Intro: Lemma 4}
If $I$ is a two-sided ideal of $R$ contained in $M^0$, then $M$ is also an $R/I$-module, and $${\rm End}_R(M) = {\rm End}_{R/I}(M).$$ In particular $M$ is an $R/M^0$-module and $${\rm End}_R(M) = {\rm End}_{R/M^0}(M).$$
\begin{proof}
Let $\cdot: R\times M \longrightarrow M$ denote the action of $R$ on $M$. It is well known that we can view the left $R$-module $M$ as a left $R/I$-module via a new action $\star_I : R/I \times M \longrightarrow M$, defined by
\begin{equation}\label{Equation 1}
\overline{r} \star_I m = (r+I) \star_I m = r \cdot m
\end{equation} where $\overline{r}$ denotes the coset $r+I$ for any $r \in R$.\\
\\
Let $\phi$ be an endomorphism of $M$. Then $\phi$ is $R$-linear if and only if $$r\cdot \phi(m) = \phi(r \cdot m)$$ for all $m \in M$, and all $r \in R$, which is true if and only if $$\overline{r}\star_I \phi(m) = \phi(\overline{r}\star_I m)$$ for all $m \in M$ and all $\overline{r} \in R/I$ by (\ref{Equation 1}). We conclude that $\phi$ is $R$-linear if and only if $\phi$ is $R/I$-linear, hence ${\rm End}_R(M)={\rm End}_{R/I}(M)$. The second claim follows by choosing $I=M^0$.
\end{proof}
\end{lemma}

\section{Nonassociative Algebras and Associative Central Simple Algebras}
Let $F$ be a field 
and $A$ be a vector space over $F$. We call $A$ a {\it (nonassociative) algebra} over $F$ (alternatively an {\it $F$-algebra}) if there exists a bilinear map $A \times A \longrightarrow A$, $(a,b) \mapsto ab$ which we call the {\it multiplication} of $A$. We assume throughout that any $F$-algebra $A$ is {\it unital}, i.e. there exists a multiplicative identity $1 \in A$ such that $1a=a1=a$ for all $a \in A$.\\
\\
The {\it associator} of $A$ is the bracket $[x,y,z] = (xy)z - x(yz)$. If $[x,y,z]=0$ for all elements $x,y,z \in A$, then $A$ is an {\it associative} algebra.  We define the {\it left nucleus} of $A$ by ${\rm Nuc}_l(A) = \lbrace x \in A : [x,A,A]=0 \rbrace$, the {\it middle nucleus} of $A$ by ${\rm Nuc}_m(A) = \lbrace x \in A : [A,x,A]=0 \rbrace$, and the {\it right nucleus} of $A$ by ${\rm Nuc}_r(A) = \lbrace x \in A : [A,A,x]=0 \rbrace$. The intersection ${\rm Nuc}(A) = {\rm Nuc}_l(A) \cap {\rm Nuc}_m(A) \cap {\rm Nuc}_r(A)$ is called the {\it nucleus} of $A$, and the sets ${\rm Nuc}_l(A)$, ${\rm Nuc}_m(A)$, ${\rm Nuc}_r(A)$, and ${\rm Nuc}(A)$ are associative subalgebras of $A$.\\
\\
 The {\it commutator bracket} on $A$ is defined by $[x,y]=xy-yx$ for all $x,y \in A$. Let $B \subset A$ be a nonempty subset of $A$, then the {\it centraliser of $B$ in $A$} is the set ${\rm {Cent_A}}(B) = \lbrace x \in A : [x,y]=0 \text{ for all } y \in B \rbrace$. We call the centraliser of $A$ in $A$ the {\it commutator} of $A$, and write ${\rm{Comm}}(A)={\rm {Cent_A}}(A)$. The intersection ${\rm C}(A) = {\rm Comm}(A) \cap {\rm Nuc}(A)$ consisting of all elements in $A$ which both commute and associate with all others is called the {\it center} of $A$. The center of any algebra $A$ is an associative, unital,  commutative subalgebra of $A$ containing the field $F$, since we can identify $F$ with $F1 \subseteq {\rm C}(A)$. \\
\\
 For $a \in A$ we define the {\it left (resp. right) multiplication map} $L_a : A \longrightarrow A$ (resp. $R_a: A \longrightarrow A$) by $L_a(x) = ax$ (resp. $R_a(x) = xa$) for all $x \in A$. For $a \neq 0$, if $L_a$ (resp. $R_a$) is bijective, then $a$ is called {\it left (resp. right) invertible}, and when $a$ is both left and right invertible then $a$ is said to be {\it invertible}. When all nonzero elements in $A$ are left (resp., right) invertible, then $A$ is called a {\it left (resp., right) division algebra}. We call $A$ a {\it division algebra} if it is both a left and right division algebra. If $A$ has finite dimension over $F$, then $A$ is a division algebra if and only if $A$ has no non-trivial zero divisors. Unital nonassociative division algebras with a finite number of elements are also known as {\it (finite) semifields}.\\
\\
An algebra $A$ is called {\it simple} if it contains no proper nontrivial two-sided ideals. If $A$ is a unital simple $F$-algebra, then ${\rm C}(A)$ is a field extension of $F$, and we may regard $A$ as an algebra over $C(A)$. If additionally $[A:C(A)]$ is finite, then we call $A$ a {\it central simple algebra} over $C(A)$. Any division ring $D$ contains no nontrivial two-sided ideals other than itself, and additionally if $D$ is assumed to be unital, then its center ${\rm C}(D)$ is a field. Thus any unital division ring $D$ such that $[D:C(D)] < \infty$ is a central simple algebra over $C(D)$. In this case, $D$ is also called a {\it central division algebra} over $C(D)$ (e.g. \cite{albert1952nonassociative}). It is well known that if $A$ is an associative central simple algebra over $C(A)$, then $[A:C(A)]$ is equal to the square of a positive integer, and $A \cong M_k(D)$ for $k$ a positive integer. Moreover, then $D$ is an associative central division algebra over $C(A)$, with $D$ unique up to isomorphism. The {\it degree} of $A$ is ${\rm deg}(A) = \sqrt{[A:C(A)]}$, and the {\it index} of $A$ is ${\rm ind}(A) = \sqrt{[D:C(A)]} = {\rm deg}(D)$ (e.g. \cite[Chapter VIII]{bourbaki2003elements}). When referring to a central simple algebra in this thesis, we assume that it is associative unless otherwise stated.

\section{Skew Polynomial Rings}
Let $D$ be a unital, associative division algebra over its center $C$, let $\sigma$ be an endomorphism of $D$, and let $\delta$ be a left $\sigma$-derivation of $D$, i.e. $\delta$ is an additive map on $D$ satisfying the Leibniz product rule $$\delta(xy) = \sigma(x)\delta(y) + \delta(x)y,$$ for all $x,y \in D$. An element $a \in D$ is said to be {\it fixed} by $\sigma$ if it satisfies $\sigma(a) = a$, and the set ${\rm Fix}(\sigma)$ of such elements is a division subring of $D$. If $\sigma$ is an automorphism, then  $\sigma$ is said to be an {\it inner automorphism} if there exists $u \in D^{\times}$ such that $\sigma(a) = uau^{-1}$ for all $a \in D$, otherwise $\sigma$ is said to be an {\it outer automorphism}. For $u \in D^{\times}$ we write $\iota_u$ to denote the inner automorphism of $D$ defined by $a \mapsto uau^{-1}$. If there exists $n \in \mathbb{Z}^+$ such that $\sigma^n = \iota_u$ for some $u \in D^{\times}$, and $\sigma^i$ is an outer automorphism for $1 \leq i < n$, then $\sigma$ is said to have \emph{finite inner order} $n$. If no such integer $n$ exists, then $\sigma$ is said to have \emph{infinite inner order}.
We say that $a \in D$  is a {\it $\delta$-constant} if $\delta(a) = 0$. The set of all $\delta$-constants also forms a division subring of $D$ denoted by ${\rm Const}(\delta)$. The derivation $\delta$ is said to be an {\it inner derivation} if there exists $c \in D$ such that $\delta(a) = [c,a] = ca - ac$ for all $a \in D$, otherwise $\delta$ is called an {\it outer derivation}. For $c \in D$, $\delta_c$ denotes the inner derivation of $D$ defined by $\delta_c(a) = [c,a]$ for all $a \in D$.
 \\
\\
The {\it skew polynomial ring} $R=D[t;\sigma,\delta]$ is the set of polynomials $$a_mt^m + a_{m-1}t^{m-1} + \dots + a_1t + a_0$$ in the indeterminate $t$, with coefficients $a_i \in D$, endowed with term-wise addition, and multiplication defined by $$ta = \sigma(a)t + \delta(a),$$ for all $a \in D$. Under this addition and multiplication, $R$ forms a unital, associative ring. A simple induction yields $$(at^j)(bt^k) = \sum\limits_{i=0}^j a \Delta_{j,i}(b) t^{i+k}$$ for all $a,b \in D$ and $j,k \in \mathbb{N}$, where the maps $\Delta_{j,i}:D \longrightarrow D$ are defined recursively by the relation $$\Delta_{j,i} = \delta(\Delta_{j-1,i}) + \sigma(\Delta_{j-1,i-1})$$ with $\Delta_{0,0} = {\rm id}_D$, $\Delta_{1,0} = \delta$, and $\Delta_{1,1} = \sigma$. If $\delta = 0$, then $R$ is called a {\it twisted polynomial ring}, denoted $R = D[t;\sigma]$, and we have $\Delta_{j,i} = 0$ for $j \neq i$, and $\Delta_{j,j} = \sigma^j$. If $\sigma = {\rm  id}_D$, then $R$ is called a {\it differential polynomial ring}, denoted $R = D[t;\delta]$, and $\Delta_{j,i}= \binom{j}{i}\delta^{j-i}$. Finally, if both $\sigma = {\rm id}$ and $\delta = 0$, then $R = D[t]$ is the ring of left polynomials with coefficients in $D$, and $\Delta_{j,i} = 0$ for all $i \neq j$, and $\Delta_{j,j} = {\rm id}_D$; in particular $ta = at$ for all $a \in D$.\\
\\
We refer the reader to \cite{mcconnell2001noncommutative}, \cite{goodearl2004introduction}, or \cite{ore1933theory} for an introduction to the theory of skew polynomial rings.\\
\\
For $f(t) = a_mt^m + a_{m-1}t^{m-1} + \dots + a_1t + a_0$ with $a_m \neq 0$, the {\it degree} of $f$, denoted by ${\rm deg}(f)$, is $m$, and by convention ${\rm deg}(0) = - \infty$. In the particular case that $a_m = 1$, we call $f$ {\it monic}. The degree of polynomials in $R$ satisfies ${\rm deg}(fg) = {\rm deg}(f) + {\rm deg}(g)$ and ${\rm deg}(f+g) \leq {\rm max}({\rm deg}(f),{\rm deg}(g))$ for all $f,g \in R$. A polynomial $f \in R$ is called {\it reducible} if $f=gh$ for some $g,h \in R$ such that ${\rm deg}(g),{\rm deg}(h) < {\rm deg}(f)$, otherwise we call $f$ {\it irreducible}. A polynomial $f \in R$ is called {\it right (resp. left) invariant} if $fR \subseteq Rf$ (resp. $Rf \subseteq fR$), i.e. $Rf$ (resp. $fR$) is a two-sided ideal of $R$. We call $f$ {\it invariant} if it is both right and left invariant. \\
\\
The ring $R=D[t;\sigma,\delta]$ is a left Euclidean domain, and as such there is a right  Euclidean division algorithm in $R$: for any $f,g \in R$ with $f \neq 0$, then there exist unique polynomials $q,r \in R$ such that $$g = qf + r,$$ with ${\rm deg}(r) < {\rm deg}(f)$. As a result $R$ is a left principal ideal domain. If $\sigma$ is an automorphism of $R$, then $R$ is also a right Euclidean domain, hence there is also a left Euclidean division algorithm in $R$, and $R$ is also a right principal ideal domain, i.e. $R$ is a 
principal ideal domain. Henceforth we write PID to mean principal ideal domain, and left (resp. right) PID to mean left (resp. right) principal ideal domain.\\
\\
Now, since $R$ is a left PID, every left ideal of $R$ is generated by a single skew polynomial in $R$, that is, for any left ideal $I$, there exists a polynomial $f \in R$, such that $I=Rf$. The {\it left idealiser} of a polynomial $f$ in $R$ is defined by the set $$\mathcal{I}(f) = \lbrace g \in R : fg \in Rf \rbrace,$$ which is the largest subring of $R$ within which $Rf$ is a two-sided ideal. We define the {\it eigenring} of $f$ to be the associative quotient ring $$\mathcal{E}(f) = \frac{\mathcal{I}(f)}{Rf}= \lbrace g \in R : {\rm deg}(g) < m \text{ and } fg \in Rf \rbrace.$$ A nonzero skew polynomial $f \in R$ is said to be {\it bounded} if there exists another nonzero skew polynomial $f^{\star} \in R$, called a {\it bound} of $f$, such that $Rf^{\star}$ is the unique largest two-sided ideal of $R$ contained in the left ideal $Rf$. Equivalently, a nonzero polynomial in $f \in R$ is said to be bounded, if there exists a right-invariant polynomial $f^{\star} \in R$, which is called a bound of $f$, such that $$Rf^{\star} = {\rm Ann}_R(R/Rf) \neq \{ 0 \}.$$ Recall that the annihilator ${\rm Ann}_R(R/Rf)$ of the left $R$-module $R/Rf$ is always a two-sided ideal of $R$. When the skew polynomial $f$ is bounded, it is well known that nontrivial zero divisors in the eigenspace of $f$ are in one-to-one correspondence with proper right factors of $f$ in $R$:

\begin{theorem}\cite[Lemma 3, Proposition 4]{gomez2012basic}\label{Intro: Theorem 1}
Let $f \in R$ have positive degree. Then the following are satisfied:
\begin{enumerate}[(i)]
\item If $f$ is bounded and $\sigma$ is an automorphism of $D$, then $f$ is irreducible if and only if $\mathcal{E}(f)$ has no non-trivial zero divisors.
\item Each non-trivial zero divisor $q$ of $f$ in $\mathcal{E}(f)$ gives a proper factor ${\rm gcrd}(q,f)$ of $f$.
\end{enumerate}
\end{theorem} 

Carcanague determines sufficient conditions for all non-zero polynomials in the skew polynomial ring $R$ to be bounded \cite[Theorem IV]{carcanague1969quelques}. The following result appears in \cite[Proposition 21]{brown2018nonassociative}, and is an immediate Corollary to \cite[Theorem IV]{carcanague1969quelques}:

\begin{proposition}\cite[Proposition 21]{brown2018nonassociative}\label{Intro: Proposition 1}
If $\sigma$ is an automorphism of $D$, and $R=D[t;\sigma,\delta]$, then the following are equivalent:
\begin{enumerate}
\item $R$ has finite rank over its center;
\item $D$ has finite rank over the field $F_0 =  C \cap {\rm Fix}(\sigma) \cap {\rm Const}(\delta)$.
\end{enumerate}
Moreover, if (1) and (2) are satisfied, then all non-zero polynomials in $R$ are bounded, and if $f$ is irreducible, then $S_f$ is a division algebra.
\end{proposition}

For $\sigma = {\rm id}$, this is \cite[Proposition 3]{pumpluen2018algebras}. In many cases, the skew polynomial ring $R=D[t;\sigma,\delta]$ is isomorphic to either a twisted polynomial ring, or a differential polynomial ring. In fact, if $D$ has finite dimension as an algebra over its center $C$, then $R=D[t;\sigma,\delta]$ is either a twisted polynomial ring or a differential polynomial ring \cite[Theorem 1.1.21]{jacobson2009finite}.  In many parts of this thesis, we will require $D$ to be finite dimensional over $F_0=C \cap {\rm Fix}(\sigma) \cap {\rm Const}(\delta)$, so that all polynomials in $R=D[t;\sigma,\delta]$ are bounded; in this case, it is necessary that $[D:C]$ be finite. Hence, in many places we lose no generality when we examine the special cases $R=D[t;\sigma]$ and $R=D[t;\delta]$ separately.\\
\\
We obtain the following Corollary to Proposition \ref{Intro: Proposition 1}, describing special cases for which all non-zero polynomials in $R$ are bounded:

\begin{corollary}\label{Intro: Corollary 1}
Let $D$ be a central division algebra over $C$ of degree $d$.
\begin{enumerate}[(i)]
\item If $\sigma$ has finite inner order, then all non-zero polynomials in $R = D[t;\sigma]$ are bounded.
\item If $D$ has characteristic zero, and $\delta$ is an inner derivation, then all non-zero polynomials in $R = D[t;\delta]$ are bounded.
\item If $D$ has prime characteristic, and $\delta\vert_C$ is algebraic, then all non-zero polynomials in $R = D[t;\delta]$ are bounded.
\end{enumerate}
\begin{proof}
We refer the reader to Sections 1.4 and 1.5 of \cite{jacobson2009finite}, in which Jacobson shows that $D$ has finite dimension as an $F$-vector space, hence finite rank as an $F$-module in cases (i)-(iii). The result follows immediately by Proposition \ref{Intro: Proposition 1}.
\end{proof}
\end{corollary}

Since any bound $f^{\star}$ of a non-zero polynomial $f$ in $R$ is a right-invariant polynomial in $R$, we recall some results which determine the right-invariant polynomials in  $R$, when $R$ is either a twisted polynomial ring, or a differential polynomial ring.\\
\\
Independently, Jacobson \cite{jacobson2009finite} and Petit \cite{petit1966certains} determine the right-invariant polynomials in $R=D[t;\sigma]$ and the center of $R$:

\begin{proposition}(\cite[(15)]{petit1966certains}, \cite[Theorem 1.1.22]{jacobson2009finite})\label{Intro: Proposition 2}
\begin{enumerate}[(i)]
\item The polynomial $f(t) = t^m - \sum_{i=0}^{m-1}a_it^i \in D[t;\sigma]$ is right-invariant in $D[t;\sigma]$ if and only if $a_i \in {\rm Fix}(\sigma)$ and $\sigma^m(c)a_i = a_i\sigma^i(c)$ for all $i \in \lbrace 0,1,\dots,m-1 \rbrace$ and for all $c \in D$.
\item $f(t)$ is right-invariant in $D[t;\sigma]$ if and only if $f(t)=ag(t)t^s$ for some $a\in D^\times$, $g(t)\in C(R)$ and
integer $s\geq0$.
\item If no non-zero power of $\sigma$ is an inner automorphism then $C(R) = F$.
\item If $\sigma$ has finite inner order $n$ with $\sigma^n = \iota_u$ for some $u \in D^{\times}$, then $C(R)$ is equal to the set of polynomials of the form $$\gamma_0 + \gamma_1 u^{-1}t^n + \gamma_2 u^{-2}t^{2n} + \dots + \gamma_su^{-s}t^{sn}$$ for some $s \in \mathbb{N} \cup \lbrace 0 \rbrace$ and $\gamma_i \in F$ such that $\gamma_iu^{-i} \in {\rm Fix}(\sigma)$. Moreover, if $n$ is also the order of $\sigma\vert_C$, then $u$ can be taken from ${\rm Fix}(\sigma)$, and $$C(R) = F[u^{-1}t^n].$$ The last situation holds if $[D:C] < \infty$.
\end{enumerate}
\end{proposition}

Also, Amitsur \cite[Theorem 1.1.32]{jacobson2009finite} determines the right-invariant polynomials in $R = D[t;\delta]$ and its center:

\begin{theorem}\cite[Theorem 1.1.32]{jacobson1943theory}\label{Intro: Theorem 2}
\begin{enumerate}[(i)]
\item The right-invariant polynomials of $R=D[t;\delta]$ are the elements $ap(t)$ where $a \in D$ and $q(t) \in C(R)$.
\item For $D$ of characteristic zero, the following are satisfied:
\begin{enumerate}
\item If $\delta = \delta_c$ for some $c \in D^{\times}$, then $$C(R) = F[t-c]$$ which is isomorphic to $F[x]$ under the map fixing elements of $F$ and sending $t-c$ to $x$.
\item  If $\delta$ is an outer derivation, then $C(R) = F.$
\end{enumerate}
\item If $D$ has prime characteristic $p$, then the following are satisfied:
\begin{enumerate}
\item If $\delta\vert_C$ is algebraic with minimum polynomial $$g(t) = t^{p^e} + \gamma_1t^{p^{e-1}} + \dots + \gamma_et \in F[t]$$ such that $g(\delta) = \delta_c$ for some $c \in D^{\times}$, then $$C(R) = F[g(t)-c]$$ which is isomorphic to $F[x]$ under the map fixing elements of $F$ and sending $g(t)-c$ to $x$.
\item If $\delta\vert_C$ is transcendental, then $C(R) = F$.
\end{enumerate}
\end{enumerate}
\end{theorem}

\section{Petit Algebras}
This section is dedicated to a construction of nonassociative algebras obtained from a skew polynomial ring $R$, which generalises the construction of the associative quotient algebras of $R$ obtained from factoring $R$ by a two-sided principal ideal.\\
\\ 
Let $R = D[t;\sigma,\delta]$ be a skew polynomial ring, and let $f \in R$ have positive degree $m$. The skew polynomials of degree less than $m$ canonically represent the elements of the right $R$-module $R/Rf$. Moreover, the set of skew polynomials of degree less than $m$ $$R_m = \lbrace g \in R : {\rm deg}(g) < m \rbrace$$ endowed with the usual term-wise addition, and the multiplication defined by
$$g \circ_f h = \begin{cases}
			   gh, &\quad\text{if ${\rm deg}(g) + {\rm deg}(h) < m$ }\\
			   gh\,{\rm mod_r}f, &\quad\text{if ${\rm deg}(g) + {\rm deg}(h) \geq m$ }
			  \end{cases}$$
where ${\rm mod_r}f$ denotes the remainder on right division by $f$, is a unital, nonassociative ring $S_f$, which we also denote by $R/Rf$. When the context is clear we will use $\circ$, or simply juxtaposition in place of $\circ_f$ for the multiplication in $S_f$. The ring $S_f$ has the structure of a unital, nonassociative algebra over its subfield $F = \lbrace a \in D : ag = ga \text{ for all } g \in S_f \rbrace = {\rm Comm}(S_f) \cap D$. We remark that $S_f$ is a free left $D$-module of finite rank with basis $t^0=1, t^1=t, t^2, \dots, t^{m-1}$, and $S_f$ is associative if and only if $f$ is right-invariant in $R$, in which case $S_f$ is equal to the associative quotient algebra $R/(f)$. The algebras $S_f$ were first introduced in 1966 by Petit, hence they are known as Petit algebras \cite{petit1966certains}.  The skew polynomials $f$ and $df$ yield isomorphic Petit algebras 
for any $d \in D^{\times}$, and if ${\rm deg}(f)=1$, then $S_f \cong D$. \\
\\
It can be easily shown that the field $F = {\rm Comm}(S_f) \cap D$ is equal to the subfield $C \cap {\rm Fix}(\sigma) \cap {\rm Const}(\delta)$ of $D$, i.e.
$$F = {\rm Comm}(S_f) \cap D = C \cap {\rm Fix}(\sigma) \cap {\rm Const}(\delta).$$
To see this, suppose that $a \in F \subseteq D$, so that $ga = ag$ for all $g \in S_f$. In particular $ab = ba$ for all $b \in D$, and $at = ta$. That is, $a \in C$, and $$at = ta = \sigma(a)t + \delta(a),$$ which yields $\sigma(a) = a$ and $\delta(a) = 0$. Therefore $a \in C \cap {\rm Fix}(\sigma) \cap {\rm Const}(\delta)$.
On the other hand, if $a \in C \cap {\rm Fix}(\sigma) \cap {\rm Const}(\delta)$, then $\Delta_{j,i}(a) = 0$, and $\Delta_{j,j}(a) = a$ for any non-negative integers $i,j$, such that $i < j$. So for any $g(t) = \sum\limits_{j=0}^{m-1} g_jt^j \in S_f$ we have  $$ga= (\sum\limits_{j=0}^{m-1} g_jt^j)a = \sum\limits_{j=0}^{m-1} g_j \sum\limits_{i=0}^j \Delta_{j,i}(a)t^i = \sum\limits_{i=0}^{m-1} g_jat^j = a\sum\limits_{i=0}^{m-1} g_jt^j = ag,$$ i.e. $a \in F$. Hence $F = C \cap {\rm Fix}(\sigma) \cap {\rm Const}(\delta)$.\\
\\
In \cite{petit1966certains}, Petit determines the right, middle, and left nuclei of $S_f$ and provides equivalent conditions for the indeterminant $t$ to be contained in the right nucleus of $S_f$:

\begin{theorem}\cite[(2),(5)]{petit1966certains}\label{Intro: Theorem 3}
Let $f \in R$ have degree $m > 1$.
\begin{enumerate}[(i)]
\item The right nucleus of $S_f$ is equal to the eigenring of $f$, that is $${\rm Nuc}_r(S_f) = \mathcal{E}(f).$$
\item If $f$ is not right-invariant, then the left and middle nuclei of $S_f$ are equal to $D$, that is $${\rm Nuc}_l(S_f) = {\rm Nuc}_m(S_f) = D.$$
\item $ft \in Rf$ if and only if $t \in {\rm Nuc}_r(S_f)$ if and only if $t \circ_f t^m = t^m \circ_f t$ if and only if the powers of $t$ are associative.
\end{enumerate}
\end{theorem}

\begin{remark}
We note here that for $f \in R$ of degree $1$, i.e. $f(t)=t-a$ for some $a \in D$, the right nucleus is ${\rm Nuc}_r(S_f) = D$, as $S_f \cong D$, however, the eigenring of $f$ is given by $$\mathcal{E}(f) = \{ b \in D: \sigma(b)a=ab-\delta(b) \},$$ which is not equal to $D$ in general. For instance, if $D$ is commutative and $\delta = 0$, then the eigenring of $f$ is $$\mathcal{E}(f) = \{ b \in D : \sigma(b) = b \} = {\rm Fix}(\sigma),$$ which is not equal to $D$ when $\sigma$ is not the identity map. Hence Theorem \ref{Intro: Theorem 3} does not apply for ${\rm deg}(f)=1$.
\end{remark}

\chapter{The Eigenring of $f \in D[t;\sigma]$}

\section{The Minimal Central Left Multiple of $f \in D[t;\sigma]$}

Throughout this chapter, unless stated otherwise, let $D$ be a central division algebra over $C$ of degree $d$, let $\sigma$ be an automorphism of $D$ of finite order $n$ modulo inner automorphisms with $\sigma^n = \iota_u$, and let $f \in R=D[t;\sigma]$ be a monic, non-constant polynomial of degree $m \geq 1$. Then $D$ has finite dimension $d^2n$ over the field $F = C \cap {\rm Fix}(\sigma)$. Recall that $R$ has center $$C(R) = F[u^{-1}t^n] \cong F[x],$$ and that all non-constant polynomials in $R$ are bounded. We define the minimal central left multiple of $f$ to be the polynomial $h(t)=\hat{h}(u^{-1}t^n)$ where $\hat{h} \in F[x]$ is a monic polynomial of minimal degree such that $f$ right divides $h$ in $R$. In the following we use the notation $h$ and $\hat{h}$ for the minimal central left multiple of $f$. \\
\\
We begin by providing some results on the minimal central left multiple of $f \in D[t;\sigma]$, and the relationship between irreducible factors of $h$ in $R$ and irreducible factors of $f$ in $R$.\\
\\
Given a bound $f^*$ of $f$, it is straightforward to determine the minimal central left multiple of $f$:

\begin{proposition}\label{C2.1: Proposition 1}
Let $f \in R$ be a non-constant polynomial with a given bound $f^*=a c(t)t^r$ where $a \in D^{\times}$, $c(t) = \hat{c}(u^{-1}t^n)$ for some monic polynomial $\hat{c} \in F[x]$, and some $r \in \{0,1,\dots , n-1 \}$.
\begin{enumerate}[(i)]
\item If $r=0$, then $$h(t) = a^{-1}f^* = c(t),$$ i.e. $\hat{h}(x) = \hat{c}(x)$.
\item If $r \neq 0$, then $$h(t) = pf^* = c(t)u^{-1}t^n,$$ where $p(t) = u^{-1}\sigma^{n-s}(a^{-1})t^{n-r} \in R$, i.e. $\hat{h}(x) = \hat{c}(x)x$.
\end{enumerate}
\begin{proof}
(i) If $r=0$, then $f^*= ac(t) $ is a central left multiple of $f$, and if $c^{\prime}$ is any other central left multiple of $f$, then $c^{\prime}$ has degree greater than or equal to that of $f^*$, or else $Rc^{\prime}$ is a two-sided ideal of $R$ contained in $Rf$ that is not contained in $Rf^*$, which contradicts the definition of the bound of $f$. Hence $f^*$ is a central left multiple of $f$ of minimal degree. Now $f^*= bh(t) = b\hat{h}(u^{-1}t^n)$ for some $b\in D^{\times}$, i.e. $$a\hat{c}(x) = b\hat{h}(x).$$ Comparing lead coefficients on both sides yields $a = b$, i.e. $h(t) = a^{-1}f^* = c(t)$ as claimed.\\
\\ (ii) If $0 < r < n-1$, then the polynomial 
$$u^{-1}\sigma^{n-r}(a^{-1})t^{n-r}f^*$$
 is clearly a left multiple of $f$, because $f^*$ is a left multiple of $f$.
   Now, since $c(t)$ commutes with all polynomials in $R$ (in particular, with powers of $t$ and elements of $D$) we have that:
\begin{align*}
u^{-1}\sigma^{n-r}(a^{-1})t^{n-s}f^* &= (u^{-1}\sigma^{n-r}(a^{-1})t^{n-r}) (ac(t)t^r)\\
&= u^{-1}\sigma^{n-r}(a^{-1})\sigma^{n-r}(a)c(t)t^n \\
&= c(t)u^{-1}t^{n}.
\end{align*}
Therefore the polynomial $pf^*$ with $p(t)=u^{-1}\sigma^{n-r}(a^{-1})t^{n-r}$ is a central left multiple of $f$, that is also monic as polynomial in $F[x]$ when we make the identification $x=u^{-1}t^n$. Hence in order to confirm that $pf^*$ is the {\it minimal} central left multiple of $f$, we need only show that there is no central left multiple of $f$ of strictly lower degree. To this end suppose for a contradiction that $e$ is a central left multiple of $f$ of degree less than that of $pf^*$. Since the degree of $e$ is a multiple of $n$, it has degree at most ${\rm deg}(pf^*) - n$, and since ${\rm deg}(pf^*) = {\rm deg}(f^*) + n - r$, we have the following inequalities $${\rm deg}(e) \leq {\rm deg}(f^*) - r < {\rm deg}(f^*)$$ as $r>0$. Therefore $Re$ is a two-sided ideal of $R$ contained in $Rf$, but $Re \nsubseteq Rf^*$, which is a contradiction, since $f^*$ is a bound of $f$. Hence $pf^*$ achieves minimal degree for a central left multiple of $f$ and the result follows immediately.
\end{proof}
\end{proposition}

As we have already seen, all non-constant polynomials in $R$ have a uniquely determined bound (unique up to left multiplication by an element of $D^{\times}$), hence the preceding result guarantees that any non-constant polynomial has a uniquely determined minimal central left multiple which is a multiple of its bound.\\
\\
We note that in the notation of Proposition \ref{C2.1: Proposition 1}, if $(f,t)_r=1$ then $r=0$, and $f^*$ is equal to $ah(t)$, i.e. $f^*$ is central up to left multiplication by an element in $D^{\times}$.

\begin{proposition}\label{C2.1: Proposition 2}
Let $f(t) = t^m - \sum\limits_{i=0}^{m-1} a_it^i \in R$, and let $f$ have minimal central left multiple $h(t) = \hat{h}(u^{-1}t^n)$ for some monic polynomial $\hat{h} \in F[x]$ with constant coefficient $h_0$.
\begin{enumerate}[(i)]
\item The following are equivalent:
\begin{enumerate}
\item $a_0 \neq 0$,
\item $(f,t)_r = 1$,
\item $h_0 \neq 0$,
\item $(h,t)_r=1$.
\end{enumerate}
\item If the equivalent conditions of (i) are satisfied, then $f(t) \neq t$, and $\hat{h}(x) \neq x$.
\end{enumerate}
\end{proposition}
\begin{proof}
(i) 
Clearly, $(a)$ is equivalent to $(b)$, and $(c)$ is equivalent to $(d)$, therefore we need only show that $(a)/(b)$ implies $(c)/(d)$ and vice versa. \\
Firstly, we show that $(c) \Rightarrow (b)$ via the contrapositive. So suppose that $(f,t)_r \neq 1$, then we must have that $(f,t)_r=t$ since $1$ and $t$ are the only monic right divisors of $t$. Then $t$ right divides $f$, i.e. $f=pt$ for some $p \in R$, and so $h=gf =(gp)t$ for some $g \in R$. We conclude that $(h,t)_r \neq 1$.\\
Finally, to show that $(b) \Rightarrow (c)$, suppose that $(f,t)_r=1$, and suppose for a contradiction that $(h,t)_r \neq 1$. Then $h=qt$ for some $q \in R$. Since $t$ is a proper two-sided divisor of $h$ in $R$, there exists a proper factorisation of $f$ in $R$ of the form $f=ab$ where $b = (f,t)_r$, and $a$ is bounded \cite[Proposition 5.1]{gomez2013computing}. However, $b = (f,t)_r = 1$ by assumption, which is a contradiction, $b=1$ is not a proper factor of $f$. Hence we conclude that $(b)$ and $(c)$ are equivalent.\\
\\
(ii) Now assume that the equivalent conditions of (i) are satisfied. Immediately we obtain $f(t) \neq t$ (else $a_0 = 0$). Moreover if $\hat{h}(x)=x$, then $h_0=0$ thus (ii) follows via the contrapositive.
\end{proof}

\begin{corollary}\label{C2.1: Corollary 1}
Let $f = t^m - \sum\limits_{i=0}^{m-1}a_it^i \in  R$ be irreducible, and let $f$ have minimal central left multiple $h(t) = \hat{h}(u^{-1}t^n)$ for some monic, irreducible $\hat{h} \in F[x]$ with constant coefficient $h_0$. Then the following are equivalent:
\begin{enumerate}
\item $a_0 \neq 0$,
\item $(f,t)_r=1$
\item $h_0 = 0$,
\item $(h,t)_r=1$,
\item $\hat{h}(x) \neq x$,
\item $f(t) \neq t$.
\end{enumerate}
\end{corollary}
\begin{proof}
By Proposition \ref{C2.1: Proposition 2}, (1) - (4) are equivalent, and the equivalent conditions (1) - (4) imply (5) and (6). To complete the proof, we show that (5) implies (6), and that (6) implies (2).\\
\\
\underline{(5) $\Rightarrow$ (6)}\\
We show that (5) $\Rightarrow$ (6) by proving the contrapositive, i.e. if $f(t) = t$, then $\hat{h}(x) = x$.\\
\\
Suppose that $f(t) = t$, and let $g(t) = u^{-1}t^{n-1}$, then $$g(t)f(t)=(u^{-1}t^{n-1})t$$ is a monic polynomial of minimal positive degree in the commutative polynomial ring $F[u^{-1}t^n]$. Therefore we must have $$h(t) = \hat{h}(u^{-1}t^n) = u^{-1}t^n$$ by definition of the minimal central left multiple of $f$, and we conclude that $\hat{h}(x) = x$.\\
\\
\underline{(6) $\Rightarrow$ (2)}\\
Finally, suppose that $f(t) \neq t$. As $f$ is assumed to be a monic, irreducible polynomial in $R$, we must have either $(f,t)_r=f$ or $(f,t)_r=1$. If $(f,t)_r=f$ then, by definition of the greatest common right divisor, $t=pf$ for some $p \in R$. Since both $t$ and $f$ are monic, irreducible polynomials we are forced to take $p=1$, and $f=t$, a contradiction. Hence we are left with $(f,t)_r=1$.
\end{proof}

\begin{lemma}\label{C2.1: Lemma 1} (for  finite fields this is \cite{giesbrecht1998factoring})
Suppose that $h \in R$ is such that $h=\hat{h}(u^{-1}t^n)$ for some monic $\hat{h} \in F[x]$
 and either $\hat{h}(x) = x$, or  $h$ has nonzero constant term. Then the quotient algebra $ R/Rh$ has center
$$ C(R/Rh)\cong  F[x]/ (\hat{h}(x) ).$$
\end{lemma}

\begin{proof}
Since $\hat{h}(u^{-1}t^n) \in C(R)$, $Rh$ is a two-sided ideal in $R$.
If $\hat{h}(x) = x$, then a direct calculation shows that $C(R/Rh) = F[x]/(x) = F$.

So suppose that $\hat{h}(x) \not= x$, and $h$ has nonzero constant term $h_0$.
Let $\phi : R \longrightarrow R/Rh$, $\phi(a) = a + Rh$. If $a \in C(R)$, then $(a+Rh)(b+Rh)=(b+Rh)(a+Rh)$ for any $b \in R$, and so $a+Rh = \phi(a) \in C(R/Rh)$. Therefore $\phi(C(R)) \subset  C(R/Rh)$.

Now let $\bar{a} \in C(R/Rh)$, then $\bar{a} = a+Rh$ for some $a \in R$ such that ${\rm deg}(a) < {\rm deg}(h) = mn$,
and
$(a+Rh)(b+Rh)=(b+Rh)(a+Rh)$ for all $b \in R$, which is equivalent to $ab+Rh = ba+Rh$ for all $b \in R$.
  This means that for all $b \in R$, we have
\begin{equation}\label{center of R/Rh}
ab = ba + r_bh
\end{equation} for some $r_b \in R$.

As (\ref{center of R/Rh}) is satisfied for all $b\in R$, let $b \in D$ and suppose $r_b\neq 0$. By comparing degrees on both sides of Equation (\ref{center of R/Rh}), we obtain
$${\rm deg}(ab) = {\rm deg}(ba + r_bh) \, \Leftrightarrow \, {\rm deg}(a) = {\rm deg}(r_bh).$$
This is a contradiction; thus we conclude $r_b = 0$ and $ab = ba$, i.e. $a$ commutes with all elements of $D$.
Now suppose that $b=t$. Assume for a contradiction that $r_t \notin D$, so ${\rm deg}(r_t)\geq 1$. Comparing degrees yields
$${\rm deg}(at)={\rm deg}(ta + r_th)\Rightarrow {\rm deg}(a)+1 = {\rm deg}(r_th) \geq {\rm deg}(h)+1,$$
which is a contradiction, thus $r_t \in D$. Comparing constant terms on both sides of (\ref{center of R/Rh}), we see that $r_th_0 = 0$, where $h_0$ is the constant term of $h$. This implies that either $r_t = 0$ or $h_0 = 0$. Since
$h_0 \neq 0$ this forces $r_t = 0$, and $at =ta$, i.e. $a$ commutes with $t$, and therefore by induction, also with $t^j$. Thus $a\in C(R).$ Hence for all $\bar{a} \in C(R/Rh)$, there exists $a \in C(R)$ such that $\bar{a} = \phi(a) \in \phi(C(R))$, and so $C(R/Rh) = \phi(C(R))$.

Due to this, any element $\bar{a} \in \phi(C(R))$ can be written in the form $\bar{a} = a+Rh$ for some $a \in C(R) \cong F[x]$, where ${\rm deg}(a) < {\rm deg}(h)$. Moreover, $\phi(C(R))$ inherits the multiplication of $R/Rh$.
Define a map from $\phi(C(R))$ to $F[x]/( \hat{h}(x) )$ which fixes elements of $F$ and maps $u^{-1}t^n+Rh$ to $x+( \hat{h}(x) )$.
This yields an $F$-algebra isomorphism and thus $C(R/Rh) \cong F[x]/( \hat{h}(x) ).$
\end{proof}


\section{The Eigenring of $f \in D[t;\sigma]$ for $\hat{h}$ Irreducible}

Recall that $D$ is a central division algebra of finite degree $d$ over $C$, and $\sigma$ is an automorphism of $D$ of finite inner order $n$, with $\sigma^n = \iota_u$ for $u \in D^{\times}$. In the following, let $f\in R=D[t;\sigma]$ be monic of degree $m$ with minimal central left multiple $h(t) = \hat{h}(u^{-1}t^n)$, where $\hat{h}$ is a monic polynomial in $F[x]$. In this section, we determine the structure of the eigenring of $f$, if $\hat{h}(x)$ is irreducible in $F[x]$. 

\begin{lemma}\cite[Theorem 13]{jacobson1943theory}\label{C2.2: Lemma 1}
Let $f$ be irreducible in $R$ such that $(f,t)_r=1$. Then $\hat{h}(x)$ is irreducible in $F[x]$.
\end{lemma}

\begin{proof}
Suppose that $\hat{h}  = \hat{a}\hat{b}$ for some $\hat{a},\hat{b}  \in F[x]$, with ${\rm deg}(\hat{a}), {\rm deg}(\hat{b})<{\rm deg}(\hat{h})$. 
  If $f$ divides $\hat{a}(u^{-1}t^n)$ on the right, then $\hat{a}(u^{-1}t^n)$ is a left multiple of $f$ of degree less than $\hat{h}(u^{-1}t^n)$, contradicting the minimality of $h$. Therefore $f$ does not divide $\hat{a}(u^{-1}t^n)$. Since $f$ is irreducible, we have that $(f, \hat{a}(u^{-1}t^n))_r = 1$, and since $R$ is a right Euclidean domain, there exist $p,q \in R$ such that
 $$pf + q\hat{a}(u^{-1}t^n) = 1.$$
Multiplying through by $\hat{b}(u^{-1}t^n)$ on the right gives
$$pf \hat{b}(u^{-1}t^n) + q\hat{a}(u^{-1}t^n)\hat{b}(u^{-1}t^n) = \hat{b}(u^{-1}t^n).$$
Since $f$ right divides $\hat{h}(u^{-1}t^n)$ there exists $g \in R$ such that $\hat{h}(u^{-1}t^n) = \hat{a}(u^{-1}t^n) \hat{b}(u^{-1}t^n) = g(t)f(t)$, therefore we get
 $$(p\hat{b}(u^{-1}t^n) + qg)f= \hat{b}(u^{-1}t^n).$$
That is, $\hat{b}(u^{-1}t^n)$ is a left multiple of $f$, and since ${\rm deg}(\hat{b}) < {\rm deg}(\hat{h})$, $\hat{b}$ is a left multiple of $f$ of degree less than $\hat{h}$, contradicting the minimality of $h$.
 Hence  $\hat{h}$ is irreducible in $F[x]$.
\end{proof}

Although the converse to Lemma \ref{C2.2: Lemma 1} is false in general, Gomez-Torrecillas et al \cite{gomez2013computing} give sufficient conditions for which it holds true. Note that ${\rm deg}(h)=mnd$  is the largest possible degree of $h$. We rephrase the original statement of \cite[Proposition 4.1]{gomez2013computing} to suit our theory since the authors did not recognise $\mathcal{E}(f)$ as the right nucleus of the nonassociative algebra $S_f$.

\begin{proposition} \label{C2.2: Proposition 1}\cite[Proposition 4.1]{gomez2013computing} 
Let $f \in R$ have degree $m > 1$ and satisfy $(f,t)_r=1$.
If ${\rm deg}(h)=mnd$ and $\hat{h}$ is irreducible in $F[x]$, then $f$ is irreducible and ${\rm Nuc}_r(S_f)\cong F[x]/(\hat{h}(x))$ . 
\end{proposition}

If the minimal central left multiple of $f$ is irreducible as a polynomial in $F[x]$, then the irreducible factors in any factorisation of $f$ in $R$ are mutually similar:

\begin{lemma}\label{C2.2: Lemma 2}
Let $f \in R$ satisfy $(f,t)_r=1$, and let  $\hat{h}(x)$ be irreducible in $F[x]$. Then all irreducible factors of $f$ are similar to all irreducible factors of $h$. In particular, all irreducible factors of $f$ are mutually similar to each other.
\begin{proof}
In the language of \cite[Theorem 1.2.19]{jacobson2009finite}, $h(t)$ is a two-sided maximal element of $R$, hence any factorisation $h(t)=h_1(t)\cdots h_k(t)$ into irreducible polynomials in $R$ is unique up to similarity of polynomials, and $h_i \sim h_j$ for all $i,j$. Let $f(t)=f_1(t)\cdots f_l(t)$ for $f_i$ irreducible polynomials in $R$. Then, by definition of the minimal central left multiple $$h_1(t)h_2(t) \cdots h_k(t) = p(t)f_1(t)f_2(t)\cdots f_l(t)$$ for some $p \in R$. Since the decomposition of $h$ is unique up to similarity, for each $i \in \{ 1,\dots , l \}$, there must exist $j \in \{ 1,\dots , k\}$ such that $f_i \sim h_j$. Finally, as $h_i \sim h_j$ for all $i,j$, we conclude that $f_i \sim f_j$ for all $i,j$.
\end{proof}
\end{lemma}

Define $E_{\hat{h}}=F[x]/ (\hat{h}(x) )$. This is a commutative, associative algebra over $F$ of dimension ${\rm deg}(\hat{h})$.  If $\hat{h}$ is irreducible in $F[x]$, then $E_{\hat{h}}$ is a field extension of $F$ of degree ${\rm deg}(\hat{h})$. We first consider the case that $f$ has degree at least $2$ and that $f$ is irreducible in $R$, which is a sufficient condition for $\hat{h}$ to be irreducible in $F[x]$ by Lemma \ref{C2.2: Lemma 1}. We show that  ${\rm Nuc}_r(S_f)$ is a central division algebra of degree $s = dn/k$ over a field extension $E_{\hat{h}}$ of $F$ determined by $h$, where $k$ is the number of irreducible factors in any factorisation of $h$ into irreducible polynomials in $R$. \\
\\ After we have explored some special cases, we loosen the restriction that $f$ is irreducible in $R$, instead assuming only that $\hat{h}$ is irreducible in $F[x]$. We show that ${\rm Nuc}_r(S_f)$ is a central simple algebra of degree $s=ldn/k$ over the same field extension $E_{\hat{h}}$ of $F$, with $k$ as above, and $l$ the number of irreducible factors in any factorisation of $f$ into irreducible polynomials in $R$.

\begin{lemma}\label{C2.2: Lemma 3} \cite[p.~16]{jacobson2009finite}
Suppose that $h \in R$ is such that $h(t)=\hat{h}(u^{-1}t^n)$ for some $\hat{h} \in F[x]$, $\hat{h}(x) \neq x$, and such that
 $\hat{h}$ is irreducible in $F[x]$. Then $h$ generates a maximal two-sided ideal $Rh$ in $R$.
 \end{lemma}
 
 Let $f(t)=t-a \in R$ for some $a \in D$. Then the $F$-algebra $S_f$ is equal to the associative ring $D$, hence ${\rm Nuc}_r(S_f) = D$. Moreover for any $b \in D$, it can be easily seen that $fb$ lies in $Rf$ if and only if $\sigma(b)a = ab$. Thus the eigenring $\mathcal{E}(f)$ is equal to $\{b \in D: \sigma(b)a=ab\}$. In particular, if $a \in C$ then $\mathcal{E}(f) = {\rm Fix}(\sigma)$ and if $\sigma={\rm id}$, then $\mathcal{E}(f)={\rm Cent}_D(a)$. Henceforth, for the rest of the chapter, unless stated otherwise, we assume that $f$ has degree $m>1$, in which case ${\rm Nuc}_r(S_f) = \mathcal{E}(f)$.

\begin{theorem}\label{C2.2: Theorem 1}
Let $f \in R=D[t;\sigma]$ be irreducible of degree $m > 1$. Then ${\rm Nuc}_r(S_f)$ is a central division algebra over $E_{\hat{h}}=F[x]/(\hat{h}(x))$ of degree $s=dn/k$, where $k$ is the number of irreducible factors of $h$ in $R$, and
 $$ R/Rh \cong M_k({\rm Nuc}_r(S_f)).$$
 This means that ${\rm deg}(\hat{h})=\frac{md}{s}$, ${\rm deg}(h)=\frac{mnd}{s}$, and
 $$[{\rm Nuc}_r(S_f) :F]=mds.$$
 Moreover, $s$ divides ${\rm gcd}(md,dn)$. If $f$ is not right-invariant, then $k>1$ and $s\not=dn$.
\end{theorem}

\begin{proof}
The minimal central left multiple $h$ of $f$ is a two-sided maximal element in $R$ in the terminology of \cite{jacobson2009finite}, and  $h=gf$ for some $g\in R$ by the definition of $h$.
Since $R$ is a principal ideal domain, the irreducible factors $h_i$ of any factorization $h=h_1h_2\cdots h_k$ of $h$ into irreducible polynomials are all similar as polynomials. This implies that all irreducible factors of $h$ have the same degree.

Moreover, $R/Rh$ is a simple Artinian ring with $R/Rh \cong M_k(D_h)$, where $D_h \cong \mathcal{I}(h_i)/Rh_i$ and $\mathcal{I}(h_i)=\lbrace g \in R: h_ig \in Rh_i \rbrace$ is the idealiser of $Rh_i$ \cite[Theorem 1.2.19]{jacobson2009finite}.

Since $f$ is an irreducible divisor of $h$ with $h=gf$ for some $g \in R$, we thus obtain that  $h=h_1h_2\cdots h_{k-1}f$ for some irreducible polynomials $h_i \in R$ of degree $m$,  $D_h \cong \mathcal{I}(f)/Rf = \mathcal{E}(f)$, and therefore
$$R/Rh \cong M_k({\rm Nuc}_r(S_f))$$ since the eigenring of $f$ is equal to the right nucleus of $S_f$ when $f$ has degree at least $2$.
In particular, here $h$ has degree $km$ and since $f$ is irreducible, ${\rm Nuc}_r(S_f)$ is a division algebra.
Now $R/Rh$ is a central simple algebra over $E_{\hat{h}}$ and so
  ${\rm Nuc}_r(S_f)$ is a central division algebra over $E_{\hat{h}}$ of dimension $s^2$.
Comparing the dimensions of $R/Rh$ and $M_k({\rm Nuc}_r(S_f))$ over $F$ it follows that $d^2n \, {\rm deg}(h)= d^2n^2 \,{\rm deg} (\hat{h})=k^2s^2
\, {\rm deg}(\hat{h})$, so that we get $d^2n^2 =k^2s^2 $, that is $dn=ks$. In particular, this implies that ${\rm deg}(h)=\frac{dnm}{s}$ and ${\rm deg}(\hat{h})=\frac{dm}{s}$. Moreover $$[{\rm Nuc}_r(S_f):F]= [{\rm Nuc}_r(S_f):E_{\hat{h}}][E_{\hat{h}}:F] =s^2 {\rm deg}(\hat{h}) = \frac{dms^2}{s} = dms.$$
Now since ${\rm Nuc}_r(S_f)$ is a subalgebra of $S_f$ we have 
\begin{equation}\label{C2.2: Theorem 1 Eq 1}
[S_f:F]=[S_f:{\rm Nuc}_r(S_f)][{\rm Nuc}_r(S_f):F]= d^2mn.
\end{equation}
Substituting $dn = ks$ and $[{\rm Nuc}_r(S_f):F]=dms$ into Equation (\ref{C2.2: Theorem 1 Eq 1}) we obtain the equality $[S_f:{\rm Nuc}_r(S_f)]=k$.
If $f$ is not right-invariant, then $k>1$ and so we derive $s\not=dn$ looking at the degree of $h$. More precisely, for $f$ of degree $m>1$, $f$ being not right-invariant is equivalent to $S_f$ being not associative which in turn is equivalent to
$k>1$.
\end{proof}

\begin{remark}
If we consider the degree one polynomial $f=t-a$ for some $a \in D^{\times}$, then an analogous proof to the one of Theorem \ref{C2.2: Theorem 1} yields ${\rm deg}(h) = \frac{nd}{s}$, ${\rm deg}(\hat{h})=\frac{d}{s}$ and that $\mathcal{E}(f)$ is a central division algebra over $F$ of degree $s$ contained in $D$. In particular $s$ divides $d$. Recall that for $f$ of degree one, ${\rm Nuc}_r(S_f)$ is not necessarily equal to the eigenring of $f$.
\end{remark}

Now that we have dealt with the right nucleus of $f$ for $f$ an irreducible polynomial in $R$, we turn our attention to the more general setting with $f$ not necessarily irreducible in $R$, but with $\hat{h}$ an irreducible polynomial in $F[x]$. Note that this setting includes the previous case that $f \in R$ is irreducible, which corresponds to the case $l=1$ in the following result (Theorem \ref{C2.2: Theorem 2}).

\begin{theorem}\label{C2.2: Theorem 2}
Let $f \in R$ satisfy $(f,t)_r=1$, and suppose that $\hat{h}(x)$ is irreducible in $F[x]$. Then $f$ is the product of $l \geq 1$ irreducible polynomials in $R$ all of which are mutually similar to each other, and $$R/Rh \cong M_k(\mathcal{E}(g))$$ where $g \in R$ is any irreducible divisor of $h$ in $R$. If ${\rm deg}(g)=r \geq 1$, then $m=rl$, $\mathcal{E}(g)$ is a central division algebra over $E_{\hat{h}}=F[x]/(\hat{h}(x))$ of degree $s^{\prime}=dn/k$, where $k$ is the number of irreducible factors of $h$, and
 $$ {\rm Nuc}_r(S_f) \cong M_l(\mathcal{E}(g)).$$
 In particular, ${\rm Nuc}_r(S_f)$ is a central simple algebra over $E_{\hat{h}}$ of degree $s = ls^{\prime}$, ${\rm deg}(\hat{h})=\frac{rd}{s^{\prime}}=\frac{md}{s}$, ${\rm deg}(h) = \frac{rnd}{s^{\prime}} = \frac{mnd}{s}$, and
 $$[{\rm Nuc}_r(S_f) :F] = l^2rds^{\prime} = mds.$$ In particular $s^{\prime}$ divides ${\rm gcd}(rd,dn)$, and $s$ divides ${\rm gcd}(md,dn)$.\\
\begin{proof}
Since $\hat{h}$ is irreducible in $F[x]$, $h$ is a two-sided maximal element of $R$ in the language of \cite{jacobson2009finite}, and $h=pf$ for some $p \in R$ by the definition of the minimal central left multiple. Since $R$ is a principal ideal domain, the irreducible factors $h_i$ of any factorisation $h=h_1h_2\cdots h_k$ of $h$ in $R$ into irreducible polynomials are all similar as polynomials. In particular, $R/Rh_i \cong R/Rh_j$ for all $i,j$, and all irreducible factors of $h$ have the same degree.\\
\\
Moreover, $R/Rh$ is a simple Artinian ring with $R/Rh \cong M_k(\mathcal{E}(g))$, where $g \in R$ is an irreducible polynomial similar to $h_i$ for all $i$, and $\mathcal{E}(g)$ denotes the eigenring of $g$ in $R$ \cite[Theorem 1.2.19]{jacobson2009finite}.\\
\\
Let $A=R/Rh$, and suppose that $f=f_1f_2\cdots f_l$ where $f_i \in R$ are irreducible polynomials. Then $f_i \sim f_j$ for any $i,j$, and each of the polynomials $g_i$ has minimal central left multiple $h$ \cite[Proposition 5.2]{gomez2013computing}. Any left $A$-module is isomorphic to a direct sum of simple left $A$-modules, and any two simple left $A$-modules are isomorphic  \cite[Theorem 25]{jacobson1943theory}. It follows that $$R/Rf \cong R/Rf_1 \oplus R/Rf_2 \oplus \cdots \oplus R/Rf_l$$ as left $A$-modules (e.g. see \cite[Corollary 4.7]{gomez2013computing}).  Since $R/Rf_i \cong R/Rg$ for $g$ any irreducible factor of $h$, we have $$R/Rf \cong (R/Rg)^{\oplus l},$$ as left $A$-modules.
By Lemmas \ref{Intro: Lemma 1} and \ref{Intro: Lemma 2}, we obtain $${\rm End}_A(R/Rf) \cong {\rm End}_A((R/Rg)^{\oplus l}) \cong M_l({\rm End}_A(R/Rg)),$$ as rings. Since $h$ is a bound of both $f$ and $g$, the two-sided ideal $Rh$ is equal to ${\rm Ann}_R(R/Rf)$ and to ${\rm Ann}_R(R/Rg)$ \cite[pg.~38]{jacobson1943theory}. Hence ${\rm End}_R(R/Rf)={\rm End}_A(R/Rf)$ and ${\rm End}_R(R/Rg) = {\rm End}_A(R/Rg)$, by Lemma \ref{Intro: Lemma 4}, and $${\rm End}_R(R/Rf) \cong M_l({\rm End}_R(R/Rg)).$$ Finally, for any $p \in R$, the eigenring $\mathcal{E}(p)$ is isomorphic to ${\rm End}_R(R/Rp)$, therefore $$\mathcal{E}(f) \cong  M_l(\mathcal{E}(g)).$$
\\
Suppose that ${\rm deg}(g)=r \geq 1$ for $g$ any irreducible factor of $h$ in $R$. Then $f=f_1f_2\cdots f_l$ for $f_i \in R$ irreducible, implies that ${\rm deg}(f_i)=r$, and $$m ={\rm deg}(f) = \sum\limits_{i=1}^l {\rm deg}(f_i) = rl.$$ Now since $g$ is irreducible of degree $r$ with minimal central left multiple $h(t) = \hat{h}(u^{-1}t^n)$, $\mathcal{E}(g)$ is a central division algebra over $E_{\hat{h}}$ of degree $s^{\prime}=dn/k$, where $k$ is the number of irreducible divisors of $h$ in $R$, ${\rm deg}(\hat{h})=\frac{rd}{s^{\prime}}=\frac{md}{ls^{\prime}}$ and ${\rm deg}(h)=\frac{rdn}{s^{\prime}}=\frac{mdn}{ls^{\prime}}$ by Theorem \ref{C2.2: Theorem 1}. Finally, since $$\mathcal{E}(f) \cong M_l(\mathcal{E}(g)),$$ $\mathcal{E}(f)$ is a central simple algebra over $E_{\hat{h}}$ of degree $s = ls^{\prime}$, and $$[\mathcal{E}(f) : F] = s^2{\rm deg}(\hat{h}) = mds.$$
\end{proof}
\end{theorem}

An immediate Corollary to Theorem \ref{C2.2: Theorem 2} is:

\begin{corollary}\label{C2.2: Corollary 1}
Under the assumptions of Theorem \ref{C2.2: Theorem 2} 
$$[S_f:F] = \frac{k}{l}[{\rm Nuc}_r(S_f):F].$$
\end{corollary}

In Theorem \ref{C2.2: Theorem 2}, in the particular case that ${\rm deg}(g) = 1$, $f$ is the product of $m$ linear polynomials in $R$ which are all mutually similar to each other, ${\rm Nuc}_r(S_f)$ is a central simple algebra over $E_{\hat{h}}$ of degree $ms$, and $s$ divides $d$. We obtain the following result as a Corollary to Theorem \ref{C2.2: Theorem 2}:

\begin{corollary}\label{C2.2: Corollary 2}
Suppose that $(f,t)_r=1$, that ${\rm gcd}(m,n)=1$ and that $\hat{h}$ is irreducible in $F[x]$. Then $s$ divides $d$, and $f$ is not right-invariant unless $n=1$ and $s=d$. 
Additionally, if we suppose that $d$ is prime, then one of the following holds: 
\begin{enumerate}
\item $f$ is irreducible and ${\rm Nuc}_r(S_f) \cong E_{\hat{h}}$ is a field extension of $F$ of degree $md$.
\item $f$ is irreducible and ${\rm Nuc}_r(S_f)$ is a central division algebra over $E_{\hat{h}}$ of degree $d$, ${\rm deg}(\hat{h}) = m$, and $[{\rm Nuc}_r(S_f):F]=d^2m$.
\item $f$ is the product of $d$ irreducible polynomials in $R$ of degree $\frac{m}{d}$, all of which are mutually similar to each other. Moreover $${\rm Nuc}_r(S_f) \cong M_d(E_{\hat{h}})$$ is a central simple algebra over $E_{\hat{h}}$ of degree $d$, ${\rm deg}(\hat{h}) = m$, and $[{\rm Nuc}_r(S_f):F]=d^2m$. 
\end{enumerate}
\end{corollary}
\begin{proof}
Since ${\rm gcd}(dm,dn) = {\rm gcd}(m,n)d = d$, $s$ divides $d$ by Theorem \ref{C2.2: Theorem 2}. Moreover, $f$ is right-invariant if and only if $[S_f:F]=[{\rm Nuc}_r(S_f):F]$, i.e. if and only $d^2mn = dms$. This yields $s=d$ and $n=1$ since $s \leq d$.\\
Now suppose additionally that $d$ is prime. Then $s$ dividing $d$ means that $s=1$ or $s=d$. If $s=1$, then ${\rm deg}(h)=mnd$ and we obtain (1) by Proposition \ref{C2.2: Proposition 1}.\\
On the other hand if $s=d$, then there are two possible scenarios; either $l=1$ in which case $f$ satisfies (2), or $l=d$ in which case we obtain case (3), both following Theorem \ref{C2.2: Theorem 2}.
\end{proof}

Note that in Corollary \ref{C2.2: Corollary 2}, if $d$ does not divide $m$ (e.g. if $d$ is greater than $m$), then (3) cannot occur. Also if $n=1$ and $f$ is not right-invariant, then $f$ must satisfy (1).

Now we consider some certain special cases of Theorems \ref{C2.2: Theorem 1} and \ref{C2.2: Theorem 2}.

\subsection*{What if $\sigma$ is an inner automorphism?}
Suppose now that $\sigma = \iota_u$ for some $u \in D^{\times}$ (i.e. $\sigma$ has finite inner order $n=1$). Then we have that $C \subseteq {\rm Fix}(\sigma)$, i.e. $F = C \cap {\rm Fix}(\sigma) = C$, and
$$C(R)=C[u^{-1}t]\cong C[x]$$
under the map which fixes elements of $C$ and sends $u^{-1}t$ to $x$.
Then $f$ has minimal central left multiple $h(t)=\hat{h}(t^n)$ for some $\hat{h} \in C[x]$. 
We obtain the following results as Corollaries to Theorem \ref{C2.2: Theorem 1}:

\begin{corollary}\label{C2.2': Corollary 1}
 Let $f$ be irreducible. Then:\\
(i) ${\rm Nuc}_r(S_f)$ is a central division algebra over $E_{\hat{h}}=C[x]/(\hat{h}(x))$ of degree $s=d/k$, where $k$ is the number of irreducible factors of $h$ in $R$, and
 $$ R/Rh \cong M_k({\rm Nuc}_r(S_f)).$$
 This means that ${\rm deg}(\hat{h})=\frac{md}{s}$, ${\rm deg}(h)=\frac{md}{s}$, $s$ divides $d$ and $$[{\rm Nuc}_r(S_f) :F]=mds.$$
 (ii) If $d$ is prime and $f$ not right-invariant, then
$${\rm Nuc}_r(S_f) \cong E_{\hat{h}}$$
is a field extension and $[{\rm Nuc}_r(S_f):C]=md$, ${\rm deg}(\hat{h})=md$, and ${\rm deg}(h)=md$.
\end{corollary}

\begin{proof}
(i) 
We note that $$R/Rh \cong M_k({\rm Nuc}_r(S_f))$$ and that ${\rm Nuc}_r(S_f)$ is a central division algebra over $E_{\hat{h}}$ of dimension $s^2$ by an identical proof of Theorem \ref{C2.2: Theorem 1}.
Now comparing the dimensions of $R/Rh$ and $M_k({\rm Nuc}_r(S_f))$ over $F$ it follows that $d^2 \, {\rm deg}(h)= d^2 \,{\rm deg} (\hat{h})=k^2s^2 \, {\rm deg}(\hat{h})$, so that we get $d^2 =k^2s^2 $, that is $d=ks$. In particular, this implies that ${\rm deg}(h) = {\rm deg}(\hat{h}) = \frac{md}{s}$.
\\ (ii) If $d$ is prime and $f$ not right-invariant, then in the above proof $d=ks$ forces $s=1$, so that here
$ R/Rh \cong M_k(E_{\hat{h}})$, ${\rm Nuc}_r(S_f) \cong E_{\hat{h}}$ and ${\rm deg}(h) = {\rm deg}(\hat{h}) = md.$
\end{proof}

\begin{corollary}\label{C2.2': Corollary 2}
Let $d=pq$ for $p$ and $q$ prime, and $f$ be irreducible and not right-invariant.
Then ${\rm Nuc}_r(S_f)\cong E_{\hat{h}}$ is a field extension of $C$ of degree $md$, or  ${\rm Nuc}_r(S_f)$ is a central division algebra over $E_{\hat{h}}$ of degree $q$ (resp., $p$), and
 $[{\rm Nuc}_r(S_f):C] = mdq^2$ (resp., $= mdp^2$).
\end{corollary}

\begin{proof}
(i) Since $f$ is not right-invariant, we note that $k>1$. If $d=pq$ then the equation $d=ks$ in the proof of Corollary \ref{C2.2': Corollary 1} (i) forces either that $s=1$ and $k=d$, hence that ${\rm Nuc}_r(S_f)\cong E_{\hat{h}}$, or that $s\not=1$ and then w.l.o.g. that  $k=p$ and $s=q$, so that here
${\rm Nuc}_r(S_f)$ is a central division algebra over $E_{\hat{h}}$ of degree $q$, ${\rm deg}(h) = {\rm deg} (\hat{h}) = mp$,  and $[{\rm Nuc}_r(S_f):C]= mds = mpq^2$.
\end{proof}

This observation generalizes as follows by induction:

\begin{corollary} \label{C2.2': Corollary 3}
Let $d=p_1\cdots p_l$ be the prime decomposition of $d$, and $f$ be irreducible and not right-invariant.
Then ${\rm Nuc}_r(S_f)\cong E_{\hat{h}}$ is a field extension of $C$ of degree $md$, or
${\rm Nuc}_r(S_f)$ is a central division algebra over $E_{\hat{h}}$ of degree $q = q_1\cdots q_r$, with $q_i\in \{p_1,\dots, p_l\}$, and
$$[{\rm Nuc}_r(S_f):C] = m p q^2,$$ where $p = \frac{d}{q_1 \cdots q_r} = \frac{d}{q}$.
\end{corollary}

Again, we next turn our attention to the more general setting that $f \in R$ is reducible polynomial of degree $m > 1$ such that $(f,t)_r = 1$ with minimal central left multiple $h(t)=\hat{h}(u^{-1}t^n)$ for some monic, irreducible polynomial $\hat{h}(x) \in C[x]$. The following results are Corollaries to Theorem \ref{C2.2: Theorem 2}.

\begin{corollary}\label{C2.2': Corollary 4}
The polynomial $f$ is the product of $l \geq 1$ irreducible polynomials in $R$ all of which are mutually similar to each other, and $$R/Rh \cong M_k(\mathcal{E}(g))$$ where $g \in R$ is any irreducible divisor of $h$ in $R$. \\
\\ (i) Suppose that ${\rm deg}(g)=r \geq 1$. Then $m=rl$, $\mathcal{E}(g)$ is a central division algebra over $E_{\hat{h}}=C[x]/(\hat{h}(x))$ of degree $s^{\prime}=d/k$, where $k$ is the number of irreducible factors of $h$, and
 $${\rm Nuc}_r(S_f) \cong M_l(\mathcal{E}(g)).$$
 In particular, ${\rm Nuc}_r(S_f)$ is a central simple algebra over $E_{\hat{h}}$ of degree $s=ls^{\prime}$, ${\rm deg}(\hat{h})=\frac{rd}{s^{\prime}}=\frac{md}{s}$, ${\rm deg}(h) = \frac{rd}{s^{\prime}} = \frac{md}{s}$, and
 $$[{\rm Nuc}_r(S_f):C] = l^2rds^{\prime} = mds.$$ In particular $l$, $s^{\prime}$ and $s$ all divide $d$, and $l$ divides $k$.\\
 \\ (ii) 
 If ${\rm gcd}(d,m)=1$, then $f$ is irreducible, and ${\rm Nuc}_r(S_f)$ is a central division algebra over $E_{\hat{h}}$ of degree $s=d/k$, and ${\rm deg}(\hat{h})={\rm deg}(h) = \frac{md}{s}$. \\
 \\ (iii) If $d$ is prime, and $f$ is not right-invariant, then $f$ is irreducible and $${\rm Nuc}_r(S_f) \cong E_{\hat{h}}$$ is a field extension of $F$ of degree $md$.
\end{corollary}

In Corollary \ref{C2.2': Corollary 4} (i), if ${\rm deg}(g) = 1$, then $f$ is the product of $m$ linear polynomials in $R$ which are all mutually similar to each other, and ${\rm Nuc}_r(S_f)$ is a central simple algebra over $E_{\hat{h}}$ of degree $md/k$.

\subsection*{What if $D$ is a field?}
Now suppose that $D=K$ for $K$ a cyclic Galois field extension of $F = {\rm Fix}(\sigma)$ of degree $n$ with Galois group ${\rm Gal}(K/F) = \langle \sigma \rangle$. Then $\sigma$ has order $n$, 
and $R=K[t;\sigma]$ has center $$C(R)=F[t^n]$$ which is isomorphic to $F[x]$ under the map which fixes elements of $F$ and sends $t^n$ to $x$. Also $f$ has minimal central left multiple $h(t)=\hat{h}(t^n)$ for $\hat{h} \in F[x]$. We obtain the following results as corollaries to Theorem \ref{C2.2: Theorem 1}: 

\begin{theorem}\label{C2.2'': Theorem 1}
Let $f$ be irreducible, then:
 \\ (i) ${\rm Nuc}_r(S_f)$ is a central division algebra over $E_{\hat{h}}=F[x]/(\hat{h}(x))$ of degree $s=n/k$, where $k$ is the number of irreducible factors of $h$, and
 $$ R/Rh \cong M_k({\rm Nuc}_r(S_f)).$$
 In particular, this means that ${\rm deg}(\hat{h})=\frac{m}{s}$, ${\rm deg}(h)=\frac{nm}{s}$, and
 $$[{\rm Nuc}_r(S_f):F]=ms.$$
 Moreover, $s$ divides $m$ and $n$.
 \\
 (ii) If ${\rm gcd}(m,n)=1$, or $n$ is prime and $f$ not right-invariant, then
$${\rm Nuc}_r(S_f) \cong E_{\hat{h}}.$$
In particular, then $[{\rm Nuc}_r(S_f):F]=m$, ${\rm deg}(\hat{h})=m$, and ${\rm deg}(h)=mn$.
\end{theorem}

\begin{proof}
(i) 
We note that $$R/Rh \cong M_k({\rm Nuc}_r(S_f))$$ and that
${\rm Nuc}_r(S_f)$ is a central division algebra over $E_{\hat{h}}$ of dimension $s^2$ by an identical proof of Theorem \ref{C2.2: Theorem 1}.
Now comparing the dimensions of $R/Rh$ and $M_k({\rm Nuc}_r(S_f))$ over $F$ it follows that $n \, {\rm deg}(h)= n^2 \,{\rm deg} (\hat{h})=k^2s^2
\,{\rm deg}(\hat{h})$, so that we get $n^2 =k^2s^2 $, that is $n=ks$. In particular, this implies that ${\rm deg}(h)=nm/s$ and ${\rm deg}(\hat{h})=\frac{m}{s}$.
\\ (ii) If $n$ is prime then in the above proof $n=ks$ forces $s=1$, so that here
$ R/Rh \cong M_k(E_{\hat{h}})$, ${\rm Nuc}_r(S_f)\cong E_{\hat{h}}$ and ${\rm deg}(h)=mn.$
\end{proof}

\begin{remark}
As a byproduct of Theorem \ref{C2.2'': Theorem 1}, we obtain \cite[Lemma 4]{lavrauw2013semifields}. This is due to the fact that $E_{\hat{h}}$ in the proof of (i) is a finite extension of the field $F$, and so for $F= \mathbb{F}_q$ a finite base field, the only central division algebra over $E_{\hat{h}}$ is itself, i.e. $s=1$. In this case $[{\rm Nuc}_r(S_f) : \mathbb{F}_q] = m$, hence $\vert {\rm Nuc}_r(S_f) \vert = q^m$.
\end{remark}

\begin{corollary}\label{C2.2'': Corollary 1}
Let $n=pq$ for $p$ and $q$ prime.  Let $f$ be irreducible and not right-invariant.
\\ (i) ${\rm Nuc}_r(S_f)\cong E_{\hat{h}}$ is a field extension of $F$ of degree $m$, or  ${\rm Nuc}_r(S_f)$ is a central division algebra over $E_{\hat{h}}$ of prime degree $q$ (resp., $p$),
 $[{\rm Nuc}_r(S_f):F] =qm$ (resp., $=pm$), and $q$ (resp., $p$) divides $m$.
 \\ (ii) If ${\rm gcd}(m,n)=1$, then ${\rm Nuc}_r(S_f)\cong E_{\hat{h}}$ is a field extension of $F$ of degree $m$.
\end{corollary}

This observation generalizes as follows by induction:

\begin{corollary} \label{C2.2'': Corollary 2}
 Let $f$ be irreducible and not right-invariant.
\\ (i) ${\rm Nuc}_r(S_f)\cong E_{\hat{h}}$ is a field extension of $F$ of degree $m$, or
${\rm Nuc}_r(S_f)$ is a central division algebra over $E_{\hat{h}}$ of degree $q_1\cdots q_r$, with $q_i\in \{p_1,\dots, p_l\}$,
$[{\rm Nuc}_r(S_f):F] =q_1\cdots q_r m$, and $q_1\cdots q_r$  divides $m$.
 \\ (ii)  If ${\rm gcd}(m,n)=1$ (i.e., $m$ is not divisible by any set of prime factors of $n$), then ${\rm Nuc}_r(S_f)\cong E_{\hat{h}}$ is a field extension of $F$ of degree $m$.
\end{corollary}

\begin{corollary}\label{C2.2'': Corollary 3}
Let $f$ be irreducible. Suppose that $m$ is prime. Then $f$ is not right-invariant and one of the following holds:
 \\ (i) ${\rm Nuc}_r(S_f)\cong E_{\hat{h}}$,
 $[{\rm Nuc}_r(S_f) :F]=m$, and ${\rm deg}(h)=mn$.
\\  (ii) ${\rm Nuc}_r(S_f)$ is a central division algebra over $F=E_{\hat{h}}$ of prime degree $m$, $[{\rm Nuc}_r(S_f) :F]=m^2$, and $m$ divides $n$. This case occurs when $\hat{h}(x)=x-a\in F[x]$, i.e. when $h(t)=t^n-a$.
\end{corollary}

We therefore found examples of polynomials $f\in R$ whose eigenspace is a central simple algebra over $F$. Thus for any splitting field $L$ of ${\rm Nuc}_r(S_f)$ of degree $m$, the polynomial $f$ in case (ii) of Theorem \ref{C2.2'': Theorem 1} will be reducible in $L[t;\sigma]$.

\begin{corollary}\label{C2.2'': Corollary 4}
Suppose that $n$ is prime or  that ${\rm gcd}(m,n)=1$, and let $f \in F[t]\subset K[t;\sigma]$ be irreducible and not right-invariant.
Then ${\rm Nuc}_r(S_f)\cong F[t]/(f(t)).$
\end{corollary}

\begin{proof}
If $f$ is irreducible in $K[t;\sigma]$, then $F[t]/(f(t))$ is a subfield of the right nucleus of degree $m$ \cite[Proposition 2]{brown2018nonassociative}, hence must be all of the right nucleus, since that has dimension $m$ due to our assumptions (Theorem \ref{C2.2'': Theorem 1} (ii), Corollary \ref{C2.2'': Corollary 2} (ii)).
\end{proof}

\begin{theorem}\label{C2.2'': Theorem 2}
Suppose that $\hat{h}(x)$ is irreducible in $F[x]$. Then:
\begin{enumerate}[(i)]
\item  $f$ is the product of $l \geq 1$ irreducible polynomials in $R$ all of which are mutually similar to each other, and $$R/Rh \cong M_k(\mathcal{E}(g))$$ where $g \in R$ is any irreducible divisor of $h$ in $R$. If ${\rm deg}(g)=r \geq 1$, then $m=rl$, $\mathcal{E}(g)$ is a central division algebra over $E_{\hat{h}}=F[x]/(\hat{h}(x))$ of degree $s^{\prime}=n/k$, where $k$ is the number of irreducible factors of $h$, and
 $$ {\rm Nuc}_r(S_f) \cong M_l(\mathcal{E}(g)).$$
 In particular, ${\rm Nuc}_r(S_f)$ is a central simple algebra over $E_{\hat{h}}$ of degree $s = ls^{\prime}$, ${\rm deg}(\hat{h})=\frac{r}{s^{\prime}}=\frac{m}{s}$, ${\rm deg}(h) = \frac{rn}{s^{\prime}} = \frac{mn}{s}$, and
 $$[{\rm Nuc}_r(S_f) :F] = l^2rs^{\prime} = ms.$$ In particular $s^{\prime}$ divides ${\rm gcd}(r,n)$, $s$ divides ${\rm gcd}(m,n)$ and $l$ divides ${\rm gcd}(m,n)$.
\item If $m$ is prime, then one of the following holds:
\begin{enumerate}
	\item ${\rm Nuc}_r(S_f) \cong E_{\hat{h}}$ is a field extension of $F$ of prime degree $m$,
	\item ${\rm Nuc}_r(S_f)$ is a central division algebra over $F$ of prime degree $m$,
	\item ${\rm Nuc}_r(S_f) \cong M_m(F)$ is a central simple algebra over $F$ of prime degree $m$.
\end{enumerate}
\item If ${\rm gcd}(m,n)=1$, or $n$ is prime and $f$ not right-invariant, then $f$ is irreducible and $${\rm Nuc}_r(S_f) \cong E_{\hat{h}}$$ is a field extension of $F$ of degree ${\rm deg}(\hat{h}) = m$, and ${\rm deg}(h)=mn$.
\end{enumerate}

\begin{proof}
This result corresponds to the special case that $d=1$ in Theorem \ref{C2.2: Theorem 2}, and the corollaries that follow it.
\end{proof}
\end{theorem}

\begin{remark}
 Let $K/F$ be a cyclic Galois extension of degree $n$ with Galois group ${\rm Gal}(K/F)=\langle \sigma \rangle$ and let $f \in R$ be monic. If $(f,t)_r=1$, then the minimal central left multiple of $f$ is equal to the minimal polynomial of the matrix $A_f=C_f C_f^\sigma \cdots C_f^{\sigma^{n-1}}$ over $F$,
where $C_f$ is the companion matrix of $f$ \cite[Section 3.4]{sheekey2018new}.
\end{remark}


\chapter{The Eigenring of $f \in D[t;\delta]$}
Let $D$ be a central division algebra of finite degree $d$ over its center $C$. We now investigate the dimension of the right nucleus of $S_f$ for $f \in R = D[t;\delta]$ monic of degree $m$. 
Similarly to Chapter 2, if $f(t)=t-a$ for some $a \in D$, then it can be easily seen that ${\rm Nuc}_r(S_f) = D$ as the $F$-algebra $S_f$ is equal to the associative ring $D$. Moreover, for any $b \in D$, $fb \in Rf$ if and only if $\delta(b) = ab-ba$. Therefore the eigenring $\mathcal{E}(f)$ is equal to $\{ b \in D: \delta(b) = [a,b] \}$. Hence for the rest of the chapter, unless explicitly stated otherwise, we assume that $m > 1$. 
Since the center of $R$ depends on the characteristic of $D$, we split our investigation into the cases ${\rm Char}(C) = 0$, and ${\rm Char}(C)=p$ prime. 
\section{Zero Characteristic}
In the following, let ${\rm Char}(D)=0$. Let $\delta$ be a derivation of $D$, and recall that $F = C \cap {\rm Const}(\delta)$. Recall that $\delta$ is said to be an inner derivation of $D$ if there exists $c \in D$ such that $\delta(a)=[c,a]=ca-ac$ for all $a \in D$, otherwise it is called an outer derivation of $D$. If $\delta$ is an outer derivation, then $R = D[t;\delta]$ is a simple ring and ${\rm C}(R) = F $ (Amitsur, cf. \cite[Theorem 1.1.32]{jacobson2009finite}). In this case there are no non-constant bounded polynomials in $R$. 
Henceforth we assume that $\delta = \delta_c$ is the inner derivation of $D$ defined by
 $\delta(a)= ca - ac$ for $c \in D^{\times}$. Note that here ${\rm Const}(\delta) = {\rm Cent}_D(c)$  is the centraliser of $c$ in $D$, since $\delta_c(a) = 0$ if and only if $ca = ac$. Since $C \subseteq {\rm Cent}_D(c)$,  we have $F = C \cap {\rm Const}(\delta) = C$.
Every polynomial in $R=D[t;\delta]$ is bounded (Corollary \ref{Intro: Corollary 1}). Again, we define the {\it minimal central left multiple} of a polynomial $f \in R$ to be the monic polynomial $h(t) = \hat{h}(t-c)$ for some $\hat{h} \in F[x]$ of minimal degree, such that $h = gf$ for some $g \in R$.\\
\\
Every polynomial $f$ in $R$ has a unique minimal central left multiple $h(t)$, and for any bound $f^*$ of $f$, $h(t) = af^*$ for some nonzero scalar $a$ in $D^{\times}$. We note that $f^* \in C(R)$ for all $f \in R$ by \cite[Theorem 1.1.32]{jacobson2009finite}, and so we need not assume that $(f,t)_r=1$ as in the case $R=D[t;\sigma]$.

\begin{lemma}\label{C3.1: Lemma 1} 
Let $f \in R$ have degree $m \geq 1$. Then the following are satisfied.
\begin{enumerate}[(i)]
\item Every $f \in R =D[t;\delta] $ has a unique minimal central left multiple, which is a bound of $f$.
\item \cite[Theorem 13]{jacobson1943theory} If $f$ is irreducible in $R$, then $\hat{h}$ is irreducible in $F[x]$.
\item \cite{jacobson2009finite} The quotient algebra $R/Rh$ has center $$C(R/Rh) \cong F[x]/(\hat{h}(x)).$$
\item Suppose that $\hat{h}$ is irreducible in $F[x]$.
\begin{enumerate}
\item \cite[Theorem 13]{jacobson1943theory} $h$ generates a maximal two sided ideal in $R$.
\item All irreducible factors of $f$ are similar to all irreducible factors of $h$ in $R$. In particular, all irreducible factors of $f$ are mutually similar to each other.
\end{enumerate}
\end{enumerate}
\end{lemma}
\begin{proof}
(i) follows directly from Theorem \ref{Intro: Theorem 3}, and (ii), (iii), (iv)(a), and (iv)(b) follow from identical arguments to those of Lemmas \ref{C2.2: Lemma 1}, \ref{C2.1: Lemma 1}, \ref{C2.2: Lemma 3}, and \ref{C2.2: Lemma 2}, respectively.
\end{proof}

Define $E_{\hat{h}}=F[x]/ (\hat{h}(x) )$. This is a commutative algebra over $F$ of dimension ${\rm deg}(\hat{h})$.  If $\hat{h}$ is irreducible in $F[x]$, then $E_{\hat{h}}$ is a field extension of $F$ of degree equal to the degree of $\hat{h}$ in $F[x]$. Gomez-Torrecillas et al. show that the converse of Lemma \ref{C3.1: Lemma 1} (ii) holds when $h$ achieves the maximum possible degree $mnd$:

\begin{proposition}\cite[Proposition 4.1]{gomez2013computing}\label{C3.1: Proposition 1}
Let $f$ have degree $m \geq 1$.
If ${\rm deg}(h)=md$ and $\hat{h}$ is irreducible in $F[x]$, then $f$ is irreducible and ${\rm Nuc}_r(S_f)\cong E_{\hat{h}}$ .
\end{proposition} 

Again, we first consider the case that $f$ is irreducible in $R$, which is a sufficient condition for $\hat{h}$ to be irreducible in $F[x]$ by Lemma \ref{C3.1: Lemma 1} (ii). We show that ${\rm Nuc}_r(S_f)$ is a central division algebra of degree $s = d/k$ over a field extension $E_{\hat{h}}$ of $F$ determined by $h$, where $k$ is the number of irreducible factors in any complete factorisation of $h$ in $R$. \\
\\ After we have explored some special cases, we loosen the restriction that $f$ is irreducible in $R$, instead assuming only that $\hat{h}$ is irreducible in $F[x]$. We note that this setting includes the previous one that $f$ is irreducible as a special case. We obtain that ${\rm Nuc}_r(S_f)$ is a central simple algebra of degree $s=ldn/k$ over the same field extension $E_{\hat{h}}$ of $F$, with $k$ as above, and $l$ the number of irreducible factors in any complete factorisation of $f$ in $R$.  


\begin{theorem}\label{C3.1: Theorem 1}
Let $f$ be irreducible.
 \\
(i) ${\rm Nuc}_r(S_f)$ is a central division algebra over $E_{\hat{h}}=F[x]/(\hat{h}(x))$ of degree $s=d/k$, where $k$ is the number of irreducible factors of $h$ in $R$, and
 $$ R/Rh \cong M_k({\rm Nuc}_r(S_f)).$$
 In particular, this means that ${\rm deg}(\hat{h})=\frac{md}{s}$, ${\rm deg}(h)=\frac{md}{s}$, and
 $$[{\rm Nuc}_r(S_f) :F]=mds.$$
 Moreover, $s$ divides $d$.
 \\
(ii) If $d$ is prime and $f$ not right-invariant, then
$${\rm Nuc}_r(S_f) \cong E_{\hat{h}}.$$
In particular, then $[{\rm Nuc}_r(S_f) :F]=md$, ${\rm deg}(\hat{h})=md$, and ${\rm deg}(h)=md$.
\end{theorem}

\begin{proof}
(i) The first part of the proof of (i) may be taken verbatim to be the first part of the proof of Theorem \ref{C2.2: Theorem 1}. Comparing the dimensions of $R/Rh$ and $M_k({\rm Nuc}_r(S_f))$ over $F$ it follows that $d^2 \, {\rm deg}(h)= d^2 \,{\rm deg} (\hat{h})=k^2s^2
\, {\rm deg}(\hat{h})$, so that we get $d^2 =k^2s^2 $, that is $d=ks$. In particular, this implies that ${\rm deg}(h)=\frac{dm}{s}$ and ${\rm deg}(\hat{h})=\frac{dm}{s}$. Moreover $$[{\rm Nuc}_r(S_f):F]= [{\rm Nuc}_r(S_f):E_{\hat{h}}][E_{\hat{h}}:F] =s^2 {\rm deg}(\hat{h}) = \frac{dms^2}{s} = dms.$$
Now since ${\rm Nuc}_r(S_f)$ is a subalgebra of $S_f$ we have 
\begin{equation}\label{C2.2: Theorem 1 Eq 1}
[S_f:F]=[S_f:{\rm Nuc}_r(S_f)][{\rm Nuc}_r(S_f):F]= d^2m.
\end{equation}
Substituting $d = ks$ and $[{\rm Nuc}_r(S_f):F]=dms$ into Equation (\ref{C2.2: Theorem 1 Eq 1}) we obtain the equality $[S_f:{\rm Nuc}_r(S_f)]=k$.
If $f$ is not right-invariant, then $k>1$ and so we derive $s\not=d$ looking at the degree of $h$. More precisely, for $f$ of degree $m>1$, $f$ being not right-invariant is equivalent to $S_f$ being not associative which in turn is equivalent to
$k>1$.\\
(ii) If $d$ is prime then $d=ks$ forces $s=1$, as $k=1$ implies that $f$ is right-invariant, so that here
$R/Rh \cong M_k(E_{\hat{h}})$, ${\rm Nuc}_r(S_f) \cong E_{\hat{h}}$ and ${\rm deg}(h)=md.$
\end{proof}

\begin{corollary} \label{C3.1: Corollary 1}
Let $d=pq$ for primes $p$ and $q$ and suppose that $f$ is irreducible and not right-invariant. Then ${\rm Nuc}_r(S_f) \cong E_{\hat{h}}$ is a field extension of $C$ of degree $md$, or ${\rm Nuc}_r(S_f)$ is a central division algebra over $E_{\hat{h}}$ of degree $q$ (resp. $p$), and
$$[{\rm Nuc}_r(S_f):C] =mpq^2 \,\, ({\text resp. }\, mqp^2).$$
\begin{proof}
If $s=1$, then ${\rm Nuc}_r(S_f)\cong E_{\hat{h}}$ is a field extension of $C$ of degree $md$. Suppose that $s \neq 1$. Since $s$ divides $d=pq$, we have w.l.o.g. $s = q$. Then ${\rm Nuc}_r(S_f)$ is a central division algebra over $E_{\hat{h}}$ of degree $q$, and $$[{\rm Nuc}_r(S_f):C] =mds = mpq^2.$$
\end{proof}
\end{corollary}

Corollary \ref{C3.1: Corollary 1} generalises by induction as follows:

\begin{corollary} \label{C3.1: Corollary 2}
Let $d=p_1\cdots p_l$ be the prime decomposition of $n$ and suppose that $f$ is irreducible and not right-invariant. Then ${\rm Nuc}_r(S_f)\cong E_{\hat{h}}$ is a field extension of $C$ of degree $md$, or ${\rm Nuc}_r(S_f)$ is a central division algebra over $E_{\hat{h}}$ of degree $q=p_1\cdots p_r$, where $r < l$, and if $p = \frac{d}{q} = p_{r+1} \cdots p_l$, then
$$[{\rm Nuc}_r(S_f):C] =mpq^2.$$
\end{corollary}

\begin{theorem}\label{C3.1: Theorem 2}
(i) \quad Suppose that $\hat{h}(x)$ is irreducible in $F[x]$. Then $f$ is the product of $l \geq 1$ irreducible polynomials in $R$ all of which are mutually similar to each other, and $$R/Rh \cong M_k(\mathcal{E}(g))$$ where $g \in R$ is any irreducible divisor of $h$ in $R$. If ${\rm deg}(g)=r \geq 1$, then $m=rl$, $\mathcal{E}(g)$ is a central division algebra over $E_{\hat{h}}=F[x]/(\hat{h}(x))$ of degree $s^{\prime}=d/k$, where $k$ is the number of irreducible factors of $h$, and
 $$ {\rm Nuc}_r(S_f) \cong M_l(\mathcal{E}(g)).$$
 In particular, ${\rm Nuc}_r(S_f)$ is a central simple algebra over $E_{\hat{h}}$ of degree $s= ls^{\prime}$, ${\rm deg}(\hat{h})=\frac{rd}{s^{\prime}}=\frac{md}{s}$, ${\rm deg}(h) = \frac{rd}{s^{\prime}} = \frac{md}{s}$, and
 $$[{\rm Nuc}_r(S_f) :F] = l^2rds^{\prime} = mds.$$
 In particular $s^{\prime}$, $s$, and $l$ divide $d$.\\
(ii) If $d$ is prime and $f$ not right-invariant, then
$${\rm Nuc}_r(S_f) \cong E_{\hat{h}}.$$
In particular, then $[{\rm Nuc}_r(S_f) :F]=md$, ${\rm deg}(\hat{h})=md$, and ${\rm deg}(h)=md$.
\begin{proof}
(i) \quad Once again, we can take the first part of the proof (i) verbatim from the proof of Theorem \ref{C2.2: Theorem 2}. Now since $g$ is irreducible of degree $r$ with minimal central left multiple $h(t) = \hat{h}(u^{-1}t^n)$, $\mathcal{E}(g)$ is a central division algebra over $E_{\hat{h}}$ of degree $s^{\prime}=d/k$, where $k$ is the number of irreducible divisors of $h$ in $R$, ${\rm deg}(\hat{h})=\frac{rd}{s^{\prime}}=\frac{md}{ls^{\prime}}$ and ${\rm deg}(h)=\frac{rd}{s^{\prime}}=\frac{md}{ls^{\prime}}$ by Theorem \ref{C2.2: Theorem 1}. Finally, since $$\mathcal{E}(f) \cong M_l(\mathcal{E}(g)),$$ $\mathcal{E}(f)$ is a central simple algebra over $E_{\hat{h}}$ of degree $s = ls^{\prime}$, and $$[\mathcal{E}(f) : F] = s^2{\rm deg}(\hat{h}) = mds.$$\\
(ii) \quad Suppose that $d$ is prime and $f$ not right-invariant. As $s$ divides $d$, we are forced to take $s=1$ (else $s=d$ and $[{\rm Nuc}_r(S_f):F] = d^2m$, i.e. $f$ is right-invariant). The result follows immediately.
\end{proof}
\end{theorem}

\section{Prime Characteristic}
In this section, let ${\rm Char}(C)=p>0$. Let $\delta$ be a derivation of $D$, such that $\delta\vert_C$ is algebraic with minimum polynomial
$$g(t) = t^{p^e} + \gamma_1t^{p^{e-1}} + \gamma_2t^{p^{e-2}} + \dots + \gamma_{e-1}t^p + \gamma_et \in F[t],$$
 such that $g(\delta) = \delta_c$ is the inner derivation defined by $c \in D^{\times}$. Then we have $[C:F] = p^e$ for $F = C \cap {\rm Const}(\delta)$.  Then $R$ has dimension $d^2p^e$ over $F$ and every polynomial in $R$ is bounded (Corollary \ref{Intro: Corollary 1}). Recall that the two-sided elements of $R$ are the elements $uq(t)$ where $u \in D$ and $q(t) \in C(R)$, and $C(R) = F[g(t)-c]$, which is isomorphic to $F[x]$ under the map fixing elements of $F$ and sending $g(t)-c$ to $x$ \cite[Theorem 1.1.32]{jacobson2009finite}. Again we define the {\it minimal central left multiple} of a polynomial $f \in R$ to be the monic polynomial $h(t) = \hat{h}(g(t)-c)$ for some $\hat{h} \in F[x]$ of minimal degree, such that $h = gf$ for some $g \in R$.

\begin{lemma}\label{C3.2: Lemma 1}
\begin{enumerate}[(i)]
\item $f$ has a unique minimal central left multiple, which is a bound of $f$.
\item \cite[Theorem 13]{jacobson1943theory} If $f$ is irreducible, then $\hat{h}$ is irreducible in $F[x]$.
\item \cite{jacobson2009finite} The quotient algebra $R/Rh$ has center $$C(R/Rh) \cong F[x]/(\hat{h}(x)).$$
\item Suppose that $\hat{h}$ is irreducible in $F[x]$.
\begin{enumerate}
\item \cite[Theorem 13]{jacobson1943theory} $h$ generates a maximal two sided ideal in $R$.
\item All irreducible factors of $f$ are similar to all irreducible factors of $h$ in $R$. In particular, all irreducible factors of $f$ are mutually similar to each other.
\end{enumerate}
\end{enumerate}
\end{lemma}
\begin{proof}
This is identical to the proof of Lemma \ref{C3.1: Lemma 1}.
\end{proof}

Define $E_{\hat{h}}=F[x]/ (\hat{h}(x) )$. This is a commutative algebra over $F$ of dimension ${\rm deg}(\hat{h})$.  If $\hat{h}$ is irreducible in $F[x]$, then $E_{\hat{h}}$ is a field extension of $F$ of degree ${\rm deg}(\hat{h})$.

\begin{theorem}\label{C3.2: Theorem 2}
Let $f$ be irreducible. Then ${\rm Nuc}_r(S_f)$ is a central division algebra over $E_{\hat{h}}=F[x]/(\hat{h}(x))$ of degree $s=dp^e/k$, where $k$ is the number of irreducible factors of $h$ in $R$, and
 $$R/Rh \cong M_k({\rm Nuc}_r(S_f)).$$
 In particular, this means that ${\rm deg}(\hat{h})=\frac{md}{s}$, ${\rm deg}(h)=\frac{mdp^e}{s}$, and
 $$[{\rm Nuc}_r(S_f) :F]= mds.$$
 Moreover, $s$ divides ${\rm gcd}(dm,dp^e)$.
\end{theorem}

\begin{proof}
The first part of this proof can be taken verbatim to be the first part of the proof of Theorem \ref{C2.2: Theorem 1}. Comparing the dimensions of $R/Rh$ and $M_k({\rm Nuc}_r(S_f))$ over $F$ it follows that $d^2p^e \, {\rm deg}(h)= d^2p^{2e} \,{\rm deg} (\hat{h})=k^2s^2
\, {\rm deg}(\hat{h})$, so that we get $d^2p^{2e} =k^2s^2 $, that is $dp^e=ks$. In particular, this implies that ${\rm deg}(h)=\frac{dmp^e}{s}$ and ${\rm deg}(\hat{h})=\frac{dm}{s}$. Moreover $$[{\rm Nuc}_r(S_f):F]= [{\rm Nuc}_r(S_f):E_{\hat{h}}][E_{\hat{h}}:F] =s^2 {\rm deg}(\hat{h}) = \frac{dms^2}{s} = dms.$$
Now since ${\rm Nuc}_r(S_f)$ is a subalgebra of $S_f$ we have 
\begin{equation}\label{C2.2: Theorem 1 Eq 1}
[S_f:F]=[S_f:{\rm Nuc}_r(S_f)][{\rm Nuc}_r(S_f):F]= d^2mp^e.
\end{equation}
Substituting $dp^e = ks$ and $[{\rm Nuc}_r(S_f):F]=dms$ into Equation (\ref{C2.2: Theorem 1 Eq 1}) we obtain the equality $[S_f:{\rm Nuc}_r(S_f)]=k$.
If $f$ is not right-invariant, then $k>1$ and so we derive $s\not=dp^e$ looking at the degree of $h$. More precisely, for $f$ of degree $m>1$, $f$ being not right-invariant is equivalent to $S_f$ being not associative which in turn is equivalent to
$k>1$.
\end{proof}

\begin{corollary}\label{C3.2: Corollary 1}
Let $f$ be irreducible. If ${\rm gcd}(m,p^e)=1$, then $s$ divides $d$, and $f$ is not right-invariant. Additionally, if $d$ is prime, then one of the following holds:\\
(i)\quad ${\rm Nuc}_r(S_f) \cong E_{\hat{h}}$ is a field extension of $F$ of degree $md$. \\
(ii)\quad ${\rm Nuc}_r(S_f)$ is a central division algebra over $E_{\hat{h}}$ of degree $d$, ${\rm deg}(h) = mp^e$, ${\rm deg}(h)=m$, and $[{\rm Nuc}_r(S_f):F]=md^2$.
\end{corollary}

\begin{proof}
Since ${\rm gcd}(m,p^e)=1$,  ${\rm gcd}(dm,dp^e) 
=d$. Hence $s$ divides $d$ by Theorem \ref{C3.2: Theorem 2}, and we have $$[{\rm Nuc}_r(S_f):F]=mds \leq md^2 < md^2p^e = [S_f:F],$$ i.e. $f$ is not right-invariant. Additionally, if we suppose that $d$ is prime, then $s$ dividing $d$ implies that $s=1$ or $s=d$. (i) follows by considering $s=1$ and (ii) follows from $s=d$.
\end{proof}

The following result partially answers a question by Amitsur, who asked when the right nucleus of a polynomial is a central simple algebra.

\begin{theorem}\label{C3.2: Theorem 3}
 Suppose that ${\rm gcd}(d,p^e)=1$ and that $f$ is not right-invariant. Then $s=1$, or $s\not=1$ and $s$ divides either $d$ or $p^e$. \\
Suppose additionally that $d$ is prime and $e=1$. Then one of the following holds:
\\ (i) ${\rm Nuc}_r(S_f)\cong E_{\hat{h}},$ $dp=k$, ${\rm deg}(\hat{h})=dm$ and ${\rm deg}(h)=dp m$. In particular, then $[{\rm Nuc}_r(S_f) :F]=dm$.
\\ (ii) ${\rm Nuc}_r(S_f)$ is a central division algebra
 over $E_{\hat{h}}$ of degree $d$,  $p$ is the number of irreducible factors of $h$ in $R$, ${\rm deg}(h)=pm$, ${\rm deg}(\hat{h})=m$ and
 $$ R/Rh \cong M_k({\rm Nuc}_r(S_f)).$$
In particular, then $[{\rm Nuc}_r(S_f) :F]=d^2m$.
\\ (iii) ${\rm Nuc}_r(S_f)$ is a central division algebra over $E_{\hat{h}}$ of degree $p$,  $d$ is the number of irreducible factors of $h$ in $R$, ${\rm deg}(\hat{h})=dm/p$, ${\rm deg}(h)=dm$, and
 $[{\rm Nuc}_r(S_f) :F]=p^2/dm $.
 \\ Note that case (iii) cannot happen if $p$ does not divide $dm$ or if $dm$ does not divide $p^2$.
\end{theorem}

\begin{proof}
 It is clear that $s=1$, or else $s\not=1$ and $s$ divides either $d$ or $p^e$.
Suppose additionally that $d$ is prime, $e=1$. Then the equation $dp=ks$ forces that either $s=1$ and $k=pd$, or that $s\not=1$ and then $d=k$ and $p=s$ (or resp., $d=s$ and $p=k$).
 As before, $s=1$ yields (i).
\\
If $d=s\not=1$ and $p=k$ then this implies (ii) employing that $[{\rm Nuc}_r(S_f) :F]=[{\rm Nuc}_r(S_f) :E_{\hat{h}}][E_{\hat{h}}:F]=d^2 {\rm deg}(\hat{h})=d^2m$.
\\
If $d=k$ and $p=s\not=1$ then this implies (iii) using that $[{\rm Nuc}_r(S_f) :F]=[{\rm Nuc}_r(S_f) :E_{\hat{h}}][E_{\hat{h}}:F]=p^2 {\rm deg}(\hat{h})= p^2/dm $. In particular, this case means that ${\rm deg}(\hat{h})=dm/p$, which forces $n$ to divide $dm$, as well as
$[{\rm Nuc}_r(S_f) :F]=p^2/dm $ which in turn forces $dm$ to divide $n^2$.
\end{proof}

\begin{theorem}\label{C3.2: Theorem 4}
Suppose that $\hat{h}$ is irreducible in $F[x]$. Then $f$ is the product of $l \geq 1$ irreducible polynomials in $R$ all of which are mutually similar to each other, and $$R/Rh \cong M_k(\mathcal{E}(g))$$ where $g \in R$ is any irreducible divisor of $h$ in $R$. If ${\rm deg}(g)=r \geq 1$, then $m=rl$, $\mathcal{E}(g)$ is a central division algebra over $E_{\hat{h}}=F[x]/(\hat{h}(x))$ of degree $s^{\prime}=dp^e/k$, where $k$ is the number of irreducible factors of $h$, and
 $$ {\rm Nuc}_r(S_f) \cong M_l(\mathcal{E}(g)).$$
 In particular, ${\rm Nuc}_r(S_f)$ is a central simple algebra over $E_{\hat{h}}$ of degree $s= ls^{\prime}$, ${\rm deg}(\hat{h})=\frac{rd}{s^{\prime}}=\frac{md}{s}$, ${\rm deg}(h) = \frac{rdp^e}{s^{\prime}} = \frac{mdp^e}{s}$, and
 $$[{\rm Nuc}_r(S_f) :F] = l^2rds^{\prime} = mds.$$ In particular $s^{\prime}$, $s$, and $l$ divide ${\rm gcd}(dp^e,dm)$.
\begin{proof}
The first part of this proof is identical to the first part of the proof of Theorem \ref{C2.2: Theorem 2}. Now since $g$ is irreducible of degree $r$ with minimal central left multiple $h(t) = \hat{h}(u^{-1}t^n)$, $\mathcal{E}(g)$ is a central division algebra over $E_{\hat{h}}$ of degree $s^{\prime}=dp^e/k$, where $k$ is the number of irreducible divisors of $h$ in $R$, ${\rm deg}(\hat{h})=\frac{rd}{s^{\prime}}=\frac{md}{ls^{\prime}}$ and ${\rm deg}(h)=\frac{rdp^e}{s^{\prime}}=\frac{mdp^e}{ls^{\prime}}$ by Theorem \ref{C2.2: Theorem 1}. Finally, since $$\mathcal{E}(f) \cong M_l(\mathcal{E}(g)),$$ $\mathcal{E}(f)$ is a central simple algebra over $E_{\hat{h}}$ of degree $s = ls^{\prime}$, and $$[\mathcal{E}(f) : F] = s^2{\rm deg}(\hat{h}) = mds.$$
\end{proof}
\end{theorem}

\begin{corollary}\label{C3.2: Corollary 2}
Let $f \in R$ be reducible (i.e. $l \geq 2$) and suppose that $\hat{h}(x)$ is irreducible in $F[x]$. If ${\rm gcd}(m,p^e)=1$ (i.e. $p$ does not divide $m$), then $s$ divides $d$. If $d$ is prime and not equal to $p$,  then $f$ is not right-invariant and has $d$ irreducible factors in any factorisation into irreducible polynomials in $R$, each of which has degree $r=m/d$. Moreover $${\rm Nuc}_r(S_f) \cong M_d(E_{\hat{h}})$$ is a central simple algebra over $E_{\hat{h}}$ of degree $d$, ${\rm deg}(\hat{h})=m$, ${\rm deg}(h)=mp^e$, and $[{\rm Nuc}_r(S_f):F]=md^2$.
\end{corollary}

\begin{proof}
Suppose that $d$ is a prime not equal to $p$. Since $l$ divides $d$ by Theorem \ref{C3.2: Theorem 4}, we are forced to take $l=d$. 
Also by Theorem \ref{C3.2: Theorem 4}, $s=ds^{\prime}$ divides $d$, and we are forced to take $s^{\prime}=1$ and $s=d$. Therefore $f$ is not right-invariant, as $[{\rm Nuc}_r(S_f):F] = md^2 < [S_f:F]$, and the rest follows immediately from Theorem \ref{C3.2: Theorem 4} by applying $s=d$.

\end{proof}

\begin{proposition}\label{C3.2: Proposition 1}
Suppose that $\hat{h}$ is irreducible in $F[x]$. If $d$ divides $p$ and ${\rm gcd}(m,p^e) = 1$, then $f$ is irreducible.
\begin{proof}
First we note that $l$ divides ${\rm gcd}(dm,dp^e)$ in the notation of Theorem \ref{C3.2: Theorem 3}. Since ${\rm gcd}(m,p^e)=1$ by assumption, we must have that $l$ divides $d$, and so $l$ must also divide the prime $p$. Hence $l=1$ or $l=p$. Suppose that $l=p$, then $p$ divides $m$ (since $l$ divides $m$ by Theorem \ref{C3.2: Theorem 3}), which contradicts our assumption that ${\rm gcd}(m,p^e)=1$. We conclude that $l$ must be equal to $1$, i.e. that $f$ is irreducible.
\end{proof}
\end{proposition}

\subsection*{What if $D=K$ is a Field of Prime Characteristic?}
Let $K$ be a field with ${\rm Char}(K)=p>0$. Let $\delta$ be an algebraic derivation of $K$ with minimum polynomial
$$g(t) = t^{p^e} + \gamma_1t^{p^{e-1}} + \gamma_2t^{p^{e-2}} + \dots + \gamma_{e-1}t^p + \gamma_et \in F[t]$$
such that $g(\delta) = 0$ where $F = {\rm Const}(\delta)$. If $R=K[t;\delta]$ then $$C(R) = F[g(t)] \cong F[x].$$
\\
 We note that in this setting the results Lemma \ref{C3.2: Lemma 1} through to Corollary \ref{C3.2: Corollary 2} hold analogously, with $D=K$, $g(\delta)=0$, and $d=1$. Note that
$$S_f = K[t;\delta]/K[t;\delta]f(t)$$ is a nonassociative algebra over $F$ of dimension $mp^e$.
In this setting $f$ has minimal central left multiple $h(t)=\hat{h}(g(t))$ for some $\hat{h} \in F[x]$.

\begin{theorem}\label{C3.2': Theorem 1}
Let $f$ be irreducible. Then ${\rm Nuc}_r(S_f)$ is a central division algebra over $E_{\hat{h}}=F[x]/(\hat{h}(x))$ of degree $s=p^e/k$, where $k$ is the number of irreducible factors of $h$ in $R$, and
 $$ R/Rh \cong M_k({\rm Nuc}_r(S_f)).$$
 In particular, this means that ${\rm deg}(\hat{h})=\frac{m}{s}$, ${\rm deg}(h)=\frac{mp^e}{s}$, and
 $$[{\rm Nuc}_r(S_f) :F]= ms.$$
 Moreover, $s$ divides ${\rm gcd}(m,p^e)$ and $s = p^{\epsilon}$ for some $\epsilon \in \mathbb{Z}$, with $0 \leq \epsilon \leq e$.
\end{theorem}

\begin{proof}
We refer to the proof of Theorem \ref{C3.2: Theorem 2} for
$$R/Rh \cong M_k({\rm Nuc}_r(S_f)),$$ since it is identical to the proof in this case, too.
We note that $h$ has degree $km$, and since $f$ is irreducible, ${\rm Nuc}_r(S_f)$ is a division algebra.
Now $R/Rh$ is a central simple algebra over $E_{\hat{h}}$ and so
  ${\rm Nuc}_r(S_f)$ is a central division algebra over $E_{\hat{h}}$ of dimension $s^2$.
Comparing the dimensions of $R/Rh$ and $M_k({\rm Nuc}_r(S_f))$ over $F$ it follows that $p^e \, {\rm deg}(h)= p^{2e} \,{\rm deg}(\hat{h})=k^2s^2
\, {\rm deg}(\hat{h})$, so that we get $p^{2e} =k^2s^2$, that is $p^e=ks$. In particular, this implies that ${\rm deg}(h)=\frac{mp^{e}}{s}$ and ${\rm deg}(\hat{h})=\frac{m}{s}$.
\end{proof}

\begin{corollary}\label{C3.2': Corollary 1}
Let $f$ be irreducible of prime degree $m$. Then
 ${\rm Nuc}_r(S_f)$ a field extension of $F$ of degree $m$, or $E_{\hat{h}} = F$ and ${\rm Nuc}_r(S_f)$ is a central division algebra over $F$ of degree $p$, and ${\rm deg}(\hat{h})=1$.
\end{corollary}

\begin{proof}
By Theorem \ref{C3.2': Theorem 1}, $s$ divides $m$, and since $m$ is prime we must have $s=1$ or $s=m$. Suppose first that $s=1$, then ${\rm Nuc}_r(S_f)$ is a central divsion algebra over $E_{\hat{h}}$ of degree 1, i.e. ${\rm Nuc}_r(S_f)=E_{\hat{h}}$, a field extension over $F$ of degree ${\rm deg}(\hat{h}) = \frac{m}{s} = m$.\\
Now suppose that $s = m$, then ${\rm deg}(\hat{h}) = \frac{m}{s} = 1$ and $E_{\hat{h}} = F$. Hence ${\rm Nuc}_r(S_f)$ is a central division algebra over $F$ of degree $m$. Moreover, $m = s = \frac{p^e}{k}$, so that $m$ divides $p^e$, and since $m$ is prime $m=p$.
\end{proof}

\begin{corollary}\label{C3.2': Corollary 2}
Suppose that $\delta$ has minimum polynomial $g(t) = t^p - \gamma t \in F[t]$ (i.e. $e=1$) If $f$ is irreducible and not right-invariant, then ${\rm Nuc}_r(S_f)$ is a field extension of $F$ of degree $m$.
\end{corollary}

\begin{proof}
Since $p$ is prime and $f$ not right-invariant, by Theorem \ref{C3.2': Theorem 1} we are forced to take $s=1$ and $k=p$. Then ${\rm Nuc}_r(S_f) = E_{\hat{h}}$ is a field extension over $F$ of degree ${\rm deg}(\hat{h}) = \frac{m}{s} = m$.
\end{proof}

\begin{theorem}\label{C3.2': Theorem 2}
Suppose that $\hat{h}(x)$ is irreducible in $F[x]$. Then $f$ is the product of $l \geq 1$ irreducible polynomials in $R$ all of which are mutually similar to each other, and $$R/Rh \cong M_k(\mathcal{E}(g))$$ where $g \in R$ is any irreducible divisor of $h$ in $R$. If ${\rm deg}(g)=r \geq 1$, then $m=rl$, $\mathcal{E}(g)$ is a central division algebra over $E_{\hat{h}}=F[x]/(\hat{h}(x))$ of degree $s^{\prime}=p^e/k$, where $k$ is the number of irreducible factors of $h$, and
 $${\rm Nuc}_r(S_f) \cong M_l(\mathcal{E}(g)).$$
 In particular, ${\rm Nuc}_r(S_f)$ is a central simple algebra over $E_{\hat{h}}$ of degree $s= ls^{\prime}$, ${\rm deg}(\hat{h})=\frac{r}{s^{\prime}}=\frac{m}{s}$, ${\rm deg}(h) = \frac{rp^e}{s^{\prime}} = \frac{mp^e}{s}$, and
 $$[{\rm Nuc}_r(S_f) :F] = l^2rs^{\prime} = ms.$$ In particular $s^{\prime}$, $s$, and $l$ divide ${\rm gcd}(m,p^e)$.
\begin{proof}
This follows immediately from Theorem \ref{C3.2: Theorem 2} with $d=1$.
\end{proof}
\end{theorem}


\chapter{The Eigenring of $f \in R$ with $\hat{h}$ Reducible and Squarefree}
Let $D$ be a central division algebra over $C$ of degree $d$, let $\sigma$ be an automorphism of $D$ of finite inner order $n$, with $\sigma^n = \iota_u$, and let $f \in R=D[t;\sigma]$ have degree $m$. Recall that $R$ has center $$C(R) = F[u^{-1}t^n] \cong F[x],$$ and that all non-constant polynomials in $R$ are bounded. We let $h(t) = \hat{h}(u^{-1}t^n)$ denote the minimal central left multiple of $f$ for some monic polynomial $\hat{h} \in F[x]$. If $f(t)=t-a$ for some $a \in D^{\times}$ (i.e. $m=1$), then ${\rm Nuc}_r(S_f) = D$ and $\mathcal{E}(f)=\{ b \in D : \sigma(b)a =ab\}$. From now on, we assume that $m > 1$, so that $${\rm Nuc}_r(S_f) = \mathcal{E}(f) = {\rm End}_R(R/Rf).$$ If $(f,t)_r = 1$, then any bound $f^*$ of $f$ has the form $f^* = a\hat{h}(u^{-1}t^n)$ for some $a \in D^{\times}$ (alternatively, $f^*(t) = ah(t)$).
\\
\\
For the remainder of this chapter we assume that $(f,t)_r=1$ and that $\hat{h}$ is a squarefree polynomial in $F[x]$, that is $\hat{h}(x)=\hat{\pi}_1(x)\hat{\pi}_2(x) \cdots \hat{\pi}_z(x)$ for some $z \in \mathbb{N}$ and some monic, irreducible polynomials $\hat{\pi}_z,\hat{\pi}_2,\dots,\hat{\pi}_z$ in $F[x]$. We also use the notation $\pi_i(t) = \hat{\pi}_i(u^{-1}t^n)$ for $i \in \{1,2,\dots,z\}$.

\begin{lemma}\label{C4: Lemma 1}
If $\hat{\pi}_i(x) \neq \hat{\pi}_j(x)$ for $i,j \in \{1,2,\dots,z\}$, then ${\rm gcd}(\pi_i(t),\pi_j(t))=1$.
\begin{proof}
Suppose that $\hat{\pi}_i(x) \neq \hat{\pi}_j(x)$ for some $i,j \in \{1,2,\dots,z\}$. Since $\hat{\pi}_i,\hat{\pi}_j$ are monic, irreducible polynomials in $F[x]$, ${\rm gcd}(\hat{\pi}_i(x),\hat{\pi}_j(x))=1$. By Bezout's Lemma, there exist polynomials $\hat{p},\hat{q} \in F[x]$ such that $$\hat{p}(x)\hat{\pi}_i(x) + \hat{q}(x)\hat{\pi}_j(x) = 1.$$ Let $p(t) = \hat{p}(u^{-1}t^n)$ and $q(t) = \hat{q}(u^{-1}t^n)$, then $$p(t)\pi_i(t) + q(t)\pi_j(t) = 1.$$ We conclude that ${\rm gcd}(\pi_i(t),\pi_j(t))=1$, by Bezout's Lemma.
\end{proof}
\end{lemma}

Strictly, the proof of Lemma \ref{C4: Lemma 1} shows that ${\rm gcrd}(\pi_i(t),\pi_j(t))=1$. However, the polynomials $\pi_i(t)$ and $\pi_j(t)$ lie in $C(R)$, hence they commute with all other polynomials in $R$, and so there is no distinction between the greatest common right divisor and the greatest common left divisor of $\pi_i(t)$ and $\pi_j(t)$. We call this simply the greatest common divisor. For the rest of the chapter we assume that $\hat{\pi}_i(x) \neq \hat{\pi}_j(x)$ for all $i,j \in \{1,2,\dots,z\}$ such that $i \neq j$. In this case $\hat{h}$ is said to be a {\it squarefree polynomial} in $F[x]$.

\begin{lemma}\cite[Chapter 1]{lam2006exercises}\label{C4: Lemma 2}
Let $S$ be a semisimple Artinian ring with direct sum decomposition $$S=S_1 \oplus S_2 \oplus \cdots \oplus S_z$$ for $S_i$ a simple Artinian ring $1 \leq i \leq z$, and let $M$ be a left $S$-module. Then $M$ is a semisimple module over $S$, and there are exactly $z$ isomorphism classes of simple left $S$-modules. Moreover, if $\{ V_i: i=1,2,\dots,z \}$ is a collection of representatives of the $z$ isomorphism classes of simple left $S$-modules that occur as submodules of $M$, then $$M = \bigoplus\limits_{i=1}^z M_i,$$ where $M_i$ denotes the direct sum of all submodules of $M$ that are isomorphic to $V_i$. That is, $$M \cong \bigoplus_{i=1}^z (V_i)^{\oplus n_i},$$ where $n_i$ is the number of submodules of $M$ that are isomorphic to $V_i$.
\end{lemma}

\begin{remark}[Submodules of $R/Rf$]\label{C4: Remark 1}
By \cite[pg.~33]{jacobson1943theory}, any submodule of the left $R$-module $R/Rf$ has the form $Rg/Rf$, where $f=pg$ for some $p \in R$. Moreover, there is an $R$-module isomorphism $$Rg/Rf \cong R/Rp.$$ Therefore $B$ is a simple submodule of $R/Rf$ if and only if $B \cong R/Rp$ for some irreducible polynomial $p \in R$ such that $f = pg$ for some $g \in R$\\
\\
 Now, by \cite[Proposition 5.2]{gomez2013computing}, there is a ``rough decomposition'' of $f$, of the form $f = g_1g_2\cdots g_z$ for some $g_i \in R$, such that $g_i$ has minimal central left multiple $\pi_i$ for each $i$. We note that, for all $i$, any two irreducible factors of $g_i$ are also both irreducible factors of $\pi_i$, and are therefore similar to each other as polynomials \cite[Theorem 1.2.19]{jacobson2009finite}, as $\pi_i$ is a two-sided maximal element. Since the irreducible factors $\hat{\pi}_i$ of $\hat{h}$ in $F[x]$ can be permuted in the decomposition of $\hat{h}(x)$ (as $F[x]$ is a commutative ring of polynomials), we infer that the ``rough decomposition'' of $f$ is not unique, i.e. for any permutation $\beta \in {\rm S}_z$, there exist $z$ polynomials $g_1^{\prime},g_2^{\prime},\dots,g_z^{\prime} \in R$, such that $f=g_1^{\prime}g_2^{\prime}\cdots g_z^{\prime}$, and $g_i^{\prime}$ has minimal central left multiple $\pi_{\beta(i)}$.\\
\\ For $f_i \in R$ any irreducible factor of $f$ with minimal central left multiple $\pi_i(t)$, there exists a ``rough decomposition'' $f=g_1^{\prime}\cdots g_z^{\prime}$ and a permutation $\beta \in {\rm S}_z$, such that $f_i$ is an irreducible factor of $g_1^{\prime}$ and that $i = \beta(1)$. Furthermore, since all irreducible factors of $g_1^{\prime}$ are mutually similar, there exists some irreducible $f_i^{\prime} \in R$, and some $g \in R$ such that $f_i \sim f_i^{\prime}$, and $g_1^{\prime}=f_i^{\prime}g$, i.e. $f=f_i^{\prime}gg_2^{\prime}\cdots g_z^{\prime}$. This means that $R/Rf_i^{\prime}$ is a simple submodule of $R/Rf$, and $$R/Rf_i \cong R/Rf_i^{\prime}.$$
\end{remark}
From the argument of Remark \ref{C4: Remark 1}, we immediately obtain the following:

\begin{lemma}\label{C4: Lemma 3}
If $f_i$ is an irreducible factor of $f$ in $R$, then the $R$-module $R/Rf_i$ is isomorphic to a simple submodule of $R/Rf$.
\end{lemma}

\begin{theorem}\label{C4: Theorem 1}

\begin{enumerate}[(i)]
\item $A=R/Rh$ is a semisimple Artinian ring which is isomorphic to the direct sum of simple Artinian rings $$R/Rh \cong R/R\pi_1 \oplus R/R\pi_2 \oplus \cdots \oplus R/R\pi_z.$$ Moreover, this is an isomorphism of $E_{\hat{h}}$-algebras.
\item The center of $R/Rh$ is a direct sum of fields $$C(R/Rh) = C(R/R\pi_1) \oplus C(R/R\pi_2) \oplus \cdots \oplus C(R/R\pi_z).$$ That is, $$E_{\hat{h}} \cong E_{\hat{\pi}_1} \oplus E_{\hat{\pi}_2} \oplus \cdots \oplus E_{\hat{\pi}_z},$$ where $E_{\hat{\pi}_i} := F[x]/\langle \hat{\pi}_i(x) \rangle$ is a field extension of $F$ of degree equal to the degree of $\hat{\pi}_i$ in $F[x]$ for each $i$.
\item Let $f_1,f_2,\dots,f_z \in R$ be irreducible polynomials of degrees $m_1,m_2,\dots,m_z$ respectively, such that $f_i(t)$ is an irreducible factor of $\pi_i(t)$, and let $k_i$ be the number of irreducible factors in any factorisation of $\pi_i$ into irreducibles in $R$, for each $i$. Then $\mathcal{E}(f_i)$ is a central division algebra over $E_{\hat{\pi}_i}$ of degree $s_i = dn/k_i$, and $$R/R\pi_i \cong M_{k_i}(\mathcal{E}(f_i))$$ is a central simple algebra over $E_{\hat{\pi}_i}$ of degree $k_i s_i$. In particular, ${\rm deg}(\hat{\pi}_i) = \frac{m_id}{s_i}$, ${\rm deg}(\pi_i) = \frac{m_idn}{s_i}$, and the following is an isomorphism of rings and of $E_{\hat{h}}$-algebras $$R/Rh \cong M_{k_1}(\mathcal{E}(f_1)) \oplus M_{k_2}(\mathcal{E}(f_2)) \oplus \cdots \oplus M_{k_z}(\mathcal{E}(f_z)).$$
\item There is a factorisation $f(t)=g_1(t)g_2(t)\cdots g_z(t)$ for $g_1,g_2,\dots,g_z \in R$ such that, for each $i$, $g_i$ has minimal central left multiple $\pi_i$, and all irreducible factors of $g_i$ are mutually similar to each other. 
\item The following is an isomorphism of left $A$-modules, $$R/Rf \cong (R/Rf_1)^{\oplus r_1} \oplus (R/Rf_2)^{\oplus r_2} \oplus \cdots \oplus (R/Rf_z)^{\oplus r_z},$$ where $r_i$ denotes the number of irreducible factors of $g_i$ for each $i$.
\item The following is an isomorphism of rings, and of $E_{\hat{h}}$-algebras: $${\rm Nuc}_r(S_f) \cong M_{r_1}(\mathcal{E}(f_1)) \oplus M_{r_2}(\mathcal{E}(f_2)) \oplus \cdots \oplus M_{r_z}(\mathcal{E}(f_z)).$$ In particular, ${\rm Nuc}_r(S_f)$ is both a semisimple Artinian ring and a semisimple $E_{\hat{h}}$-algebra.
\end{enumerate} 
\begin{proof}
\begin{enumerate}[(i)]
\item By assumption, the polynomial $\hat{h}$ admits a factorisation into distinct monic, irreducible polynomials in $F[x]$, i.e. $\hat{h}(x)= \hat{\pi}_1(x) \hat{\pi}_2(x) \cdots \hat{\pi}_z(x)$ where $\hat{\pi}_i(x) \neq \hat{\pi}_j(x)$ for $i \neq j$. A decomposition of $\hat{h}$ into irreducible polynomials in $F[x]$ always exists, and is unique, since $F[x]$ is a unique factorisation domain (e.g. see \cite[pg. 92-95]{adkins2012algebra}). Now, since $\hat{\pi}_i(x) \neq \hat{\pi}_j(x)$ for $i \neq j$, we have that $(\pi_i(t),\pi_j(t))=1$ for each distinct pair $i,j$, by Lemma \ref{C4: Lemma 1}. Also, by \cite[pg.~3]{gomez2013computing}, we have that $R\pi_i + R\pi_j = R(\pi_i,\pi_j)_r = R$, since ${\rm gcd}(\pi_i(t),\pi_j(t)) = (\pi_i,\pi_j)_r = 1$ for any $i \neq j$. Hence the Chinese Remainder Theorem for non-commutative rings (e.g.~ \cite{ore1952general}) yields $$\frac{R}{Rh} = \frac{R}{R\pi_1 \cap R\pi_2 \cap \cdots \cap R\pi_z} \cong \frac{R}{R\pi_1} \oplus \frac{R}{R\pi_2} \oplus \cdots \oplus \frac{R}{R\pi_z}$$ with the first equality due to the fact that $$R\pi_1 \cap R\pi_2 \cap \dots \cap R\pi_z = R\, {\rm lcm}(\pi_1,\pi_2,\dots,\pi_z) = Rh$$ (e.g. \cite[pg. 14]{gomez2012basic}).
\item Since $R/Rh$ is a semisimple Artinian ring isomorphic to the direct sum of the simple Artinian rings $R/R\pi_i$ for $i=1,2,\dots,z$, we have that $$C(R/Rh) \cong C(R/R\pi_1) \oplus C(R/R\pi_2) \oplus \cdots \oplus C(R/R\pi_z),$$ and $C(R/R\pi_i)$ is a field for each $i$, by \cite[Theorems 23-24]{jacobson1943theory}. Now, since $C(R/R\pi_i) \cong F[x]/\langle \hat{\pi}_i(x) \rangle =: E_{\hat{\pi}_i}$ for all $i$ and $C(R/Rh) \cong E_{\hat{h}}$ by Lemma \ref{C2.1: Lemma 1}, $$E_{\hat{h}} \cong E_{\hat{\pi}_1} \oplus  E_{\hat{\pi}_2} \oplus \cdots \oplus E_{\hat{\pi}_z}.$$
\item By Theorem \ref{C2.2: Theorem 1}, the eigenring $\mathcal{E}(f_i)$ is a central division algebra over $E_{\hat{\pi}_i}$ of degree $s_i = dn/k_i$ for all $i$. Moreover, $$R/R\pi_i \cong M_{k_i}(\mathcal{E}(f_i)),$$  is a central simple algebra of degree $k_is_i=dn$ over $E_{\hat{\pi}_i}$, ${\rm deg}(\hat{\pi}_i)=m_id/s_i$, and ${\rm deg}(h)=m_idn/s_i$, for all $i$. The stated isomorphism follows immediately from (i) and (ii).
\item The existence of the factorisation $f(t) = g_1(t)\cdots g_z(t)$ with each $g_i$ having minimal central left multiple $\pi_i$, is essentially \cite[Proposition 5.2]{gomez2013computing}. We point out that for each $i$, any irreducible divisor of $g_i$ is also an irreducible divisor of $\pi_i$, by definition of the minimal central left multiple of $g_i$. In the language of \cite{jacobson2009finite}, $\pi_i$ is a two-sided maximal element of $R$, and by \cite[Theorem 1.2.19]{jacobson2009finite}, any two irreducible divisors of $\pi_i$ in $R$ are similar polynomials. In particular, any two irreducible divisors of $g_i$ are similar, for each $i$.
\item By (i), $A=R/Rh$ is a semisimple Artinian ring, with $z$ summands in any direct sum decomposition into simple Artinian rings. Then any left $A$-module is isomorphic to a direct sum of simple left $A$-modules, and there are precisely $z$ isomorphism classes of simple left $A$-modules, by Lemma \ref{C4: Lemma 2}. \\
\\
Consider the left $A$-module $R/Rf$. By \cite[pg.~1]{gomez2013computing}, the lattice of submodules of $R/Rf$ as a left $A$-module is identical to its lattice of submodules as a left $R$-module. By Lemma \ref{C4: Lemma 3}, we have that any irreducible factor $p$ of $f$ in $R$, yields a simple left $R$-module $R/Rp$ that is isomorphic to a simple $R$-submodule of $R/Rf$, and hence is isomorphic to a simple left $A$-submodule of $R/Rf$ also. Now, every simple left $A$-module which occurs as a submodule of $R/Rf$ appears as a direct summand in such a decomposition (up to isomorphism), by Lemma \ref{C4: Lemma 2}. Therefore, since there are $r_i$ irreducible polynomials similar to $f_i$ in any factorisation of $f$ into irreducibles in $R$ for each $i$, we have that the isomorphism class of $R/Rf_i$ has size $r_i$ for all $i$. Hence, $$R/Rf \cong (R/Rf_1)^{\oplus r_1} \oplus (R/Rf_2)^{\oplus r_2} \oplus \cdots \oplus (R/Rf_z)^{\oplus r_z},$$ as left $A$-modules. 
\item By (iii) and Lemma \ref{Intro: Lemma 1}, $${\rm End}_A(R/Rf) \cong {\rm End}_A((R/Rf_1)^{\oplus r_1} \oplus (R/Rf_2)^{\oplus r_2} \oplus \cdots \oplus (R/Rf_z)^{\oplus r_z})$$ as rings. By Lemmas \ref{Intro: Lemma 2} and \ref{Intro: Lemma 3},$${\rm End}_A(R/Rf) \cong \bigoplus_{i=1}^z {\rm End}_A((R/Rf_i)^{\oplus r_i}) \cong \bigoplus_{i=1}^z M_{r_i}({\rm End}_A(R/Rf_i)).$$ Also, since $Rh$ is a two-sided ideal of $R$ contained in ${\rm Ann}_R(R/Rf_i)$ and $Rh = {\rm Ann}_R(R/Rf)$, we have $${\rm End}_R(R/Rf) \cong \bigoplus_{i=1}^z M_{r_i}({\rm End}_R(R/Rf_i)),$$ by Lemma \ref{Intro: Lemma 4}. Therefore, since ${\rm End}_R(R/Rp) = \mathcal{E}(p)$ for any $p \in R$, and since $\mathcal{E}(p) = {\rm Nuc}_r(S_p)$ for any $p \in R$ of degree at least 2, we have $${\rm Nuc}_r(S_f) \cong \bigoplus_{i=1}^z M_{r_i}(\mathcal{E}(f_i)).$$
\end{enumerate}
\end{proof}
\end{theorem}


\chapter{Subalgebras of the Right Nucleus}
 
Throughout this chapter, unless stated otherwise, let $D$ be an associative division algebra over its center $C$, let $\sigma$ be an endomorphism of $D$, and let $\delta$ be a left $\sigma$-derivation of $D$. Let
$F = C \cap {\rm Fix}(\sigma) \cap {\rm Const}(\delta)$. The results of this chapter in the special case that $\delta = 0$ and $D$ is a cyclic field extension of $F$ appear in \cite{owen2019eigenspaces}.

\section{The General Case $R=D[t;\sigma,\delta]$}
In this section we take yet another approach to investigate the Petit algebras associated with the polynomials in the ring $R=D[t;\sigma,\delta]$, and their nuclei. Throughout this chapter we let $f \in R$ be  monic polynomial of degree $m > 1$.\\
\\
In order to better understand the right nucleus of $S_f = D[t;\sigma,\delta]/D[t;\sigma,\delta]f$, we investigate firstly which elements of the coefficient ring $D$ lie in ${\rm Nuc}_r(S_f)$, and secondly which powers of $t$ lie in ${\rm Nuc}_r(S_f)$, before constructing a particular vector subspace of ${\rm Nuc}_r(S_f)$. To this end, we define
$$L^{(\sigma,\delta,f)} = {\rm Nuc}_r(S_f) \cap D$$ which is equal to the set $$\lbrace c \in D: [a,b,c] = 0 \text{ for all } a,b \in S_f \rbrace.$$
For $\sigma = {\rm id}$, we use the notation $L^{({\rm id},\delta,f)} = L^{(\delta,f)}$, and for $\delta = 0$ we use the notation $L^{(\sigma,0,f)} = L^{(\sigma,f)}$. From the definition of $L^{(\sigma,\delta,f)}$ we immediately have:

\begin{lemma}\label{C5.1: Lemma 1}
If $f$ is not right-invariant, then $$L^{(\sigma,\delta,f)} = {\rm Nuc}(S_f)=D \cap {\rm Nuc}_r(S_f).$$
\begin{proof}
Since $f$ is not right-invariant ${\rm Nuc}_l(S_f) = {\rm Nuc}_m(S_f) = D$, by Theorem \ref{Intro: Theorem 3}. Hence
$${\rm Nuc}(S_f) = {\rm Nuc}_l(S_f) \cap {\rm Nuc}_m(S_f) \cap {\rm Nuc}_r(S_f) = D \cap {\rm Nuc}_r(S_f),$$ and the result follows.
\end{proof}
\end{lemma}

From now on, we write $L=L^{(\sigma,\delta,f)}$ when it is clear from the context which skew polynomial ring $D[t;\sigma,\delta]$ and which skew polynomial $f$ is used.\\
\\
In the following proposition we show that ${\rm Nuc}(S_f)$ is a division subring of $D$ containing the field $F$, if $f$ is not right-invariant:

\begin{proposition}\label{C5.1: Proposition 1}
\begin{enumerate}[(i)]
\item If $f$ is right-invariant then $L=D$.
\item If $f$ is not right-invariant then $L$ is a division subring of $D$ containing $F$.
\end{enumerate}
\begin{proof}
\begin{enumerate}[(i)]
\item If $f$ is right-invariant, then $S_f$ is associative, i.e. $[a,b,c]=0$ for all $a,b,c \in S_f$. In particular, $[a,b,c]=0$ for all $a,b \in S_f$ and all $c \in D$, hence $L=D$.
\item Suppose that $f$ is not right-invariant. Then $L={\rm Nuc}(S_f) = {\rm Nuc}_r(S_f) \cap D$ is a subring of $D$. Moreover, it is well known that ${\rm Nuc}(S_f)$ is an associative subalgebra of $S_f$ over $F$, hence it contains $F$ as a subfield. It remains to show that every nonzero element of $L$ has a multiplicative inverse, i.e. $L$ is division. Let $c \in L$ be nonzero. Then $c$ has a unique multiplicative inverse $c^{-1} \neq 0$ in $D$ since $L \subseteq D$. Then we have $$[a,b,c^{-1}]c = (a(bc^{-1}) - (ab)c^{-1})c = ab - ab,$$ as $c \in L \subseteq {\rm Nuc}_r(S_f)$. Hence $[a,b,c^{-1}]c = 0$ for any $a,b \in S_f$, which yields $$([a,b,c^{-1}]c)c^{-1} = [a,b,c^{-1}](cc^{-1}) = [a,b,c^{-1}] = 0$$ because $c \in D = {\rm Nuc}_m(S_f)$.
Therefore $c^{-1} \in {\rm Nuc}_r(S_f) \cap D = L$, and so $L={\rm Nuc}(S_f)$ is a division ring.
\end{enumerate}
\end{proof}
\end{proposition}

\begin{corollary}\label{C5.1: Corollary 1}
If $D$ is a field and $f$ not right-invariant, then $L = {\rm Nuc}(S_f)$ is an intermediate field of $D/F$.
\end{corollary}

We call $a \in D$ a {\it (right) semi-invariant element} with respect to $f$ (or $f$-semi-invariant) if $fa \in Df$. Let $M$ be a division subring of $D$. We say that $f$ is {\it $M$-weak semi-invariant} if $fM \subset Df$. If $f$ is $D$-weak semi-invariant, it is called {\it (right) semi-invariant}. It is clear that $fb \in Df$ for any $b \in L$. Hence any polynomial $f \in R$ is $L$-weak semi-invariant. A polynomial $f \in R$ is right-invariant if and only if it is (right) semi-invariant, and $ft \in Rf$ (e.g. \cite{lam1989invariant}). We can now restate this criteria in terms of $f$-semi-invariant elements, and $L$.

\begin{proposition}\label{C5.1: Proposition 2}
$f$ is right-invariant if and only if $L = D$ and $t \in {\rm Nuc}_r(S_f)$.
\begin{proof}
We recall that $f$ is right-invariant if and only if $S_f$ is associative.\\
Suppose that $f$ is right-invariant. Then $L=D$ by Proposition \ref{C5.1: Proposition 1}. As previously stated, if $f$ is right-invariant then $S_f$ is associative, which in turn implies that $S_f$ is equal to its right nucleus. Hence $t \in S_f$ yields $t \in {\rm Nuc}_r(S_f)$.
Conversely, suppose that $L = D$ and $t \in {\rm Nuc}_r(S_f)$. Since $L$ is contained in the right nucleus, and the right nucleus is closed under addition and multiplication the direct sum $$D \oplus Dt \oplus Dt^2 \oplus \dots \oplus Dt^{m-1}$$ is an $F$-subspace of ${\rm Nuc}_r(S_f)$. But $D \oplus Dt \oplus Dt^2 \oplus \dots \oplus Dt^{m-1}$ is equal to $S_f$ as vector spaces, hence $S_f \subseteq {\rm Nuc}_r(S_f)$. Since ${\rm Nuc}_r(S_f)$ is a subalgebra of $S_f$, $S_f={\rm Nuc}_r(S_f)$, i.e. $S_f$ is associative. The result follows immediately.
\end{proof}
\end{proposition}

The following theorem describes equivalent conditions for an element $c \in D$ to be semi-invariant with respect to $f$, and further we show that $L$ is precisely the subset of $D$ consisting of the semi-invariant elements with respect to $f$.

\begin{theorem}\label{C5.1: Theorem 1}
Let $f(t) = t^m - \sum\limits_{i=0}^{m-1}a_it^i \in R$ and $c \in D$. Then the following are equivalent:
\begin{enumerate}
\item $c$ is a semi-invariant element with respect to $f$,
\item $c \in L^{(\sigma,\delta,f)}$,
\item $\sigma^m(c)a_k = \sum\limits_{j=k}^{m-1} a_j\Delta_{j,k}(c) - \Delta_{m,k}(c)$ for all $k \in \lbrace 0,1,\dots,m-1 \rbrace$.
\end{enumerate}
\begin{proof}
To prove this we show first that (1) is true if and only if (2) is true, then we show that (1) is true if and only if (3) holds.\\
\\
Suppose that $c$ is semi-invariant with respect to $f$. Then $fc \in Df \subset Rf$ by definition, i.e. $c \in {\rm Nuc}_r(S_f)$. Hence $c \in {\rm Nuc}_r(S_f) \cap D = L$. Conversely, let $c \in L$. Then, in particular, $c \in {\rm Nuc}_r(S_f)$, and so $fc \in Rf$. This means that there exists $c^{\prime} \in R$ such that $fc = c^{\prime}f$. Comparing the degree of $fc$ and $c^{\prime}f$ it follows that $c^{\prime} \in D$, hence $c$ is a semi-invariant element with respect to $f$. Therefore the statements (1) and (2) are equivalent.\\
\\
Now suppose that $c$ is semi-invariant with respect to $f$. By definition, $fc \in Df$, that is $fc = c^{\prime}f$ for some $c^{\prime} \in D$. Computing each side gives
\begin{align}\label{Equation 1: Calculation of fc=c'f}
fc = (t^m - \sum\limits_{i=0}^{m-1}a_it^i)c = \sum\limits_{j=0}^m \Delta_{m,j}(c)t^j - \sum\limits_{i=0}^{m-1} \sum\limits_{j=0}^i a_i\Delta_{i,j}(c)t^j = c^{\prime}t^m - \sum\limits_{i=0}c^{\prime}a_it^i.
\end{align}
If we compare the coefficients of $t^m$ we get:
\begin{equation}\label{Equation 2: c^prime = sigma^m(c)}
c^{\prime} = \Delta_{m,m}(c) = \sigma^m(c).
\end{equation}
Now, comparing the coefficients of $t^k$ for $k \in \lbrace 0,1,...,m-1 \rbrace$ gives
\begin{align}\label{Equation 3: c^primea_k relation}
&  a_k\Delta_{k,k}(c) + a_{k+1}\Delta_{k+1,k}(c) + \dots + a_{m-1}\Delta_{m-1,k}(c) - \Delta_{m,k}(c) = c^{\prime}a_k.
\end{align}
Combining Equations \ref{Equation 2: c^prime = sigma^m(c)} and \ref{Equation 3: c^primea_k relation} yields
\begin{align}
\sigma^m(c)a_k &= \sum\limits_{i=k}^{m-1} a_i\Delta_{i,k}(c) - \Delta_{m,k}(c)
\end{align}
for all $k \in \lbrace 0,1,\dots,m-1\rbrace$. Conversely, suppose that $$\sigma^m(c)a_k = \sum\limits_{i=k}^{m-1} a_i\Delta_{i,k}(c) - \Delta_{m,k}(c)$$ for all $k \in \lbrace 0,1,\dots,m-1 \rbrace$. Then $$fc = \sum\limits_{j=0}^m \Delta_{m,j}(c)t^j - \sum\limits_{i=0}^{m-1} \sum\limits_{j=0}^i \Delta_{i,j}(c)t^j = \sigma^m(c)f$$ by Equation \ref{Equation 1: Calculation of fc=c'f}. Hence (1) is equivalent to (3).
\end{proof}
\end{theorem}

Due to Theorem \ref{C5.1: Theorem 1}, we sometimes refer to $L$ as the subring of semi-invariant elements in $D$ with respect to the skew polynomial $f$. This means that if $f$ is not right-invariant, then the set of semi-invariant elements in $D$ with respect to $f$ is the nucleus of $S_f$. Petit provides the following result in \cite{petit1966certains}:

\begin{theorem} \cite[(5)]{petit1966certains} \label{C5.1: Theorem 2}
Let $f(t) = t^m - \sum_{i=0}^{m-1}a_it^i \in R = D[t;\sigma,\delta].$ Then the following are equivalent:
\begin{enumerate}
\item $t \in {\rm Nuc}_r(S_f)$,
\item $t^m \circ_f t = t \circ_f t^m$,
\item $ft \in Rf$,
\item all powers of $t$ are associative in $S_f$.
\end{enumerate}
\end{theorem}

We obtain the following additional criterium for $t$ to be contained in the right nucleus:

\begin{theorem}\label{C5.1: Theorem 3}
Let $f(t) = t^m - \sum\limits_{i=0}^{m-1}a_it^i \in R$. If $a_i \in {\rm Fix}(\sigma) \cap {\rm Const}(\delta)$ for all $i \in \lbrace 0,1,...,m-1 \rbrace$, then $t \in {\rm Nuc}_r(S_f)$.
\end{theorem}
\begin{proof}
Suppose that $a_i \in {\rm Fix}(\sigma) \cap {\rm Const}(\delta)$. Then $$ft = (t^m-\sum\limits_{i=0}^{m-1}a_it^i)t = t^{m+1} - \sum\limits_{i=0}^{m-1}a_it^{i+1} = t^{m+1} - \sum\limits_{i=0}^{m-1}(t\sigma^{-1}(a_i) - \delta(\sigma^{-1}(a_i)))t^i,$$ but $a_i \in {\rm Fix}(\sigma) \cap {\rm Const}(\delta)$ for all $i \in \lbrace 0,1,\dots,m-1\rbrace$, which gives $\sigma^{-1}(a_i) = a_i$ and $\delta(\sigma^{-1}(a_i))=\delta(a_i)=0$ for all $i$. So $$ft = t^{m+1} - \sum\limits_{i=0}^{m-1}(t\sigma^{-1}(a_i) - \delta(\sigma^{-1}(a_i)))t^i = t(t^m-\sum\limits_{i=0}^{m-1}a_it^i) = tf \in Rf.$$ Hence $t \in {\rm Nuc}_r(S_f)$ by Theorem \ref{C5.1: Theorem 2}.
\end{proof}

From now on, we denote by $F_k$ the set ${\rm Fix}(\sigma^k) \cap {\rm Const}(\delta)$ for any $k \in \mathbb{Z}$. We obtain the following generalization of Theorem \ref{C5.1: Theorem 3}:

\begin{theorem}\label{C5.1: Theorem 4}
Let $f(t) = t^m - \sum\limits_{i=0}^{m-1} a_it^i \in R$ and let $k \in \lbrace 1,2,\dots,m-1 \rbrace$. If $a_i \in F_k$ for all $i \in \lbrace 0,1,\dots,m-1 \rbrace$, then $t^k \in {\rm Nuc}_r(S_f)$. In particular, then
$$t^m \circ_f t^k = t^k \circ_f t^m.$$
\end{theorem}

\begin{proof}
Suppose that $a_i \in F_k$ for all $i$. Then in $R$, we have
\begin{align*}
t^kf &= t^k(t^m-\sum\limits_{i=0}^{m-1} a_it^i) \\
     &= t^{m+k} - t^k \sum\limits_{i=0}^{m-1} a_it^i \\
     &= t^{m+k} - \sum\limits_{i=0}^{m-1}\sum\limits_{j=0}^{k} \Delta_{k,j}(a_i)t^{i+j} \\
     &= t^{m+k} - \sum\limits_{i=0}^{m-1} \sigma^k(a_i)t^{i+k} \text{ (as $a_i \in {\rm Const}(\delta)$ $\forall i$) }\\
     &= t^{m+k} - \sum\limits_{i=0}^{m-1} a_it^{i+k} \text{ (as $a_i \in {\rm Fix}(\sigma^k)$ $\forall i$) }\\
     &= ft^k \in Rf,
\end{align*}
i.e. $ft^k = t^kf \in Rf$, and so $t^k \in {\rm Nuc}_r(S_f)$ as claimed.\\
\\
Since $t^k \in {\rm Nuc}_r(S_f)$, we have in particular that $[t^k,t^{m-k},t^k]=0$ in $S_f$, that is $t^k\circ_f(t^{m-k}\circ_f t^k)=(t^k\circ_f t^{m-k})\circ_f t^k$. Therefore $t^k \circ_f t^m=t^m \circ_f t^k$.
\end{proof}

\begin{proposition}\label{C5.1: Proposition 3}
Let $f(t) = t^m-\sum\limits_{i=0}^{m-1}a_it^i \in R$. Suppose that there exist $k \in \{1,2,\dots,k-1\}$ such that $a_i \in F_k$ for all $i$ and that there is no $j \in \{1,2,\dots,k-1\}$ such that $a_i \in F_j$ for all $i$.
\\ (i) If $m=qk$ for some positive integer $q$, then
$$L \oplus Lt^k \oplus Lt^{2k} \oplus \dots \oplus Lt^{(q-1)k}\oplus L(\sum_{i=0}^{m-1} a_it^i) $$
is an $F$-sub vector space of ${\rm Nuc}_r(S_f).$
\\ (ii) If $m = qk + r$ for some positive integers $q,r$ with $0 < r < k$, then $$L \oplus Lt^k \oplus Lt^{2k} \oplus \dots \oplus Lt^{qk}$$
is an $F$-sub vector space of $  {\rm Nuc}_r(S_f).$
\end{proposition}

\begin{proof}
 (i) Since $a_i \in F_k$, we have that $t^k \in {\rm Nuc}_r(S_f)$ by Theorem \ref{C5.1: Theorem 4}. The right nucleus is a subalgebra of $S_f$, which  implies that $t^{2k},\dots,t^{(q-1)k}, (t^k)^q = t^m = \sum_{i=0}^{m-1} a_it^i  \in {\rm Nuc}_r(S_f)$. Furthermore, we know that $L \subset  {\rm Nuc}_r(S_f)$, and so
$Lt^{jk} \subset  {\rm Nuc}_r(S_f)$ for any $j \in \lbrace 0,1,\dots,q \rbrace$. Therefore
   $$L \oplus Lt^k \oplus \dots \oplus Lt^{(q-1)k} \oplus L (\sum_{i=0}^{m-1} a_it^i)\subset  {\rm Nuc}_r(S_f)$$
    as claimed.
\\ (ii) We have $t^k \in {\rm Nuc}_r(S_f)$. Again since ${\rm Nuc}_r(S_f)$ is a subalgebra of $S_f$, this implies that $t^{2k},\dots t^{qk}, t^{(q+1)k},\dots \in {\rm Nuc}_r(S_f)$, hence the assertion follows by a similar argument to the proof of (i).
\end{proof}

Note that the powers $t^{qk}, t^{(q+1)k}, t^{(q+2)k},\dots$ of $t^k$ in Proposition \ref{C5.1: Proposition 3} (ii) lie in ${\rm Nuc}_r(S_f)$, but they need not be equal to polynomials in $t^k$, since $qk,(q+1)k,(q+2)k,\dots\phantom{}\geq m$.

\begin{lemma}\label{C5.1: Lemma 2}
Suppose that $f(t) = t^m - \sum\limits_{i=0}^{m-1} a_it^i \in R$ with $a_i \in F_1$ for all $i \in \lbrace 0,1,\dots,m-1 \rbrace$. If $\sigma\vert_{L}$ and $\delta\vert_{L}$ commute (i.e. $\sigma(\delta(c)) = \delta(\sigma(c))$ for all $c \in L$), then $\sigma\vert_{L}$ is an injective ring endomorphism of $L$ and $\delta\vert_{L}$ is a left $\sigma\vert_{L}$-derivation of $L$.
\end{lemma}

\begin{proof}
Let $c \in L$. Clearly, if $c = 0$, then $\sigma(c),\delta(c) \in L$, so we assume that $c \neq 0$. By Theorem \ref{C5.1: Theorem 1}
\begin{equation}\label{sigma_L delta_L Equation 1}
\sigma^m(c) a_k = \sum\limits_{j=k}^{m-1} a_j \Delta_{j,k}(c) - \Delta_{m,k}(c)
\end{equation}
for all $k \in \lbrace 0,1,\dots,m-1 \rbrace$. Applying $\sigma$ to both sides of \ref{sigma_L delta_L Equation 1} yields
\begin{equation}
\sigma^{m+1}(c) a_k = \sum\limits_{j=k}^{m-1} a_j \sigma(\Delta_{j,k}(c)) - \sigma(\Delta_{m,k}(c))
\end{equation}
as $a_i \in {\rm Fix}(\sigma)$ for all $i$. Now, since $\sigma(\delta(c)) = \delta(\sigma(c))$ we have
\begin{equation}
\sigma^m(\sigma(c)) a_k = \sum\limits_{j=k}^{m-1} a_j \Delta_{j,k}(\sigma(c)) - \Delta_{m,k}(\sigma(c)),
\end{equation}
i.e. $\sigma(c) \in L$. Therefore $\sigma\vert_{L}$ is a ring endomorphism of $L$, which is necessarily injective since $L$ is a division ring.\\
\\
Now we apply $\delta$ to both sides of Equation \ref{sigma_L delta_L Equation 1}, which yields
\begin{align*}
& \delta(\sigma^m(c) a_k) = \delta(\sum\limits_{j=k}^{m-1} a_j \Delta_{j,k}(c) - \Delta_{m,k}(c)) \\
\Rightarrow \,\, & \delta(\sigma^m(c)) a_k = \sum\limits_{j=k}^{m-1} a_j \delta(\Delta_{j,k}(c)) - \delta(\Delta_{m,k}(c)),\text{ (since $\delta(a_k) = 0$ for all $k$)}\\
\Rightarrow \,\, & \sigma^m(\delta(c)) a_k = \sum\limits_{j=k}^{m-1} a_j \Delta_{j,k}(\delta(c)) - \Delta_{m,k}(\delta(c)),
\end{align*}
for all $k \in \lbrace 0,1,\dots,m-1 \rbrace$, that is $\delta(c) \in L$.\\
\\
Finally, for any $a,b \in L$ $$\delta\vert_{L}(ab) = \sigma\vert_{L}(a)\delta\vert_{L}(b) + \delta\vert_{L}(a)b \in L.$$
Hence $\delta\vert_{L}$ is a left $\sigma\vert_{L}$-derivation of $L$.
\end{proof}

\begin{theorem}\label{C5.1: Theorem 5}
Suppose that $\sigma\vert_{L}$ and $\delta\vert_{L}$ commute, and let $f(t) = t^m - \sum_{i=0}^{m-1} a_it^i \in F[t] \subset R$. Write $\sigma=\sigma\vert_{L}$, and $\delta = \delta\vert_{L}$. Then
$$L_f :=  L[t;\sigma,\delta]/L[t;\sigma,\delta]f(t)$$
is a subalgebra of ${\rm Nuc}_r(S_f)$.
\end{theorem}

\begin{proof}
Since $\sigma$ is an injective endomorphism of $L$, and $\delta$ is a left $\sigma$-derivation of $L$, the skew polynomial ring  $L[t;\sigma,\delta]$ is well-defined, $f(t) \in F[t]\subset L[t;\sigma,\delta]$, and hence $ L_f= L[t;\sigma,\delta]/L[t;\sigma,\delta]f(t)$ is a subalgebra of $S_f$.\\
On the other hand, $L$ is contained in the right nucleus by definition, and since $a_i \in F$ for all $i$, we have that $t$ lies in the right nucleus by Theorem \ref{C5.1: Theorem 4}. Thus
$L \oplus Lt\oplus \dots \oplus Lt^{m-1}$ is a subspace of the right nucleus, as ${\rm Nuc}_r(S_f)$ is an algebra, i.e. it is closed under addition and multiplication. Moreover the elements of $L\oplus Lt\oplus \cdots \oplus Lt^{m-1}$ are precisely the elements of $L_f$, that is $L_f \subseteq {\rm Nuc}_r(S_f)$. Therefore $L_f $ is a subalgebra of ${\rm Nuc}_r(S_f)$.
\end{proof}

\section{The Special Case $R=D[t;\sigma]$}
Throughout this section we consider the Petit algebras associated with the polynomial ring $R=D[t;\sigma]$, and their nuclei. As a Corollary to Theorem \ref{C5.1: Theorem 1}, we obtain:

\begin{theorem}\label{C5.2: Theorem 1}
Let $f(t) = t^m - \sum\limits_{i=0}^{m-1}a_it^i \in R$ and $c \in D$. Then $c \in L$ if and only if $\sigma^m(c)a_k = a_k\sigma^k(c)$ for all $k \in \lbrace 0,1,\dots,m-1 \rbrace$.
\begin{proof}
For $\delta = 0$, $\Delta_{j,i} = 0$ for all non-negative integers $i,j$ with $i \neq j$, and $\Delta_{j,j} = \sigma^j$. We apply this to Theorem \ref{C5.1: Theorem 1}.
\end{proof}
\end{theorem}

From now on we denote the indices of the nonzero coefficients $a_i$ in $f(t) = t^m - \sum_{i=0}^{m-1} a_it^i \in D[t;\sigma]$ by $\lambda_1,\dots,\lambda_{r}$, $1 \leq r \leq m$. We denote the set of these indices by
 $$\Lambda_f = \{ \lambda_1,\lambda_2,\dots,\lambda_{r} \}\subset \lbrace 0,1,\dots,m-1 \rbrace.$$
  If it is clear from the context which $f$ is used, we simply write $\Lambda$.

\begin{corollary}\label{C5.2: Corollary 1}
Let $f(t) = t^m - \sum\limits_{i=0}^{m-1}a_it^i \in R$ and assume that $a_i \in C$ for all $i \in \lbrace 0,1, \dots, m-1 \rbrace$. Then $$L = \bigcap\limits_{\lambda_j \in \Lambda} {\rm Fix}(\sigma^{m-\lambda_j}).$$
\end{corollary}

In the case that $D$ is commutative (i.e. $D$ is a field and $D=C$) Petit states that $t$ lies in the right nucleus if and only if the coefficients of $f$ are all fixed by $\sigma$ \cite[(16)]{petit1966certains}. We extend this result to the case that $D$ is a (noncommutative) unital division ring:

\begin{theorem} \label{C5.2: Theorem 2}\cite[(16) {\text for $D$ commutative}]{petit1966certains}
Let $f(t) = t^m - \sum\limits_{i=0}^{m-1}a_it^i \in R.$ Then $a_i \in {\rm Fix}(\sigma)$ for all $i \in \lbrace 0,1,\dots,m-1 \rbrace$ if and only if $t \in {\rm Nuc}_r(S_f)$.
\end{theorem}

\begin{proof}
If $a_i \in {\rm Fix}(\sigma)$ for all $i$, then $t \in {\rm Nuc}_r(S_f)$ by Theorem \ref{C5.1: Theorem 3} with $\delta =0$.
Conversely, suppose that $t \in {\rm Nuc}_r(S_f)$, then $[t,t^{m-1},t]=0$ which implies that
$$(\sum_{i=0}^{m-1} a_it^i)t = t(\sum_{i=0}^{m-1} a_it^i),$$
i.e.,
$$(a_{m-1}t^{m-1} + \sum_{i=0}^{m-2}a_it^i)t = t(a_{m-1}t^{m-1} + \sum_{i=0}^{m-2} a_it^i).$$
This yields
$$a_{m-1}t^m  + \sum_{i=0}^{m-2}a_it^{i+1} = \sigma(a_{m-1})t^m  + \sum_{i=0}^{m-2} \sigma(a_i)t^{i+1}.$$ We note that $$t^m = (\sum\limits_{i=0}^{m-1} a_it^i)\,  {\rm mod}_rf.$$
Suppose that $k$ is the smallest element of $\lbrace 0,1,\dots,m-1 \rbrace$ such that $a_k \neq 0$, i.e. $a_i = 0$ for all $i<k$. Then $$\sum_{i=k}^{m-1} a_{m-1} a_it^i + \sum_{i=k}^{m-2} a_it^{i+1} = \sum_{i=k}^{m-1} \sigma(a_{m-1}) a_it^i + \sum_{i=k}^{m-2} \sigma(a_i)t^{i+1}.$$
Equating the $t^k$ terms gives
$a_{m-1}a_k = \sigma(a_{m-1})a_k \Rightarrow a_{m-1} \in {\rm Fix}(\sigma).$
Therefore
\begin{align*}
a_{m-1}\sum_{i=k}^{m-1} a_it^i + \sum_{i=k}^{m-2} a_it^{i+1} &= a_{m-1}\sum_{i=k}^{m-1} a_it^i + \sum_{i=k}^{m-2} \sigma(a_i)t^{i+1}\\ \Rightarrow \sum_{i=k}^{m-2}a_it^{i+1} &= \sum_{i=k}^{m-2} \sigma(a_i)t^{i+1}
\end{align*}
and we have $a_i \in {\rm Fix}(\sigma)$ for all $i\in \lbrace k, k+1,\dots, m-2 \rbrace$. Since $a_i =0 $ for $i < k$, and $a_{m-1} \in {\rm Fix}(\sigma)$, we have $a_i \in {\rm Fix}(\sigma)$ for all $i \in \lbrace 0,1,\dots,m-1 \rbrace.$
\end{proof}

The following result of Brown and Pumpl{\"u}n \cite[Proposition 2]{brown2018nonassociative} is an immediate Corollary to Theorem \ref{C5.2: Theorem 2}:

\begin{corollary}\label{C5.2: Corollary 2} \cite[Proposition 2]{brown2018nonassociative}
 \label{F_1[t]/F_1[t]f(t) is a subalgebra of the right nucleus}
 Let $f(t) \in F[t]\subset R$.
 Then $F[t]/(f(t))$ is a commutative subring of ${\rm Nuc}_r(S_f),$ and a subfield of degree $m$ if $f$ is irreducible in $F[t]$.
\end{corollary}

\begin{proof}
$S_f$ contains the commutative subring $F[t]/(f)$ which is isomorphic to the ring consisting of the
elements $\sum_{i=0}^{m-1}a_it^i$ with $a_i\in F$.
 In particular, the powers of $t$ are associative.
This implies that $t\in {\rm Nuc}_r(S_f)$ \cite[(5)]{petit1966certains}. Clearly $F\subset {\rm Nuc}_r(S_f)$, so
 $F\oplus Ft\oplus\dots\oplus Ft^{m-1}\subset  {\rm Nuc}_r(S_f)$, hence we obtain the assertion.
\end{proof}

\begin{corollary}\label{C5.2: Corollary 3}
 Suppose that $f(t)  \in F[t] \subset R$ is not right-invariant such that $f(t)\not =t$, and let $h=\hat{h}(t^n)$ be its minimal central left multiple. 
  If $f$ is irreducible and $[L : F]=n/k$, where $k$ is the number of irreducible factors of the minimal central left multiple $h$ of $f$, then ${\rm Nuc}_r(S_f)=L_f.$
  In particular, then $f$ is irreducible in $K[t;\sigma]$ if and only if $f$ is irreducible in $L[t;\sigma]$.
 \end{corollary}

\begin{proof}
We know that $ L_f=L[t;\sigma]/L[t;\sigma]f(t)$
is a subalgebra of ${\rm Nuc}_r(S_f)$ (Theorem \ref{C5.1: Theorem 5}). If $f$ is irreducible then ${\rm Nuc}_r(S_f)$ has degree $ms$ over $F$ by Theorem \ref{C2.2: Theorem 1}, therefore comparing the degrees of the field extensions we obtain the assertion.
\end{proof}

Below are two examples showing how we may compute the right nucleus of $S_f$ using the preceding theory as well as the Wedderburn theory of Chapter 2.\\
\\
\textbf{Example 1}\label{C5.2: Example 1}
Let $D=(-1,-1)_\mathbb{Q}$, $\sigma:D\longrightarrow D,$ $\sigma(d)=udu^{-1}$ with $u=i+1$ and $R=D[t;\sigma]$. Thus $\sigma$ has inner order one, ${\rm Fix}(\sigma)=\mathbb{Q}(i)$, and $C(R)=\mathbb{Q}[x]$ with $x=u^{-1}t$. Also since $\sigma^4=id$ and ${\rm Fix}(\sigma^2)={\rm Fix}(\sigma^3)={\rm Fix}(\sigma)=\mathbb{Q}(i)$,
every $f\in R$ is bounded. 
Let $f \in R$ be monic and irreducible of degree $m>1$ such that $(f,t)_r=1$. Let $h(t)=\hat{h}(u^{-1}t)$ be its minimal central left multiple, such that $\hat{h}(x)\not=x$. Then one of the following holds: 
\begin{itemize}
\item ${\rm Nuc}_r(S_f)\cong E_{\hat{h}}$ is a field extension of degree $2m$, $k=2$, and ${\rm deg}(h)=2 m$, or
\item ${\rm Nuc}_r(S_f)$ is a central division algebra
 over $E_{\hat{h}}$ of degree $2$,  $h$ is irreducible in $R$, ${\rm deg}(h)=m={\rm deg}(\hat{h})$ and
 $[{\rm Nuc}_r(S_f) :F] = 4m$ (Corollary \ref{C2.2: Corollary 2}), in which case $f$ must be right-invariant checking the dimensions.
\end{itemize}    
So assume that $f$ is not right-invariant, monic and irreducible, such that $(f,t)_r=1$. Then ${\rm Nuc}_r(S_f)\cong E_{\hat{h}}$ is a field extension of $\mathbb{Q}$ of degree $2m$. In particular,  let $f(t) = t^m - \sum_{i=0}^{m-1}a_it^i \in \mathbb{Q}(i)[t]\subset D[t;\sigma]$ then $\mathbb{Q}(i) \subset  L^{(\sigma,f)}$ (Proposition \ref{C5.1: Proposition 1}) and $t,t^2,\dots,t^{m-1}\in {\rm Nuc}_r(S_f)$.
Thus we obtain that
$$\mathbb{Q}(i)[t]/\mathbb{Q}(i)[t]f(t)\subset {\rm Nuc}_r(S_f)$$ and that
$${\rm Nuc}_r(S_f) =\mathbb{Q}(i)[t]/\mathbb{Q}(i)[t]f(t)$$
for irreducible $f\in \mathbb{Q}(i)[t]$, by comparing dimensions of the vector spaces. Let $f(t)=it^2+(k+1)t+j+k$, then $$f^*(t)=-4t^4+(4+4i)t^3-8it^2+(8-8i)t+8$$ \cite[Example 2.13]{gomez2013computing}. This means that
$$\hat{h}(x)=16x^4-16x^3+16x^2+16x+8\in \mathbb{Q}[x]$$ and  $$h(t)=t^4-(1+i)t^3+2it^2-(2-2i)t-2.$$
Scaling $f$ to make it monic does not change its minimal central left multiple, so w.l.o.g. assume $f(t)=t^2+i(k+1)t+k+ik$.
If $f$ is irreducible then the above result tells us that $${\rm Nuc}_r(S_f)\cong\mathbb{Q}[x]/(x^4-x^3+x^2+x+1/2).$$
\\
\textbf{Example 2}\label{C2.2: Example 2}
Let $D$ be a quaternion algebra over a field $F$ of characteristic not 2 and $\sigma\in{\rm Aut}_F(D)$ of order two. Then $K={\rm Fix}(\sigma)$ is a quadratic field extension in $D$.   We know that $\sigma(d)=udu^{-1}$ and $u^2\in F$. W.l.o.g. choose $u\in K^{\times}$. Then $C(R)=F[x]$ with $x=u^{-1}t$. Let $f \in R$ be monic and irreducible of degree $m>1$, such that $(f,t)_r=1$. Let $h(t)=\hat{h}(u^{-1}t)$ be its minimal central left multiple. Then either  
\begin{itemize}
\item ${\rm Nuc}_r(S_f)\cong E_{\hat{h}}$ is a field extension of degree $2m$, $k=2$, and ${\rm deg}(h)=2 m$, or
\item ${\rm Nuc}_r(S_f)$ is a central division algebra over $E_{\hat{h}}$ of degree $2$,  $h$ is irreducible in $R$, ${\rm deg}(h)=m={\rm deg}(\hat{h})$ and $[{\rm Nuc}_r(S_f) :F] = 4m$ (Corollary \ref{C2.2: Corollary 2}), in which case $f$ must be invariant checking the dimensions. 
\end{itemize}
So assume that $f$ is not right-invariant, monic and irreducible. Then ${\rm Nuc}_r(S_f)\cong E_{\hat{h}}$ is a field extension of degree $2m$. In particular, let $f(t) = t^m - \sum_{i=0}^{m-1}a_it^i \in {\rm  Fix}(\sigma)[t]\subset D[t;\sigma]$ then ${\rm Fix}(\sigma) \subset  L^{(\sigma,f)}$ (Proposition \ref{C5.1: Proposition 1}) and $t,t^2,\dots,t^{m-1}\in {\rm Nuc}_r(S_f)$.
Thus we obtain that
$${\rm  Fix}(\sigma)[t]/(f(t))\subset {\rm Nuc}_r(S_f),$$
with equality if $f$ is irreducible and not right-invariant.

\subsection{The Right Nucleus of $t^m - a \in D[t;\sigma]$} \label{RightNuc t^m-a}
In the particular case $f(t)=t^m-a \in R=D[t;\sigma]$ we obtain stronger results on the right nucleus than in the general case for any $f \in R$. Later we will see how this can be used to obtain irreducibility criteria for the polynomials of the form $t^m-a$ where $a \in F^{\times}$ where $F$ denotes the field $C \cap {\rm Fix}(\sigma)$.\\
\\
For $f(t) = t^m -a \in F[t] \subset D[t;\sigma]$ we get the reverse inclusion of Theorem \ref{C5.1: Theorem 5}, i.e. we can show that the right nucleus of $S_f$ is equal to the subalgebra $$L_f = L[t;\sigma]/L[t;\sigma]f(t):$$

\begin{theorem}\label{C5.2: Theorem 3}
Let $f(t) = t^m - a \in F[t] \subset R$.
Then $${\rm Nuc}_r(S_f) = L[t;\sigma]/L[t;\sigma]f(t).$$
\begin{proof}
If $f(t) \in F[t] \subset D[t;\sigma]$, then $$L[t;\sigma]/L[t;\sigma]f(t)$$ is a subalgebra of ${\rm Nuc}_r(S_f)$, by Theorem \ref{C5.1: Theorem 5} with $\delta = 0$.
\\
Now, we suppose that $g(t) = \sum\limits_{i=0}^{m-1} g_it^i \in {\rm Nuc}_r(S_f)$. Then $fg \in Rf$, i.e. there exists a polynomial $g^{\prime} \in D[t;\sigma]$ such that $fg = g^{\prime}f$. It is easy to see that ${\rm deg}(g) = {\rm deg}(g^{\prime})$, so we let $g^{\prime}(t) = \sum\limits_{j=0}^{m-1} g^{\prime}_jt^j$ for some $g^{\prime}_j \in D$. Then we have
\begin{align*}
(t^m-a)\sum\limits_{i=0}^{m-1} g_it^i &= (\sum\limits_{j=0}^{m-1} g^{\prime}_jt^j)(t^m-a) \\
\Rightarrow  \sum\limits_{i=0}^{m-1} \sigma^m(g_i)t^{m+i} - \sum\limits_{i=0}^{m-1} ag_it^i &= \sum\limits_{j=0}^{m-1} g^{\prime}_jt^{m+j} - \sum\limits_{j=0}^{m-1} g^{\prime}_jat^j.
\end{align*}
Comparing coefficients of $t^{m+j}$ for $j \in \lbrace 0,1,...,m-1 \rbrace$ yields $g^{\prime}_j = \sigma^m(g_j)$, and comparing coefficients of $t^j$ for $j \in \lbrace 0,1,\dots,m-1 \rbrace$ yields $g^{\prime}_ja = ag_j$. Combining these gives $$\sigma^m(g_j)a = ag_j$$ for all $j \in \lbrace 0,1,\dots,m-1 \rbrace$, that is $g_j \in L$ for all $j$. Hence $g(t) \in L[t;\sigma]/L[t;\sigma]f(t)$. The result follows immediately.
\end{proof}
\end{theorem}

\subsection{$L^{(\sigma,f)}$ for $f \in K[t;\sigma]$ with $K$ a Cyclic Field Extension of $F$}
For the remainder of the section (Section 5.2) suppose that $K/F$ is a Galois extension of $F$ of degree $n$ with cyclic Galois group ${\rm Gal}(K/F) = \langle \sigma \rangle$, let $R = K[t;\sigma]$ and let $$f(t) = t^m - \sum\limits_{i=0}^{m-1} a_it^i \in R.$$

\begin{proposition}\label{C5.2: Proposition 1}
 $$L = \lbrace u \in K : u \in {\rm Fix}(\sigma^{m-\lambda_j}) \text{ for all } \lambda_j \in \Lambda \rbrace = \bigcap\limits_{\lambda_j \in \Lambda} {\rm Fix}(\sigma^{m-\lambda_j}).$$
\begin{proof}
  Let $$u\in L = \lbrace u \in K : \sigma^m(u)a_i = a_i\sigma^i(u) \text{ for all } i \in \lbrace 0,1,\dots,m-1 \rbrace \rbrace.$$
 Then
\begin{align*}
& \, \sigma^m(u)a_i = a_i\sigma^i(u) \text{ for each } i\\
\Leftrightarrow & \, \sigma^m(u)a_i = \sigma^i(u)a_i \text{ for each } i\\
\Leftrightarrow & \, \sigma^{m-i}(u) = u \text{ for each } i \text{ such that } a_i \neq 0\\
\Leftrightarrow & \, \sigma^{m-\lambda_j}(u) = u \text{ for each } \lambda_j \in \Lambda.
\end{align*}
 This yields the assertion.
\end{proof}
\end{proposition}

\begin{corollary}\label{C5.2: Corollary 4}
If $a_{m-1} \neq 0$, then $L = F$.
\end{corollary}

\begin{proof}
Let $u\in L$, then $\sigma^m(u)a_{m-1} = a_{m-1}\sigma^i(u)$ yields $\sigma(u)=u$, hence $u\in F$.
This implies immediately that $L = F.$
\end{proof}

\begin{lemma} \label{C5.2: Lemma 1}
(i) Let $E = {\rm Fix}(\sigma^{u_1}) \cap {\rm Fix}(\sigma^{u_2})$ for some $u_1,u_2 \in \mathbb{Z}$ such that $1 \leq u_1,u_2 < n$. If ${\alpha}= {\rm gcd}(u_1,u_2,n)$, then ${\rm Gal}(K/E) = \langle \sigma^{\alpha} \rangle$ and $E = {\rm Fix}(\sigma^{\alpha})$.
\\ (ii) Let $E^\prime = {\rm Fix}(\sigma^{u_1}) \cap {\rm Fix}(\sigma^{u_2}) \cap \dots \cap {\rm Fix}(\sigma^{u_k})$ for some $u_1,u_2,\dots,u_k \in \mathbb{Z}$ such that $1 \leq u_1,u_2,\dots,u_k < n$. If $\alpha= {\rm gcd}(u_1,u_2,\dots,u_k,n)$, then ${\rm Gal}(K/E^\prime) = \langle \sigma^{\alpha} \rangle$, and $E^\prime = {\rm Fix}(\sigma^{\alpha})$.
\end{lemma}

\begin{proof}
(i) Let $v_1 = {\rm gcd}(u_1,n)$ and $v_2 = {\rm gcd}(u_2,n)$.
Now $${\rm ord}(\sigma^{u_1}) = \frac{n}{{\rm gcd}(u_1,n)}= \frac{n}{v_1}$$
 and
 $${\rm ord}(\sigma^{v_1}) = \frac{n}{{\rm gcd}(v_1,n)} = \frac{n}{v_1}.$$
 Therefore $\langle \sigma^{u_1} \rangle = \langle \sigma^{v_1} \rangle$, and similarly $\langle \sigma^{u_2} \rangle = \langle \sigma^{v_2} \rangle$ (note that also $E= {\rm Fix}(\sigma^{v_1}) \cap {\rm Fix}(\sigma^{v_2})$).
We have that ${\rm Gal}(K/E)$ is a cyclic subgroup of $\langle \sigma \rangle$. Thus there exists a smallest integer $\beta$, $1 \leq \beta <n$, such that ${\rm Gal}(K/E) = \langle \sigma^\beta \rangle$.

 $E$ is a subfield of ${\rm Fix}(\sigma^{v_1})$, which means that $\langle \sigma^{v_1} \rangle$ is a subgroup of $\langle \sigma^\beta \rangle$, and $\beta$ divides $v_1$. Similarly, we get that $\beta$ also divides $v_2$. Therefore $\beta$ is a common divisor of $v_1$ and $v_2$.
 By \cite[Chapter VI, Section 1, Corollary 1.4]{lang2004algebra} we know that ${\rm Gal}(K/E) = \langle \sigma^{v_1} , \sigma^{v_2} \rangle$, i.e. if $\tau \in {\rm Gal}(K/E)$, then $$\tau = \sigma^{\tau_1v_1}\sigma^{\tau_2v_2} = \sigma^{\tau_1v_1 + \tau_2v_2},$$ for some $\tau_1,\tau_2 \in \mathbb{Z}$.

Now $\langle \sigma^\beta \rangle = {\rm Gal}(K/E) = \langle \sigma^{v_1}, \sigma^{v_2} \rangle$, and so $\sigma^\beta \in \langle \sigma^{v_1},\sigma^{v_2} \rangle$. Hence $\sigma^\beta = \sigma^{xv_1 + yv_2}$ for some $x,y \in \mathbb{Z}$, and so we take $\beta=xv_1+yv_2 +zn$ for some $z \in \mathbb{Z}$.
 Let $\alpha = {\rm gcd}(v_1,v_2,n)$. Then $\alpha$ divides $v_1$, $v_2$ and $n$, by definition. Therefore $\alpha$ divides $x_0 v_1 + y_0 v_2 + z_0 n$ for any $x_0,y_0,z_0 \in \mathbb{Z}$; in particular $\alpha$ divides $xv_1 + yv_2 + zn = \beta$.
 In summary, we have shown that $\beta$ is a common divisor of $v_1$, $v_2$, and $n$, and $\alpha = {\rm gcd}(v_1,v_2,n)$ divides $\beta$; so we must have that $\alpha=\beta$. Moreover,
 $$\alpha = {\rm gcd}(v_1,v_2,n) = {\rm gcd}({\rm gcd}(u_1,n),{\rm gcd}(u_2,n),n) = {\rm gcd}(u_1,u_2,n).$$
 Hence we obtain that ${\rm Gal}(K/E)=\langle \sigma^{\alpha} \rangle,$ and $E= {\rm Fix}(\sigma^{\alpha})$ with $\alpha = {\rm gcd}(u_1,u_2,n)$ as claimed.
 \\ (ii) This follows  by induction from (i).
\end{proof}

\begin{theorem}\label{C5.2: Theorem 4}
If $\alpha = {\rm gcd}(m-\lambda_1,m-\lambda_2,\dots,m-\lambda_r,n)$, then
$$L = {\rm Fix}(\sigma^\alpha)$$
 and $[L:F]  = \alpha$.
 In particular, $L = F$  if and only if $\alpha=1$.
\begin{proof}
By Proposition \ref{C5.2: Proposition 1}, we have
$$L = \bigcap\limits_{\lambda_j \in \Lambda} {\rm Fix}(\sigma^{m-\lambda_j}) = {\rm Fix}(\sigma^{m-\lambda_1}) \cap {\rm Fix}(\sigma^{m-\lambda_2}) \cap \dots \cap {\rm Fix}(\sigma^{m-\lambda_r}).$$
Therefore it follows that
 $L = {\rm Fix}(\sigma^\alpha)$ where $\alpha = {\rm gcd}(m-\lambda_1,m-\lambda_2,\dots,m-\lambda_r,n)$
 by Lemma \ref{C5.2: Lemma 1}. Obviously $L=F$ if and only if ${\rm Fix}(\sigma^\alpha) = F$, if and only if $\langle \sigma^\alpha \rangle = \langle \sigma \rangle$, which is true if and only if $\sigma^\alpha$ has order $n$. Now ${\rm ord}(\sigma^\alpha) = \frac{n}{{\rm gcd}(n,\alpha)} = n$ if and only if ${\rm gcd}(n,\alpha) =1$.
 That is, $L = F$ if and only if ${\rm gcd}(n,\alpha) =1$, if and only if $\alpha=1$.
\end{proof}
\end{theorem}

\begin{theorem}\label{C5.2: Theorem 5}
(i) $L =K$ if and only if $m-\lambda_j$ is a multiple of $n$ for all $\lambda_j \in \Lambda$.
\\ (ii) Suppose that $n$ is prime. Then $L = F$ if and only if there exists $\lambda_j \in \Lambda$ such that $m-\lambda_j$ is not divisible by $n$.
\end{theorem}

\begin{proof}
We have $L ={\rm Fix}(\sigma^\alpha)$ where $\alpha = {\rm gcd}(m-\lambda_1,m-\lambda_2,\dots,m-\lambda_r,n)$ (Lemma \ref{C5.2: Lemma 1}).
\\ (i) $\alpha = {\rm gcd}(m-\lambda_1,m-\lambda_2,\dots,m-\lambda_r,n)=n$ if and only if $m-\lambda_j$ is a multiple of $n$ for all $\lambda_j \in \Lambda$. That is $L = {\rm Fix}(\sigma^n) = K$ if and only if $m-\lambda_j$ is a multiple of $n$ for all $\lambda_j \in \Lambda$.
\\ (ii) If $n$ is prime, then $\alpha \vert n$ $\Rightarrow$ $\alpha \in \{ 1, n\}$. The result now follows immediately from (i).
\end{proof}

Recall that $C(R) = F[x]$ with the identification $x=t^n$. We obtain the following well-known result (cf. \cite{jacobson2009finite}) rewritten in terms of the Petit algebra $S_f$:

\begin{corollary}\label{C5.2: Corollary 5}
Suppose that $f$ has the form $f(t) = ag(t^n)t^l$ for some scalar $a \in K$, some polynomial $g(t)=\hat{g}(t^n)$ such that $\hat{g} \in F[x]$ and some integer $l\geq 0$.
Then $S_f$ is associative and hence $f$ is right-invariant.
\begin{proof}
By Theorem \ref{C5.1: Theorem 5}, we know that $L[t;\sigma]/L[t;\sigma]f(t)$ is a subalgebra of the right nucleus.
Since  $m-\lambda_j$ is a multiple of $n$ for all $\lambda_j \in \Lambda$, it follows that $L = K$ by Theorem \ref{C5.2: Theorem 5} (i).
Thus $ K[t;\sigma]/K[t;\sigma]f(t)$ is a subalgebra of ${\rm Nuc}_r(S_f)$, which implies that $S_f = {\rm Nuc}_r(S_f)$. Hence $S_f$ is an associative algebra which is true if and only if $f$ in right-invariant in $R$.
\end{proof}
\end{corollary}

\subsection{The Case $n<m$.}
Let $K/F$ be a cyclic Galois extension of degree $n$ with Galois group ${\rm Gal}(K/F)=\langle \sigma \rangle$ and
$${\rm ord}(\sigma)=n < m = {\rm deg}(f).$$
 Then there exist integers $q,r$ such that $q \neq 0$, and $m=qn+r$ where $0 \leq r< n$. Moreover, $K = {\rm Fix}(\sigma^n)$ hence the coefficients of $f \in R$ always lie in ${\rm Fix}(\sigma^n)$. Applying this to Theorem \ref{C5.1: Theorem 4} yields the following:
\\
\\ 
(i) If $m=qn$, then
$$L \oplus Lt^n \oplus Lt^{2n} \oplus \dots \oplus Lt^{(q-1)n}$$
is an $F$-sub vector space of ${\rm Nuc}_r(S_f)$ of dimension $q[L:F]$ and
$$t^{qn}=t^m = \sum_{i=0}^{m-1} a_it^i \in {\rm Nuc}_r(S_f).$$
 (ii) If $m = qn + r$ for some positive integers $q,r$ with $0 < r < n$, then
$$L \oplus Lt^n \oplus Lt^{2n} \oplus \dots \oplus Lt^{qn} $$
is an $F$-sub vector space of ${\rm Nuc}_r(S_f)$ of dimension $(q+1)[L:F]$. If we assume that $n$ is either prime or ${\rm gcd}(m,n)=1$, $f$ is not right-invariant and $(f,t)_r=1$, as well as $[L:F]= n$, then $f$ is reducible.

\begin{theorem}
Suppose that $[L:F]=\rho$ where $n = b\rho$ for some $b \in \mathbb{N}$. Then $m=q\rho+r$ for some integers $q,r$ with $0\leq r < \rho$, and $f(t) = g(t^\rho)t^r$, where $g$ is a polynomial of degree $q$ in $K[t^\rho;\sigma^\rho]$.
\end{theorem}

\begin{proof}
By Theorem \ref{C5.2: Theorem 4}, we have that $L = {\rm Fix}(\sigma^\alpha)$, where $\alpha = {\rm gcd}(m- \lambda_1,m-\lambda_2,\dots,m-\lambda_r,n)$. Now $\alpha=\rho$ if and only if $m-\lambda_j$ is a multiple of $\rho$ for all $\lambda_j \in \Lambda$.  But $m-\lambda_j$ is equal to a multiple of $\rho$ if and only if $\lambda_j = r + \rho l$ for some integer $l$ such that $0 \leq l < q$ (since $m=q\rho+r$). Therefore we obtain $\Lambda \subset  \lbrace r, r+\rho,r+2\rho,\dots,r+(q-1)\rho \rbrace.$ Thus
\begin{align*}
f(t) &= t^{q\rho +r} - a_{(q-1)\rho +r}t^{(q-1)\rho +r} - \dots - a_{r+\rho}t^{r+\rho} - a_rt^r\\
&= [(t^\rho)^{q} - a_{(q-1)\rho+r}(t^\rho)^{(q-1)} - \dots - a_{r+\rho}t^\rho - a_r]t^r\\
&= g(t^\rho)t^r
\end{align*}
where $g$ has degree $q$ in $K[t^\rho;\sigma^\rho]$.
\end{proof}

\section{The Special Case $R=D[t;\delta]$}
Throughout this section we consider the Petit algebras associated with the polynomial ring $R=D[t;\delta]$, and their nuclei. As a Corollary to Theorem \ref{C5.1: Theorem 1} we obtain:

\begin{theorem}\label{C5.3: Theorem 1}
Let $f(t) = t^m - \sum\limits_{i=0}^{m-1}a_it^i \in R$ and $c \in D$. Then $c \in L$ if and only if $$ca_k = \sum\limits_{i=k}^{m-1} a_i \binom{i}{k} \delta^{i-k}(c) - \binom{m}{k}\delta^{m-k}(c)$$ for all $k \in \lbrace 0,1,\dots,m-1 \rbrace$.
\begin{proof}
Since here $\sigma = {\rm id}$, the map $\Delta_{j,i}$ is equal to $\binom{j}{i}\delta^{j-i}$ for any $i,j \in \mathbb{Z}$ such that $0 \leq i \leq j$. The result follows immediately by Theorem \ref{C5.1: Theorem 1}.
\end{proof}
\end{theorem}

\begin{corollary}\label{C5.3: Corollary 1}
Suppose that $D$ has prime characteristic $p$. Let $f(t) = t^p - \sum\limits_{i=0}^{p-1}a_it^i \in R$, and let $c \in D$. Then $c \in L$ if and only if
$$ca_k = \sum\limits_{i=k}^{p-1} a_i \binom{i}{k} \delta^{i-k}(c) $$ for all $k \in \lbrace 1, 2\dots,m-1 \rbrace$, and
$$ca_0 = \sum\limits_{i=0}^p a_i \delta^i(c)$$
where $a_p=-1$.
\end{corollary}
\begin{proof}
We note that $\binom{p}{k}$ is divisible by $p$ for all $k \in \{1,2,\dots,p-1\}$, and $\binom{p}{0} = 1$. The result follows immediately.
\end{proof}

We can also prove the converse to Theorem \ref{C5.1: Theorem 3}, for the special case $\sigma = {\rm id}$:

\begin{theorem}\label{C5.3: Theorem 2}
Let $f(t) = t^m - \sum_{i=0}^{m-1}a_it^i \in R = D[t;\delta].$ Then $a_i \in {\rm Const}(\delta)$ for all $i \in \lbrace 0,1,\dots,m-1 \rbrace$ if and only if $t \in {\rm Nuc}_r(S_f)$.
\end{theorem}

\begin{proof}
By Theorem \ref{C5.1: Theorem 3} with $\sigma = {\rm id}$, if $a_i \in {\rm Const}(\delta)$ for all $i \in \lbrace 0,1,\dots,m-1\rbrace$, then $t \in {\rm Nuc}_r(S_f)$. On the other hand, if $t \in {\rm Nuc}_r(S_f)$, then $[t,t^{m-1},t]=0$ which gives $$(\sum_{i=0}^{m-1} a_it^i)t = t(\sum_{i=0}^{m-1} a_it^i),$$ that is
\begin{equation}\label{Equation 1: t in N iff a_i in C(delta)}
a_{m-1}t^m + \sum\limits_{i=0}^{m-2}a_it^{i+1} = a_{m-1}t^m + \sum\limits_{i=0}^{m-2}a_it^i + \sum\limits_{i=0}^{m-1}\delta(a_i)t^i.
\end{equation}
We note that $t^m = (\sum\limits_{i=0}^{m-1} a_i t^i) \, {\rm mod}_rf$, however, we need not explicitly substitute this into Equation \ref{Equation 1: t in N iff a_i in C(delta)}, since only $a_{m-1}t^m$ appears on each side of \ref{Equation 1: t in N iff a_i in C(delta)}, hence all terms in $t^m$ sum to $0$. We are left with
\begin{equation}\label{Equation 2: t in N iff a_i in C(delta)}
\sum\limits_{i=0}^{m-1}\delta(a_i)t^i = 0.
\end{equation}
Since Equation \ref{Equation 2: t in N iff a_i in C(delta)} is an equality in $S_f$, we have that $\delta(f(t)):= \sum\limits_{i=0}^{m-1} \delta(a_i)t^i \in Rf$, i.e that $\delta(f) = gf$ for some $g(t) \in R$. Then $${\rm deg}(gf) = {\rm deg}(g) + {\rm deg}(f) = {\rm deg}(g) + m = {\rm deg}(\delta(f)) \leq m-1,$$ which forces $g = 0$, and $\delta(f)=gf = 0$. Therefore Equation \ref{Equation 2: t in N iff a_i in C(delta)} is also true in $R$, and so $\delta(a_i) = 0$ for each $i \in \lbrace 0,1,\dots, m-1 \rbrace$. The result follows immediately.
\end{proof}

\subsection{The Right Nucleus of $t^p-a \in D[t;\delta]$ when ${\rm Char}(D)=p > 0$}
For the remainder of the chapter we suppose that $D$ has prime characteristic $p$, and we consider the polynomial $f(t) = t^p-a \in D[t;\delta]$.\\
\\
In a similar fashion to the results of Section \ref{RightNuc t^m-a}, we obtain stronger results on the right nucleus of $S_f$ for the particular polynomial $f(t) = t^p-a$ under certain conditions on the constant coefficient $a \in D^{\times}$. In particular Theorem \ref{C5.3: Theorem 3} is analogous to Theorem \ref{C5.2: Theorem 3}. Again we will make use of the following results in the coming chapters to determine irreducibility criteria for the polynomial $f=t^p-a \in D[t;\delta]$. Recall that the bracket defined by $[b,c]=bc-cb$ for $b,c \in D$ is the commutator on $D$.


\begin{lemma}\label{C5.3: Lemma 1}
$$L = \{ c \in D: [a,c]=\delta^p(c) \}.$$
In particular, if $a \in C $, then $$L = \{ c \in D: \delta^p(c) = 0 \} = {\rm Const}(\delta^p).$$
\end{lemma}
\begin{proof}
Let $c \in L$. By definition, this is true if and only if $fc = c^{\prime}f$ for some $c^{\prime} \in D$, i.e. $(t^p-a)c = c^{\prime}(t^p-a)$. Computing both sides gives $$\sum\limits_{i=0}^p \binom{p}{i} \delta^{p-i}(c)t^i - ac = c^{\prime}t^p - c^{\prime}a.$$ Now, since $\binom{p}{i}$ is divisible by $p$ for all $i \neq 0,p$, this simplifies to $$ct^p + \delta^p(c) - ac = c^{\prime}t^p -c^{\prime}a.$$ This yields $c=c^{\prime}$ and $\delta^p(c)-ac=-ca$, i.e. $\delta^p(c)=[a,c]$. Therefore $L \subset \{ c \in D: [a,c]=\delta^p(c) \}$. Conversely, let $c \in \{ c \in D: [a,c]=\delta^p(c) \}$. Then $$fc = \sum\limits_{i=0}^p \binom{p}{i} \delta^{p-i}(c)t^i-ac = ct^p + \delta^p(c) - ac = ct^p-ca= cf.$$ 
Therefore $c \in L$, and $\{c \in D: [a,c]=\delta^p(c) \} \subset L$. The second part follows easily, since $[a,c]=0$ for all $c \in D$ when $a$ is in the center $C$ of $D$.
\end{proof}

\begin{theorem}\label{C5.3: Theorem 3}
If $a \in {\rm Const}(\delta)$, then $${\rm Nuc}_r(S_f) = L[t;\delta\vert_{L}]/L[t;\delta\vert_{L}]f(t).$$
\end{theorem}
\begin{proof}
First we note that $a \in L$, since $[a,a]=\delta^p(a) = 0$ (Lemma \ref{C5.2: Lemma 1}), i.e. $f(t) \in L[t;\delta\vert_{L}]$. Then $$L_f = L[t;\delta\vert_{L}]/L[t;\delta\vert_{L}]f(t)$$ is a well-defined Petit algebra. Now since $a \in {\rm Const}(\delta)$, the algebra $L_f$ is an associative subalgebra of ${\rm Nuc}_r(S_f)$. To complete the proof we show that the right nucleus ${\rm Nuc}_r(S_f)$ is contained in the set $$L \oplus L t \oplus \cdots \oplus L t^{p-1}.$$ Let $g(t) = \sum\limits_{i=0}^{p-1} g_it^i \in {\rm Nuc}_r(S_f)$. Then $fg = g^{\prime}f$ for some $g^{\prime} \in D[t;\delta]$. Assume w.l.o.g. that $g^{\prime}(t) = \sum\limits_{i=0}^{p-1}g_i^{\prime}t^i$, then we have
\begin{align*}
& (t^p-a)\sum\limits_{i=0}^{p-1} g_it^i = \sum\limits_{i=0}^{p-1}g_i^{\prime}t^i(t^p-a)\\
\Rightarrow & \sum\limits_{i=0}^{p-1}\sum\limits_{j=0}^p \binom{p}{j}\delta^{p-j}(g_i)t^{i+j} - \sum\limits_{i=0}^{p-1}ag_it^i = \sum\limits_{i=0}^{p-1}g_i^{\prime}t^{p+i} - \sum\limits_{i=0}^{p-1}g_i^{\prime} \sum\limits_{j=0}^i \binom{i}{j}\delta^{i-j}(a)t^j
\end{align*}
Since $\binom{p}{j}$ is divisible by $p$ for all $j \neq 0,p$, and $D$ has characteristic $p$, and $\delta(a)=0$ we reduce to the following:
\begin{align*}
& \sum\limits_{i=0}^{p-1}g_it^{p+i} + \sum\limits_{i=0}^{p-1} \delta^{p}(g_i)t^{i} - \sum\limits_{i=0}^{p-1}ag_it^i = \sum\limits_{i=0}^{p-1}g_i^{\prime}t^{p+i} - \sum\limits_{i=0}^{p-1}g_i^{\prime} a t^i.
\end{align*}
This yields $g_i = g_i^{\prime}$ and $[a,g_i] = \delta^p(g_i)$ for all $i = 0,1,\dots, p-1$, that is $g_i \in L$ for each $i = 0,1,\dots,p-1$. Hence $g(t) \in L \oplus L t \oplus \cdots \oplus L t^{p-1}$, and the result follows.
\end{proof}


\chapter{The Right Nucleus of $S_f$ for Low Degree Polynomials in $K[t;\sigma]$}

In this chapter we assume that $K$ is a cyclic Galois field extension of finite degree $n$ over $F$ with ${\rm Gal}(K/F)= \langle \sigma \rangle$.
We repeatedly use that  $ [{\rm Fix}(\sigma^\rho):F]  = {\rm gcd}(n,\rho).$ We now use the previous results to explore the structure of ${\rm Nuc}_r(S_f)$ for some polynomials of low degree in $K[t;\sigma]$.
 The same arguments then apply for those of higher degree as well but of course quickly become tedious. The results of this chapter also appear in \cite{owen2019eigenspaces}.

\setcounter{section}{1}
\subsection{${\rm deg}(f)=2$}

Let $f(t)=t^2-a_1t-a_0 \in K[t;\sigma]$. Then
$L^{(\sigma,f)} = \bigcap\limits_{\lambda_j \in \Lambda}{\rm Fix}(\sigma^{2-\lambda_j})$ by Proposition \ref{C5.2: Proposition 1}.

\begin{enumerate}
\item If $f(t)=t^2-a_0$ with $a_0 \in K^\times$, then  $L^{(\sigma,f)}   = {\rm Fix}(\sigma^2)$.
\item If $f(t)=t^2-a_1t-a_0$ with $a_1\in K^\times$, then  $L^{(\sigma,f)} = F $.
\end{enumerate}

Note that if $n = [K:F]$ is even, then $\sigma^2$ has order $\frac{n}{2}$ in ${\rm Gal}(K/F)$, which means that $F \neq {\rm Fix}(\sigma^2)$. If $n$ is odd, then ${\rm gcd}(n,2)=1$, therefore ${\rm Fix}(\sigma^2) = F$.\\
\\
If we additionally assume that $f(t) \in F[t]$, then we obtain:

\begin{proposition}
Let $f(t)=t^2-a_0 \in F[t] \subset K[t;\sigma]$, $a_0\not=0$ then
$${\rm Fix}(\sigma^2)[t;\sigma]/{\rm Fix}(\sigma^2)[t;\sigma]f(t)$$
 is a subalgebra of ${\rm Nuc}_r(S_f)$ of dimension $2[{\rm Fix}(\sigma^2):F]$ over $F$.
 In particular, if $n$ is odd, $f$ not right-invariant such that $(f,t)_r=1$, and  $\hat{h}(x)\not=x$, then $f$ is reducible.
\begin{proof}
By Theorem \ref{C5.1: Theorem 5},
$L^{(\sigma,f)}[t;\sigma]/L^{(\sigma,f)}[t;\sigma]f(t)$
 is a subalgebra of ${\rm Nuc}_r(S_f)$ and
 $L^{(\sigma,f)} = {\rm Fix}(\sigma^2)$ by (1), which yields the first assertion.
The second assertion follows from the fact that the right nucleus has dimension 2 over $F$ for irreducible right-invariant $f$ under our assumptions.
\end{proof}
\end{proposition}

\subsection{${\rm deg}(f)=3$} \label{Example m=3}

Let $f(t) \in K[t;\sigma]$ be monic of degree $3$. Then
 $L^{(\sigma,f)} = \bigcap\limits_{\lambda_j \in \Lambda}{\rm Fix}(\sigma^{3-\lambda_j}) $ by Proposition \ref{C5.2: Proposition 1}.

\begin{enumerate}
\item If $f(t) = t^3-a_0 \in K[t;\sigma]$, where $a_0 \in K^\times$, then  $L^{(\sigma,f)} =  {\rm Fix}(\sigma^3) $.
\item If $f(t) = t^3-a_1t \in K[t;\sigma]$, where $a_1 \in K^\times$, then  $L^{(\sigma,f)} =  {\rm Fix}(\sigma^2)$.
\item In all other cases,  we obtain $L^{(\sigma,f)} = F$.
\end{enumerate}

\begin{proposition}
(i) If $f(t)=t^3-a_0$ with $0\not=a_0 \in F$, then
$${\rm Fix}(\sigma^3)[t;\sigma]/{\rm Fix}(\sigma^3)[t;\sigma] f(t)$$
is a subalgebra of ${\rm Nuc}_r(S_f)$ of dimension $3[{\rm Fix}(\sigma^3):F]$ over $F$.
In particular, if $f$ is not right-invariant, $n$ is prime or not divisible by 3,  such that $(f,t)_r=1$, and  $\hat{h}(x)\not=x$ 
then $f$ is reducible.
\\ (ii) If $f(t)=t^3-a_1t $ with $0\not=a_1 \in F$, then
$${\rm Fix}(\sigma^2)[t;\sigma]/{\rm Fix}(\sigma^2)[t;\sigma]f(t)$$
is a subalgebra of ${\rm Nuc}_r(S_f)$ of dimension $3[{\rm Fix}(\sigma^2):F]$ over $F$.
\end{proposition}

\begin{proof}
By  Theorem \ref{C5.1: Theorem 5}, $L^{(\sigma,f)}[t;\sigma]/L^{(\sigma,f)}[t;\sigma]f(t)\subset {\rm Nuc}_r(S_f)$.
\\ (i)  If $f(t)=t^3-a_0 \in F[t]$ with $a_0 \neq 0$, then $L^{(\sigma,f)} = {\rm Fix}(\sigma^3)$  which proves the assertion looking at the dimensions.
 \\ (ii) If $f(t)=t^3-a_1t \in F[t]$ with $a_1 \neq 0$, then $L^{(\sigma,f)} = {\rm Fix}(\sigma^2)$.
\end{proof}

\subsection{${\rm deg}(f)=4$} \label{Example m=4}

Let $f(t) \in K[t;\sigma]$ be monic of degree $m=4$. Then
 $L^{(\sigma,f)} = \bigcap\limits_{\lambda_j \in \Lambda} {\rm Fix}(\sigma^{4-\lambda_j})  $
 by Proposition \ref{C5.2: Proposition 1}.

\begin{enumerate}
\item If $f(t)=t^4-a_0$ with $a_0 \in K^\times$ then  $L^{(\sigma,f)}   = {\rm Fix}(\sigma^4)$.
\item If $f(t)=t^4-a_1t$ with $a_1 \in K^\times$ then   $L^{(\sigma,f)}
= {\rm Fix}(\sigma^3) $.
\item If $f(t)=t^4-a_2t^2$ with $a_2 \in K^\times$ then  $L^{(\sigma,f)} =
{\rm Fix}(\sigma^2) $.
\item If $f(t)=t^4-a_2t^2-a_0$ with $a_0,a_2 \in K^\times$, then  $L^{(\sigma,f)}  = {\rm Fix}(\sigma^4) \cap {\rm Fix}(\sigma^2) = {\rm Fix}(\sigma^2)$.
\item In all other cases, $L^{(\sigma,f)} =  {\rm Fix}(\sigma) = F $.
\end{enumerate}

Observe that:
\begin{itemize}
\item If $n$ is odd then ${\rm Fix}(\sigma^4)={\rm Fix}(\sigma^2)=F$.
\item If $n \equiv 0 (mod \text{ } 4)$, then $[{\rm Fix}(\sigma^4):F]=4$.
\item If $n \equiv 2 (mod \text{ } 4)$, then $[{\rm Fix}(\sigma^4):F] =2$.
\item If $n \equiv 0 (mod \text{ } 3)$, then $[{\rm Fix}(\sigma^3):F]=3$.
\item If $n \equiv 1 \text{ or } 2 (mod \text{ } 3)$, then ${\rm Fix}(\sigma^3)=F$.
\item If $n \equiv 0 (mod \text{ } 2)$ then $[{\rm Fix}(\sigma^2):F] = 2$.
\end{itemize}

The above cases for $n$ are not mutually exclusive.

\begin{proposition}
(i) If $f(t)=t^4-a_0\in F[t]$ with $0\not=a_0 $, then the subalgebra
$${\rm Fix}(\sigma^4)[t;\sigma]/{\rm Fix}(\sigma^4)[t;\sigma]f(t)$$
of the right nucleus of $S_f$
has dimension $4[{\rm Fix}(\sigma^4):F]=4 {\rm gcd}(n,4)$ over $F$.
In particular:
\\ (a) If $f$ is irreducible, not right-invariant and either $n\not=2$ is prime or ${\rm gcd}(n,4)=1$, and $f(t)\not =t$,
then
$${\rm Nuc}_r(S_f) \cong {\rm Fix}(\sigma^4)[t;\sigma]/{\rm Fix}(\sigma^4)[t;\sigma]f(t).$$
\\ (b) If $f$ is not right-invariant and $n=2$, then $f$ is reducible.
\\ (ii) If $f(t)=t^4-a_1t\in F[t;\sigma]$ with $0\not= a_1$, then the subalgebra
$${\rm Fix}(\sigma^3)[t;\sigma]/{\rm Fix}(\sigma^3)[t;\sigma]f(t)$$
of the right nucleus of $S_f$ has dimension $4[{\rm Fix}(\sigma^3):F]=4 {\rm gcd}(n,3)$ over $F$.
\\ (iii) If $f(t)=t^4-a_2t-a_0\in F[t;\sigma]$ with $0\not= a_2 $, then the subalgebra
$${\rm Fix}(\sigma^2)[t;\sigma]/{\rm Fix}(\sigma^2)[t;\sigma]f(t)$$
of the right nucleus of $S_f$ has dimension $4[{\rm Fix}(\sigma^2):F]=4 {\rm gcd}(n,2)$ over $F$.
In particular:
\\ (a) If $f$ is irreducible, not right-invariant and either $n\not=2$ is prime or ${\rm gcd}(n,4)=1$, and $f(t)\not =t$,
then
$${\rm Nuc}_r(S_f) \cong {\rm Fix}(\sigma^2)[t;\sigma]/{\rm Fix}(\sigma^2)[t;\sigma]f(t).$$
\\ (b) If $f$ is not right-invariant and $n=2$, $f(t)\not =t$, then $f$ is reducible.
\end{proposition}

\begin{proof}
 $L^{(\sigma,f)}[t;\sigma]/L^{(\sigma,f)}[t;\sigma]f(t)\subset {\rm Nuc}_r(S_f)$ by Theorem \ref{C5.1: Theorem 5}.
\\ (i) Here $L^{(\sigma,f)} = {\rm Fix}(\sigma^4)$ by (1), and thus $$ {\rm Fix}(\sigma^4)[t;\sigma]/{\rm Fix}(\sigma^4)[t;\sigma]f(t)\subset {\rm Nuc}_r(S_f).$$
 (ii) We know $L^{(\sigma,f)} = {\rm Fix}(\sigma^3)$ by (2), and hence
$$ {\rm Fix}(\sigma^3)[t;\sigma]/{\rm Fix}(\sigma^3)[t;\sigma]f(t)\subset {\rm Nuc}_r(S_f).$$
 (iii) We have $L^{(\sigma,f)} = {\rm Fix}(\sigma^2)$ by (3), and so
$$ {\rm Fix}(\sigma^2)[t;\sigma]/{\rm Fix}(\sigma^2)[t;\sigma]f(t)\subset {\rm Nuc}_r(S_f).$$
\end{proof}

%
%

\section{Conclusion}\label{sec:conclusion}

Let $K/F$ be a cyclic Galois extension of degree $n$ with Galois group ${\rm Gal}(K/F)=\langle \sigma \rangle$.
We assume that $n$ is either prime or that ${\rm gcd}(m,n)=1$ for $m$ the degree of the polynomial $f$ we look at.
For certain types of skew polynomials $f(t) = t^m-\sum_{i=0}^{m-1}a_it^i\in K[t;\sigma]$ such that $\hat{h}(x)\not=x$ and $(f,t)_r=1$ which are not right-invariant, we can decide if they are reducible based on the following ``algorithm'' with output \fbox{TRUE} if $f$ is reducible and \fbox{STOP} if we cannot decide:
\begin{enumerate}
\item If $f\in F[t]$ then: if $f$ is not right-invariant and $f$ is reducible in $F[t]$, then $f$ is reducible in $K[t;\sigma]$ \fbox{TRUE}, else \fbox{STOP}.
If $f\not \in F[t]$ then go to (2).
\item Compute $L=L^{(\sigma,f)}={\rm Fix}(\sigma^d)$, where $d = {\rm gcd}(m-\lambda_1,m-\lambda_2,\dots,m-\lambda_r,n)$.
\\ If $[L:F]>m$, then $f$ is reducible \fbox{TRUE}.
\\ If $[L:F]\leq m$ then go to (3).
\item Find the smallest integer $\rho$, such that $a_i\in {\rm Fix}(\sigma^\rho)$ for all $i$, and such that ${\rm Fix}(\sigma^\rho)$
is a proper subfield of $K$.
\\ If ${\rm Fix}(\sigma^\rho)=L$ then $f$ is reducible  \fbox{TRUE}.
\\ If $m=q\rho$ and $[L:F]>\rho$, then $f$ is reducible  \fbox{TRUE}.
\\  If $m = q\rho + r$ with $0 < r < \rho$, and $[L:F]\geq \rho$ 
then $f$ is reducible \fbox{TRUE}.
\\ In all other cases, go to (4).
\item If all $a_i$ are not contained in a proper subfield of $K$, then we \emph{cannot decide} if $f$ is reducible \fbox{STOP}.
\end{enumerate}

Furthermore, if $f(t)  \in F[t]$ then we can use the fact that
$$ L^{(\sigma,f)}[t;\sigma]/L^{(\sigma,f)}[t;\sigma]f(t)$$
is a subalgebra of ${\rm Nuc}_r(S_f)$ to look for zero divisors in ${\rm Nuc}_r(S_f)$ in order to factor $f$.


\chapter{When is the Eigenring of $f$ a Central Simple Algebra over $F$?}
Let $K$ be a field of characteristic $0$ and $R=K[t;\delta]$ be the ring a differential polynomials with coefficients in $K$. In order to derive results on the structure of the left $R$-modules $R/Rf$, Amitsur studied spaces of linear differential operators via differential transformations \cite{amitsur1953noncommutative,amitsur1954differential,amitsur1955generic}. He observed that every central simple algebra $B$ over a field $F$ of characteristic $0$ that is split by an algebraically closed field extension $K$ of $F$, is isomorphic to the eigenring \footnote{Amitsur refers to the eigenring of $f$ as the invariant ring of $f$ in \cite{amitsur1954differential}} of some polynomial $f \in K[t;\delta]$,  for  a  suitable  derivation $\delta$ of $K$. This identification of a central simple algebra $B$ with a suitable differential polynomial $f \in K[t;\delta]$ he called A-polynomial also holds when $K$ has prime characteristic $p$ \cite[Section 10]{amitsur1953noncommutative}, \cite{pumpluen2018algebras}. Let $D$ be a central division algebra of degree $d$ over its center $C$, $\sigma$ an endomorphism of $D$ and $\delta$ a left $\sigma$-derivation of $D$. Our aim is to provide a partial answer to the following generalisation of Amitsur’s investigation: 
\begin{center}
“For which polynomials $f$ in a skew polynomial ring $D[t;\sigma,\delta]$ is the eigenring $\mathcal{E}(f)$ a central simple algebra over its subfield $F=C \cap {\rm Fix}(\sigma) \cap {\rm Const}(\delta)$?”
\end{center}

After a brief review of the Amitsur's theory in Section 7.1, we investigate three different setups, always assuming that $f$ has degree $m \geq 1$ and that the minimal central left multiple of $f$ is square-free. We look at generalised A-polynomials in $D[t;\sigma]$ in Section 7.2, where $\sigma$ is an automorphism of $D$ of finite inner order $n$ with $\sigma^n = \iota_u$ for some $u \in D^{\times}$. In Section 7.3, we study generalised A-polynomials in $D[t;\delta]$ where $C$ has characteristic $0$ and $\delta$ is the inner derivation $\delta_c$ in $D$ for some $c \in D$. Finally, in Section 7.4, we investigate A-polynomials in $D[t;\delta]$ where $C$ has prime characteristic $p$ and $\delta$ is an algebraic derivation of $D$ with minimum polynomial $g \in F[t]$ of degree $p^e$ such that $g(\delta)=\delta_c$ for some $c \in D$. The contents of this chapter will also appear in \cite{owen2021Apolynomials}. 

\section{Differential Transformations and Amitsur's A-polynomials}

All of the definitions and results in this section can be found in \cite{amitsur1954differential} unless stated otherwise, although some are attributed to Jacobson, and have appeared earlier in \cite{jacobson1937pseudo}.\\
\\
Let $K$ be a field of any characteristic, $\sigma$ be an automorphism of $K$, $\delta$ be a left $\sigma$-derivation of $K$, and let $F = {\rm Fix}(\sigma) \cap {\rm Const}(\delta)$.
Let $V$ be a vector space over $K$ of dimension $m$. A {\it pseudo-linear transformation} (p.l.t.) of $V$ is an additive map $T:V \longrightarrow V$ such that $$T(av) = \sigma(a)T(v) + \delta(a)v.$$
 The ring of transformations in $V$ generated by $T$ and the elements of $K$ is equal to the polynomial ring $K[T]$.
For the rest of this section we focus only on the case $\sigma = {\rm id}_K$, in which case $(K,\delta)$ is a differential field, and we call the p.l.t. $T$ a {\it differential transformation} (d.t.). The ring $K[T]$ of polynomials in $T$ is then a ring of left differential operators on $V$, and we write $V = (V,T)$ to make clear this association.\\
\\
Given a basis $(b_1, \dots ,b_m)$ of $V$ we can represent the d.t. $T$ by a matrix $B(T) = (a_{ij}) \in M_m(K)$, where the entries $a_{ij}$ are determined by the images of the basis vectors $b_i$ under $T$, i.e. $$T(b_i) = \sum\limits_{j=0}^m a_{ij}b_j.$$
If we write $B = B(T)$, then $$(V,T) \cong \frac{K[t;\delta]}{K[t;\delta]g_1(t)} \times \frac{K[t;\delta]}{K[t;\delta]g_2(t)} \times \cdots \times \frac{K[t;\delta]}{K[t;\delta]g_m(t)},$$ where $g_1,\dots,g_m \in K[t;\delta]$ are the invariant factors of the matrix  $B - tI_m$.\\
\\
Additionally, we assume from now on that $K$ has characteristic zero for the remainder of the section. In this case the matrix $(B - t I_m )$ has only one proper invariant factor \cite[p.~499]{jacobson1937pseudo} and $$(V,T) \cong \frac{K[t;\delta]}{K[t;\delta]g(t)}$$ for some $g \in K[t;\delta]$ of degree $m$. We call the polynomial $g(t)$ the {\it characteristic polynomial} of both the matrix $A$ and the differential transformation $T$, and we note that the characteristic polynomial of any differential transformation is unique up to similarity. On the other hand, for any differential polynomial $f \in R$ of degree $m$, write $V_f = K[t;\delta]/K[t;\delta]f(t).$ Then $V_f$ is a $K$-vector space of dimension $m$, and $f(t)$ is the characteristic polynomial of the differential transformation $T$ on $V_f$ defined by $p(t) \mapsto tp(t)$. Now let $V^{\prime}$ be a vector space over $K$ of dimension $m^{\prime}$ equipped with a differential transformation $T^{\prime}$. We define a differential transformation $T \times T^{\prime}$ on the tensor product space $V \otimes_K V^{\prime}$ by $$(T \times T^{\prime})(u) = (T \times T^{\prime})(\sum\limits_i v_i \otimes v_i^{\prime}) = \sum\limits_i T(v_i) \otimes v_i^{\prime} + \sum\limits_j v_j \otimes T^{\prime}(v_j^{\prime}),$$ for all $u = \sum\limits_i v_i \otimes  v_i^{\prime} \in V \otimes_K V^{\prime}$. Related to this `product' differential transformation, Amitsur also introduced the related notion for characteristic polynomials. The so-called resultant is defined in the following way:

\begin{definition}\label{resultant definition}
Let $f,g \in K[t;\delta]$ be the characteristic polynomials of the differential transformations $T$ and $T^{\prime}$ respectively. Then the {\it resultant} of $f$ and $g$, denoted by $f \times g$, is defined as the characteristic polynomial of the differential transformation $T \times T^{\prime}$ over $V \otimes_K V^{\prime}$.
\end{definition}

\begin{remark}
In the notation of Definition \ref{resultant definition}, with ${\rm deg}(f)=m$ and ${\rm deg}(g) = m^{\prime}$, the resultant $f \times g$ is a differential polynomial of degree $mm^{\prime}$ uniquely determined up to similarity.
\end{remark}

For $r \in \mathbb{N}$, we denote by $e_r(t)$ a fixed characteristic polynomial of the $r \times r$ zero matrix. 

\begin{definition}
Let $f \in R$ have degree $m$. Then $f$ is called an A-polynomial if there exists some polynomial $g \in R$, such that $f \times g \sim e_{mn}(t)$, where $n = {\rm deg}(g)$. 
\end{definition}

\begin{theorem}\label{C7.1: Theorem 1}
\begin{enumerate}
\item If $f \in K[t;\delta]$ is an A-polynomial, and $f \sim g$, then $g$ is also an A-polynomial.
\item If $f,g \in K[t;\delta]$ are A-polynomials, then the resultant $f \times g$ is an A-polynomial.
\item If $f,g \in K[t;\delta]$ are A-polynomials such that $f \sim g$, and $f \times p \sim g \times q$ for some $p,q \in K[t;\delta]$, then $p \sim q$.
\item If $f \in K[t;\delta]$ is an A-polynomial, then $f \sim g \times e_r(t)$ for some $r \in \mathbb{N}$ and some irreducible A-polynomial $g \in R$ uniquely determined up to similarity.
\item If $f \in K[t;\delta]$ is an A-polynomial and $f \sim g \times h$ or $f = gh$ for some $g,h \in K[t;\delta]$, then $g$ and $h$ are also A-polynomials.
\end{enumerate}
\end{theorem}

In fact, the set of all A-polynomials in $K[t;\delta]$, denoted by $\mathcal{A}(K,\delta)$, forms a semigroup under the product $\times$ (i.e. if $f,g \in \mathcal{A}(K,\delta)$, we take $f \times g$ to be their product in $\mathcal{A}(K,\delta)$) and the relation $\sim$ (similarity of polynomials in $K[t;\delta]$), with the polynomial $t$ acting as the identity in $\mathcal{A}(K,\delta)$.
In \cite{amitsur1953noncommutative} it is shown that the A-polynomials are precisely the polynomials in $K[t;\delta]$ such that $\mathcal{E}(f)$ is a central simple algebra over $F$ of degree $m = {\rm deg}(f)$, which is split by $K$. Below are the main results relating such A-polynomials to central simple algebras:

\begin{theorem}\label{C7.1: Theorem 2}
\begin{enumerate}
\item If $f,g \in K[t;\delta]$ are A-polynomials, then $h = f \times g$ is an A-polynomial, and $$\mathcal{E}(h) \cong \mathcal{E}(f) \otimes_F \mathcal{E}(g).$$
\item If $f \in K[t;\delta]$ is an A-polynomial of degree $m$, then $\mathcal{E}(f)$ is a central simple algebra over $F$ of degree $m$, which is split by $K$.
\item Every central simple algebra over $F$ of degree $m$ which is split by $K$, is isomorphic to $\mathcal{E}(f)$ for some A-polynomial $f \in K[t;\delta]$ of degree $m$.
\end{enumerate}
\end{theorem}

Suppose instead that we take $F$ to be a field of characteristic $p \neq 0$, with $K$ a finite purely inseparable field extension of exponent 1, i.e. $K^p \subseteq F \subset K$. Then the theory of differential transformations also holds in the setting. The related theory of A-polynomials and central simple algebras holds when the A-polynomial $f \in K[t;\delta]$ involved has degree $m < [K:F] = p^e$ for some $e \in \mathbb{N} \cup \{ \infty \}$ (with slight abuse of notation writing $e = \infty$ for $[K:F] = \infty$). In particular, we note that the results of Theorem \ref{C7.1: Theorem 2} hold in this setting, for A-polynomials $f$ of degree $m < [K:F]$, and for central simple algebras over $F$ of degree $m$ which are split by $K$, satisfying $m < [K:F]$ \cite[Section 10]{amitsur1954differential}. This is the most important case that we can say more on.

\section{Generalised A-polynomials}
For $f \in D[t;\sigma,\delta]$ we are interested in answering the question:
\begin{center}
``When is $\mathcal{E}(f)$ a central simple algebra over the field $F = C \cap {\rm Fix}(\sigma) \cap {\rm Const}(\delta)$?''
\end{center} 
To this end we define a generalised A-polynomial as follows:

\begin{definition}
Let $f \in D[t;\sigma,\delta]$. We call $f$ a {\it generalised A-polynomial} if $\mathcal{E}(f)$ is a central simple algebra over $F$.
\end{definition}

If $D$ is commutative (i.e. a field), then $d=1$ and without loss of generality we can take $\sigma = {\rm id}_D$ or $\delta = 0$. Then a polynomial $f$ in $D[t;\sigma]$ (resp. $D[t;\delta]$) is a generalised A-polynomial if $\mathcal{E}(f)$ is a central simple algebra over $F = {\rm Fix}(\sigma)$ (resp. $F={\rm Const}(\delta)$).\\
\\
For each $v \in D^{\times}$, we define a map $\Omega_v: D \longrightarrow D$ by $$\Omega_v(\alpha) = \sigma(v)\alpha v^{-1} + \delta(v)v^{-1}.$$
If $D$ is commutative, we note that $\delta = 0$ yields $\Omega_v(\alpha) = \sigma(v)\alpha v^{-1}$ and $\sigma = {\rm id}$ yields $\Omega_v(\alpha) = \alpha + \delta(v)v^{-1}$ for all $v \in D^{\times}$.\\
\\
The following is an easy consequence of the definition of similarity of skew polynomials in $R$, and will prove to be very useful:

\begin{lemma}\cite[Lemma 2 for $\sigma = {\rm id}$]{amitsur1954differential}\label{C7.2: Lemma 1}
Let $\alpha,\beta \in D$. Then $(t - \alpha) \sim (t - \beta)$ in $D[t;\sigma,\delta]$ if and only if  $\Omega_v(\alpha) = \beta$ for some $v \in D^{\times}$.
\end{lemma}
\begin{proof}
$(t-\alpha) \sim (t-\beta)$ is equivalent to the existence of $v,w \in D^{\times}$ such that $w(t-\alpha)=(t-\beta)v$ \cite[pg.~33]{jacobson1943theory}, i.e. there exists $v,w \in D^{\times}$ such that $$w(t-\alpha) = \sigma(v)t + \delta(v) - \beta v.$$
This is the case if and only if $w=\sigma(v)$ and $w\alpha = \sigma(v)\alpha = \beta v - \delta(v)$. The result follows immediately.
\end{proof} 

\section{Generalised A-polynomials in $D[t;\sigma]$}
Let $R$ denote the twisted polynomial ring $D[t;\sigma]$ with $D$ a central division algebra over $C$ of degree $d$ and $\sigma$ an automorphism of $D$. We assume that $\sigma$ has finite inner order $n$, with $\sigma^n = \iota_u$ for some $u \in D^{\times}$.\\
\\
Recall that $R$ has center $F[u^{-1}t^n]\cong F[x]$. For the remainder of this section we suppose that $f \in R$ is a monic polynomial of degree $m \geq 1$ such that $(f,t)_r=1$. Then $f$ has a unique minimal central left multiple $h(t)=\hat{h}(u^{-1}t^n)$ for some $\hat{h} \in F[x]$ of degree at most $dm$. Since $(f,t)_r=1$, the polynomial $h(t)$ is a bound of $f$.\\
\\
We aim to provide some necessary, and some sufficient conditions for $f \in R$ to be a generalised A-polynomial.\\
\\
First suppose that $$\hat{h}(x)= \hat{\pi}_1(x)^{e_1}\hat{\pi}_2(x)^{e_2}\cdots \hat{\pi}_z(x)^{e_z}$$ for some irreducible polynomials $\hat{\pi}_1,\hat{\pi}_2,\dots \hat{\pi}_z \in F[x]$ such that $\hat{\pi}_i \neq \hat{\pi}_j$ for $i \neq j$, and some exponents $e_1,e_2,\dots,e_z \geq 1$, and let $\pi_i(t)=\hat{\pi}_i(u^{-1}t^n)$ for each $i=1,2,\dots,z$.\\
\\
We assume that the decomposition of $\hat{h}$ is square-free, i.e. that $e_1=e_2=\cdots =e_z = 1$ and $$\hat{h}(x) = \hat{\pi}_1(x)\hat{\pi}_2(x)\cdots\hat{\pi}_z(x).$$ 

By Theorem \ref{C4: Theorem 1}, $\mathcal{E}(f)$ has center $E_{\hat{h}} = F[x]/(\hat{h}(x))$, and by the Chinese Remainder Theorem for commutative rings (see for example \cite[\textsection{5}]{cohn1963noncommutative}) $$E_{\hat{h}} \cong  E_{\hat{\pi}_1} \oplus E_{\hat{\pi}_2} \oplus \cdots \oplus E_{\hat{\pi}_z},$$ where $E_{\hat{\pi}_i} = F[x]/(\hat{\pi}_i(x))$ for each $i$. Hence $\mathcal{E}(f)$ has center $F$ if and only if $E_{\hat{\pi}_1} \oplus E_{\hat{\pi}_2} \oplus \cdots \oplus E_{\hat{\pi}_z}=F$, which is true if and only if $z=1$ and $E_{\hat{\pi}_1}=F$. Therefore we can assume $\hat{h}$ is irreducible in $F[x]$ without loss of generality, and apply the results of Chapter 2.\\
\\
In particular by Theorem \ref{C2.2: Theorem 2}, for $f \in R$ with $\hat{h} \in F[x]$ irreducible, we have that $f=f_1f_2\cdots f_l$ for $f_1,f_2,\dots,f_l \in R$ irreducible polynomials, which are all similar to each other (i.e. $f_i \sim f_j$ for all $i,j$). Moreover $$\mathcal{E}(f) \cong M_l(\mathcal{E}(f_i)),$$ and $\mathcal{E}(f_i)$ is a central division algebra over $E_{\hat{h}}$ of degree $s^{\prime}=\frac{dn}{k}$, ${\rm deg}(\hat{h})=\frac{dm}{ls^{\prime}}$, and $[\mathcal{E}(f_i):F]= \frac{dms^{\prime}}{l}$ where $k$ is the number of irreducible factors of $h$ in any complete factorisation in $R$. Hence, $\mathcal{E}(f)$ is a central simple algebra over $E_{\hat{h}}$ of degree $s=ls^{\prime}$ and $[\mathcal{E}(f):F]= dms$.\\
\\
After considering the above we are left almost immediately with the following result:

\begin{theorem}\label{C7.3: Theorem 1}
Suppose that $\hat{h}(x)$ is irreducible in $F[x]$. Then $f$ is a generalised A-polynomial in $R$ if and only if $\hat{h}(x) = x-a$ for some $a \in F$ if and only if $f$ right divides $u^{-1}t^n -a $ for some $a \in F$. In particular, if $f$ is a generalised A-polynomial, then $m \leq n$.
\begin{proof}
Suppose that $f$ is a generalised A-polynomial in $R$. By the discussion preceding this result, for $f$ to be a generalised A-polynomial it is necessary that $\hat{h}(x) = x - a$ for some $a \in F$. Conversely if $\hat{h}(x) = x - a \in F[x]$, then $E_{\hat{h}} = F[x]/(x - a) = F$. Hence $\mathcal{E}(f)$
is a central simple algebra over $F$ by Theorem \ref{C2.2: Theorem 2}, i.e. $f$ is a generalised A-polynomial. It
is easy to see that $\hat{h}(x) = x - a$ is equivalent to $f$ being a right divisor of $u^{-1}t^n - a$ by definition of the minimal central left multiple. Moreover, if $f$ right divides $u^{-1}t^n - a$, then ${\rm deg}(f) \leq n$.
\end{proof}
\end{theorem}

For $n$ prime we are able to provide a more concrete description of $f$:

\begin{theorem}\label{C7.3: Theorem 2}
Suppose that $\hat{h}$ is irreducible in $F[x]$. Suppose that $n$ is prime and not equal to $d$. Then $f$ is a generalised A-polynomial in $R$ if and only if one of the following holds:
\begin{enumerate}
\item There exists some $a \in F^{\times}$ such that $ua \neq \prod\limits_{j=1}^n\sigma^{n-j}(b)$ for every $b \in D$, and $$f(t)=t^n-ua.$$ In this case $f$ is an irreducible polynomial in $R$.
\item  $m \leq n$ and there exist some constants $c_1,c_2,\dots,c_{m-1},c_m,b \in D^{\times}$, with $c_m=1$, such that the product $u^{-1}\prod\limits_{j=0}^{n-1} \sigma^{n-j}(b)$ lies in $F^{\times}$, and $$f(t) = \prod\limits_{i=1}^m (t-\Omega_{c_i}(b)).$$ In this case $f$ is a reducible polynomial unless $m=1$.
\end{enumerate}
\end{theorem}

\begin{proof}
By Theorem \ref{C7.3: Theorem 1} $f$ is a generalised A-polynomial in $R$ if and only if $f$ right divides $u^{-1}t^n-a$ for some $a \in F^{\times}$.
So suppose that $f$ is a generalised A-polynomial in $R$, then there exists some $a \in F^{\times}$ and some  nonzero $g \in R$ such that
\begin{equation}\label{f right divides u^-1t^n-a}
u^{-1}t^n-a = gf
\end{equation}
In the notation of Theorem \ref{C2.2: Theorem 2}, $ldn = ks$ and since $f$ is a generalised A-polynomial ${\rm deg}(\hat{h}) = \frac{dm}{s}=1$, i.e. $dm=s$. Combining these yields $\frac{n}{k} = \frac{m}{l} \in \mathbb{N}$. That is $k$ must divide $n$, and so we must have that $k=1$ or $k=n$ as $n$ is prime. We analyse the cases $k=1$ and $k=n$ separately.\\
\\
First suppose that $k=1$, then $h(t)$ is irreducible. Therefore Equation (\ref{f right divides u^-1t^n-a}) becomes
\begin{equation}
u^{-1}t^n-a = gf(t)
\end{equation}
for some $a \in F^{\times}$ and some $g \in D^{\times}$. This yields $g=u^{-1}$ and $f(t) = t^n-ua$ for some $a \in F^{\times}$. Suppose that $f$ were reducible, then $f$ would be the product of $n$ linear factors as $n$ is prime, hence $f$ is irreducible if and only if $ua \neq \prod\limits_{j=1}^n \sigma^{n-j}(b)$ for any $b \in D$, by \cite[Corollary 3.4]{brown2018}. \\
\\
On the other hand, if $k=n$, then $h(t)$ is equal to a product of $n$ linear factors in $R$, all of which are mutually similar to one another. Also, since $\frac{n}{k}=\frac{m}{l}$ and $n=k$, we have $m=l \leq n$. Hence $f$ is the product of $m \leq n$ linear factors in $R$, all of which are similar. \\
\\
So there exist constants $b_1,b_2,\dots,b_m \in D^{\times}$ such that $(t-b_i) \sim (t-b_j)$ for all $i,j \in \{1,2,\dots,m\}$, and $$f(t) = \prod\limits_{i=1}^m (t-b_i).$$ In particular $(t-b_i) \sim (t-b_m)$ for all $i\neq m$, which is true if and only if there exist constants $c_1,c_2,\dots,c_{m-1},c_m \in D^{\times}$ such that $b_i = \Omega_{c_i}(b_m)$ for all $i$ by Lemma \ref{C7.2: Lemma 1}. Hence setting $b=b_m$ and $c_m=1$ yields $$f(t) = \prod\limits_{i=1}^m (t-\Omega_{c_i}(b)).$$ Finally, we note that $(t-b)\vert_r (t^n - ua)$ for some $a \in F^{\times}$ if and only if $u^{-1}\prod\limits_{j=0}^{n-1} \sigma^{n-j}(b) = a \in F^{\times}$, by \cite[Corollary 3.4]{brown2018}.
\end{proof}

So there are no generalised A-polynomials in $R$ of degree greater than $n$ under the assumption that $\hat{h}$ is irreducible in $F[x]$.\\
If $e_i > 1$ for at least one $i$, then it is not clear to the author when $\mathcal{E}(f)$ is a central simple algebra over the field $F$.

\subsection{Generalised A-polynomials in $K[t;\sigma]$ and $\mathbb{F}_{q^n}[t;\sigma]$}

Throughout this section we suppose that $R=K[t;\sigma]$ with $K$ a field, and that $\sigma$ an automorphism of $K$ of finite order $n$ with fixed field $F$. Now the center of $R$ is $F[t^n] \cong F[x]$. Let $f \in R$ be of degree $m \geq 1$ and satisfy $(f,t)_r=1$, and suppose that $f$ has minimal central left multiple $h(t)=\hat{h}(t^n)$ for some irreducible polynomial $\hat{h} \in F[x]$.

\begin{theorem}\label{C7.3: Theorem 3}
 $f$ is a generalised A-polynomial in $R$ if and only if $\hat{h}(x) = x - a$ for some $a \in F[x]$ if and only if $f$ right divides $t^n-a$ in $R$. In particular, if $f$ is a generalised A-polynomial in $R$, then $m\leq n$.
\end{theorem}

This follows from Theorem \ref{C7.3: Theorem 2}. In the case that $n$ is prime, we obtain the following as an immediate corollary to both Theorem \ref{C7.3: Theorem 2} and Theorem \ref{C7.3: Theorem 3}:

\begin{theorem}\label{C7.3: Theorem 4}
Let $n$ be prime. Then $f$ is a generalised A-polynomial in $R$ if and only if one of the following holds:
\begin{enumerate}
\item There exists some $a \in F^{\times}$ such that $a \neq N_{K/F}(b)$ for any $b \in K$, and $$f(t)=t^n-a.$$ In this case $f$ is an irreducible polynomial in $R$.
\item  $m \leq n$ and there exist some constants $c_1,c_2,\dots,c_{m-1},c_m,b \in K^{\times}$ with $c_m=1$, such that $$f(t) = \prod\limits_{i=1}^m (t-\Omega_{c_i}(b)).$$
\end{enumerate}
\end{theorem}

\begin{proof}
The proof is identical to the proof of Theorem \ref{C7.3: Theorem 2} with $d=u=1$. We also dropped the condition that $\prod\limits_{j=0}^{n-1}\sigma^j(b)$ lie in $F^{\times}$ since this product is equal to the field norm $N_{K/F}(b)$ which is always in $F^{\times}$ for $b \in K^{\times}$.
\end{proof}

In particular, let $K = \mathbb{F}_{q^n}$, where $q=p^e$ for some prime $p$ and exponent $e \geq 1$, and where $\sigma:K\longrightarrow K$, $a \mapsto a^q$, is the Frobenius automorphism of order $n$, with fixed field $F=\mathbb{F}_q$. In this case the only central division algebra over the field $\mathbb{F}_q$ is $\mathbb{F}_q$ itself. The following result is therefore an easy consequence of Theorem \ref{C7.3: Theorem 3}:

\begin{theorem}\label{C7.3: Theorem 5}
Suppose that $f \in \mathbb{F}_{q^n}[t;\sigma]$ satisfies $(f,t)_r=1$, and has minimal central left multiple $h(t) = \hat{h}(t^n)$ for some irreducible polynomial $\hat{h} \in \mathbb{F}_q[x]$. Then $f$ is an A-polynomial if and only if $m \leq n$ and there exist some constants $c_1,c_2,\dots,c_{m-1},c_m,b \in K^{\times}$ with $c_m=1$, such that $$f(t) = \prod\limits_{i=1}^m (t-\Omega_{c_i}(b)).$$ In particular, $f$ is a reducible polynomial in $\mathbb{F}_{q^n}[t;\sigma]$ unless $m=1$.

\end{theorem}
\begin{proof}
First we note that $E_{\hat{h}}= \frac{\mathbb{F}_q[x]}{(\hat{h}(x))}$ is a field extension of $\mathbb{F}_q$ of finite degree, hence it is also a finite field. In the notation of Theorem \ref{C2.2'': Theorem 2}, $\mathcal{E}(f) \cong M_l(E_{\hat{h}})$ since the only central division algebra over the finite field $E_{\hat{h}}$ is $E_{\hat{h}}$ itself. Hence $\mathcal{E}(f)$ is a central simple algebra over $E_{\hat{h}}$ of degree $s=l$. By Theorem \ref{C2.2'': Theorem 2}, ${\rm deg}(h) =\frac{mn}{s}= \frac{mn}{l}$. On the other hand, we know that ${\rm deg}(h) = k{\rm deg}(f_i)$ where $f_i$ is any irreducible divisor of $f$ in $R$, and ${\rm deg}(f_i) = \frac{m}{l}$. Combining the expressions for ${\rm deg}(h)$ yields $n = k$.\\
The rest of the proof can be taken verbatim to be the proof of Theorem \ref{C7.3: Theorem 2} from the point that we examine the $n=k$ case. We note that the part of the proof of Theorem \ref{C7.3: Theorem 2} that we invoke here does not rely on $n$ being a prime.
\end{proof}

\section{Generalised A-polynomials in $D[t;\delta]$}
This section is dedicated to finding necessary and sufficient conditions for certain polynomials $f \in D[t;\delta]$. To do this, we focus on the cases ${\rm Char}(D)=0$ and ${\rm Char}(D) = p \neq 0$, separately. Again, it can be seen that all of the results in this section follow from a similar argument to those use throughout the previous sections of this chapter.

\subsection{${\rm Char}(D)=0$}
In this section we impose that $D$ has characteristic $0$, and that $\delta$ is the inner derivation defined by some element $c \in D^{\times}$, that is $\delta(a) = [c,a] = ca-ac$ for all $a \in D$. \\
\\
Recall that $R$ has center $F[t-c] \cong F[x]$ where $F = C \cap {\rm Const}(\delta)$, and that all polynomials $f \in R = D[t;\delta]$ have a unique minimal central left multiple $h(t) = \hat{h}(t-c)$ for some $\hat{h}\in F[x]$, which is equal to a bound of $f$.

\begin{remark}
As in the search for generalised A-polynomials in a twisted polynomial ring $D[t;\sigma]$, we restrict ourselves to the case $\hat{h}$ is square-free in $F[x]$, otherwise the author was unable to determine when the eigenring of $f \in R$ is indeed a central simple algebra over $F$. Then by a similar argument to Section 3 (pg.~74), we lose no generality in the search for A-polynomials in $R$, if we take $\hat{h}$ to be an irreducible polynomial in $F[x]$.
\end{remark}

By Theorem \ref{C3.1: Theorem 2}, for any $f \in R$ of degree $m \geq 1$ with $\hat{h} \in F[x]$ irreducible, we have that $f=f_1f_2\cdots f_l$ for $f_1,\dots,f_l$ irreducible polynomials in $R$, which are all mutually similar to each other, $$\mathcal{E}(f) \cong M_l(\mathcal{E}(f_i)),$$ and $\mathcal{E}(f_i)$ is a central division algebra over $E_{\hat{h}}$ of degree $s^{\prime}=\frac{d}{k}$, where $k$ is the number of irreducible factors of $h$ in $R$. Moreover, ${\rm deg}(\hat{h})=\frac{dm}{ls^{\prime}}$, and $[\mathcal{E}(f_i):F]= \frac{dms^{\prime}}{l}$. Hence, $\mathcal{E}(f)$ is a central simple algebra over $E_{\hat{h}}$ of degree $s=ls^{\prime}$ and $[\mathcal{E}(f):F]= dms$.\\
\\
This yields the following result almost immediately:
\begin{theorem}\label{C7.3: Theorem 6}
Suppose that $f \in R$ has degree $m$, and that $f$ has minimal central left multiple $h(t)=\hat{h}(t-c)$ for some irreducible polynomial $\hat{h} \in F[x]$. Then $f$ is a generalised A-polynomial in $R$ if and only if $f(t)=t-(c+a)$ for some $a \in F$.
\end{theorem}
\begin{proof}
Suppose that $f$ is a generalised A-polynomial in $R$, that is $\mathcal{E}(f)$ is a central simple algebra over $F$ of degree $dm$. This is true if and only if $E_{\hat{h}} = F$, i.e. if and only if ${\rm deg}(\hat{h})= \frac{md}{s} = 1$. Having $\hat{h}$ of degree one is equivalent to having $h(t)=t-(c+a)$ for some $a \in F$, and by definition $h(t) = g(t)f(t)$ for some $g \in R$.
\end{proof}

\begin{remark}
The above result shows that there are in fact no generalised A-polynomials in $R$ with a square-free minimal central left multiple, other than the central element $t-c$ up to a possible shift by some element of $F$ and a possible scalar multiplication by some element of $D^{\times}$.
\end{remark}

\subsection{${\rm Char}(D) = p > 0$}
From now on let $R = D[t; \delta]$ where $D$ is a central division algebra of degree $d$ over $C$. Assume that $C$ has prime characteristic $p$, and that $\delta$ is an algebraic derivation of $D$ with minimum polynomial $g(t) = t^{p^e} + \gamma_1t^{p^{e-1}} + \cdots + \gamma_et \in F[t]$, such that $g(\delta)(a) = [c, a] = ca-ac$ for some nonzero $c \in D$ and for all $a \in D$. Here, $F = C \cap {\rm Const}(\delta)$ ($D = K$ is a field is included here as special case). Then $R$ has center $F[g(t)-c] \cong F[x]$. For every $f \in R$, the minimal central left multiple of $f$ in $R$ is the unique polynomial of minimal degree
$h \in C(R) = F[x]$ such that $h = gf$ for some $g \in R$, and such that $h(t) = \hat{h}(g(t)-c)$ for some monic $\hat{h} \in F[x]$. All $f \in R = D[t; \delta]$ have a unique minimal central left multiple, which is a bound of $f$. Again we can restrict our investigation to the case $\hat{h}$ is square-free in $F[x]$, and note that it is necessary that $\hat{h}$ be irreducible in $F[x]$ for $f$ to be a generalised A-polynomial in $R$.\\
\\
By Theorem \ref{C3.2: Theorem 4}, for any $f \in R$ of degree $m \geq 1$ with $\hat{h} \in F[x]$ irreducible, we have that $f=f_1f_2\cdots f_l$ for $f_1,\dots,f_l$ irreducible polynomials in $R$, which are all mutually similar to each other, $$\mathcal{E}(f) \cong M_l(\mathcal{E}(f_i)),$$ and $\mathcal{E}(f_i)$ is a central division algebra over $E_{\hat{h}}$ of degree $s^{\prime}=\frac{dp^e}{k}$, where $k$ is the number of irreducible factors of $h$ in $R$. Moreover, ${\rm deg}(\hat{h})=\frac{dm}{ls^{\prime}}$, and $[\mathcal{E}(f_i):F]= \frac{dms^{\prime}}{l}$. Hence, $\mathcal{E}(f)$ is a central simple algebra over $E_{\hat{h}}$ of degree $s=ls^{\prime}$ and $[\mathcal{E}(f):F]= dms$. We obtain the following:

\begin{theorem}\label{C7.3: Theorem 7}
$f$ is a generalised A-polynomial in $R$ if and only if $f$ right divides $g(t)-(b+c)$ for some $b \in F$. In particular, ${\rm deg}(f) \leq p^e$.
\end{theorem}
\begin{proof}
Suppose that $f$ is a generalised A-polynomial in $R$, that is $\mathcal{E}(f)$ is a central simple algebra over $F$ of degree $dm$. This is true if and only if $E_{\hat{h}} = F$, i.e. if and only if ${\rm deg}(\hat{h})= \frac{md}{s} = 1$. Having $\hat{h}$ of degree one is equivalent to having $h(t)=g(t)-(b+c)$ for some $b \in F$, and by definition $h(t) = r(t)f(t)$ for some $r \in R$.
\end{proof}

In $D[t;\delta]$ 
\begin{equation}\label{delta identity for right dividing lin factor 1}
(t-\alpha)^p = t^p - V_p(\alpha),\; V_p(\alpha) = \alpha^p + \delta^{p-1}(\alpha) + \nabla_\alpha
\end{equation} for all $\alpha \in D$, where $\nabla_\alpha$ is a sum of commutators of $\alpha,\delta(\alpha),\delta^2(\alpha),\dots,\delta^{p-2}(\alpha)$ \cite[pg.~17-18]{jacobson2009finite}. In particular, if $D$ is commutative, then  $\nabla_\alpha = 0$ and 
\begin{equation}
V_p(\alpha) = \alpha^p + \delta^{p-1}(\alpha)
\end{equation}
for all $\alpha \in D$. Using the identities $t^p = (t-\alpha)^p + V_p(\alpha)$ and $t = (t-\alpha) + \alpha$ for all $\alpha \in D$, we arrive at:

\begin{lemma}\cite[Proposition 1.3.25 ($e=1$)]{jacobson2009finite}\label{C7.3: Lemma 1}
Let $f(t)=t^p-\gamma_1t-\gamma \in D[t;\delta]$ and $\alpha \in D$. Then $(t-\alpha)\vert_r f(t)$ if and only if $V_p(\alpha)-\gamma_1\alpha-\gamma = 0.$
\end{lemma}

The identity (\ref{delta identity for right dividing lin factor 1}) can be iterated to obtain 
\begin{equation}
(t-\alpha)^{p^i} = t^{p^i}-V_{p^i}(\alpha)
\end{equation}
for all $\alpha \in D$ and all integer exponents $i \geq 1$, where $$V_{p^i}(\alpha) = V_p^i(\alpha) = \underbrace{(V_p \circ V_p \circ \cdots \circ V_p)}_{\text{i terms}}(\alpha).$$ In a similar fashion to Lemma \ref{C7.3: Lemma 1} the identities $t^{p^i}= (t-\alpha)^{p^i}+V_{p^i}(\alpha)$ and $t=(t-\alpha)+\alpha$ may be used to obtain:

\begin{lemma}\cite[Proposition 1.3.25]{jacobson2009finite}\label{C7.3: Lemma 2}
Let $f(t)= t^{p^e} + \gamma_1 t^{p^{e-1}} + \cdots + \gamma_e t - \gamma \in D[t;\delta]$ and $\alpha \in D$. Then $(t-\alpha) \vert_r f(t)$ if and only if $$V_{p^e}(\alpha) + \gamma_1 V_{p^{e-1}} + \cdots + \gamma_e \alpha - \gamma = 0.$$
\end{lemma}

If $e=1$ (i.e. $\delta$ is an algebraic derivation of $D$ of degree $p$), we can determine necessary and sufficient conditions for $f$ to be an A-polynomial in $R$:

\begin{theorem}\label{C7.3: Theorem 8}
Suppose that $\delta$ has minimum polynomial $g(t)=t^p-\gamma_1t$ and that $\hat{h}$ is irreducible in $F[x]$. Then $f$ is a generalised A-polynomial in $R$ if and only if one of the following holds:
\begin{enumerate}
\item $f(t) = h(t) = t^p - \gamma_1t - (\gamma + c)$ for some $\gamma \in F$, and $$V_p(\alpha) - \gamma_1\alpha - (\gamma + c) \neq 0$$ for all $\alpha \in D$. In this case $f$ is irreducible in $R$.
\item $h(t) = t^p -\gamma_1t - (\gamma + c)$ for some $\gamma \in F$, $m \leq p$ and $$f(t) = \prod\limits_{i=1}^m (t - \Omega_{c_i}(\alpha))$$ for some $c_1,c_2,\dots,c_m,\alpha \in D^{\times}$ with $c_m = 1$ (w.l.o.g.) such that $$V_p(\alpha) - \gamma_1\alpha - (\gamma + c) = 0.$$
In particular, $f$ is a reducible polynomial in $R$ unless $m=1$. 
\end{enumerate}
\end{theorem}
\begin{proof}
By Theorem \ref{C7.3: Theorem 7}, $f$ is a generalised A-polynomial in $R$ if and only if $f$ right divides $t^p - \gamma_1t - (\gamma + c)$ for some $\gamma \in F$. So suppose that $f$ is a generalised A-polynomial in $R$, then
there exists some $\gamma \in F$ and some nonzero $f^{\prime} \in R$ such that $t^p-\gamma_1t-(\gamma+c) = f^{\prime}f.$ In the notation of Theorem \ref{C3.2: Theorem 4}, $ldp = ks$ and since $f$ is a generalised A-polynomial, ${\rm deg}(\hat{h}) =
\frac{dm}{s} = 1$, i.e. $dm = s$. Combining these yields $\frac{p}{k} = \frac{m}{l} \in \mathbb{N}$. That is $k$ must divide $p$, and so we must have that $k = 1$ or $k = p$ as $p$ is prime.\\ 
First suppose that $k = 1$, then $h(t)$ is irreducible in $R$ and $t^p - \gamma_1t - (\gamma + c) = f^{\prime}f$ yields $f^{\prime} = 1$, i.e. $f(t) = t^p - \gamma_1t - (\gamma + c)$. Suppose that $f$ were reducible, then $f$ would be the product of $p$ linear factors as $p$ is prime, hence $f$ is irreducible if and only if $V_p(\alpha) - \gamma_1\alpha - (\gamma + c) \neq 0$ for
any $\alpha \in D$, by Lemma \ref{C7.3: Lemma 1}.\\
On the other hand, if $k = p$, then $h(t)$ is equal to a product of $p$ linear factors in $R$, all of which are similar to one another. Also, since $\frac{p}{k} = \frac{m}{l}$
and $p = k$, we have $m = l \leq p$. Hence
$f$ is the product of $m \leq p$ linear factors in R, all of which are mutually similar to each other.
So there exist constants $\alpha_1, \alpha_2, \dots , \alpha_m \in D^{\times}$ such that $f(t) = \prod\limits_{i=1}^m (t - \alpha_i)$, and $(t - \alpha_i) \sim (t - \alpha_j)$ for all $i,j \in \{1, 2, \dots , m\}$. In particular $(t - \alpha_i) \sim (t  \alpha_m)$ for all $i \neq m$, which is true if and only if there exist constants $c_1, c_2, \dots , c_{m-1}, c_m \in D^{\times}$ such that $\alpha_i = \Omega_{c_i}(\alpha_m)$ for
all $i$ by Lemma \ref{C7.2: Lemma 1}. Hence setting $\alpha = \alpha_m$ and $c_m = 1$ yields $f(t) = \prod\limits_{i=1}^m (t - \Omega_{c_i}(\alpha))$. Finally, we note that $(t - \alpha)$ right divides $t^p - \gamma_1t - (\gamma + c)$ if and only if $V_p(\alpha) - \gamma_1\alpha - (\gamma + c) = 0$ by
Lemma \ref{C7.3: Lemma 1}. 
\end{proof}

In the following, $\mathbb{K}$ denotes a finite field of characteristic $p$ and $\delta$ is an algebraic derivation of $\mathbb{K}$ with fixed field $\mathbb{F}$ and minimum polynomial $g(t) = t^{p^e} + \gamma_1t^{p^{e-1}} + \cdots + \gamma_et \in \mathbb{F}[t]$ such that $g(\delta) = 0$. Then $\mathbb{K}[t;\delta]$ has centre $\mathbb{F}[g(t)] \cong \mathbb{F}[x]$. In this particular case we obtain a more detailed description on the generalised A-polynomials in $\mathbb{K}[t;\delta]$.

\begin{theorem}
Suppose that $f \in \mathbb{K}[t;\delta]$ has degree $m \geq 1$ and minimal central left multiple $h(t) = \hat{h}(g(t))$ for some irreducible polynomial $\hat{h} \in \mathbb{F}[x]$. Then $f$ is a generalised A-polynomial if and only if $h(t)=g(t)-\gamma$ for some $\gamma \in \mathbb{F}$, $m \leq p^e$, and $$f(t) = \prod\limits_{i=1}^m (t-\Omega_{c_i}(\alpha)),$$ for some $c_1,c_2,\dots,c_{m-1},c_m,\alpha \in \mathbb{K}^{\times}$ with $c_m=1$ (w.l.o.g.), such that $$V_{p^e}(\alpha) + \gamma_1 V_{p^{e-1}}(\alpha) + \cdots + \gamma_e \alpha - \gamma = 0.$$ In particular, $f$ is a reducible polynomial in $\mathbb{K}[t;\delta]$ unless $m=1$.
\end{theorem}
\begin{proof}
First we note that $E_{\hat{h}}= \frac{\mathbb{F}[x]}{(\hat{h}(x))}$ is a field extension of $\mathbb{F}$ of finite degree, hence it is also a finite field. In the notation of Theorem \ref{C3.2: Theorem 4}, $\mathcal{E}(f) \cong M_l(E_{\hat{h}})$ since the only central division algebra over the finite field $E_{\hat{h}}$ is $E_{\hat{h}}$ itself. Hence $\mathcal{E}(f)$ is a central simple algebra over $E_{\hat{h}}$ of degree $s=l$. By Theorem \ref{C3.2: Theorem 4}, ${\rm deg}(h) =\frac{mp^e}{s}= \frac{mp^e}{l}$. On the other hand, we know that ${\rm deg}(h) = k{\rm deg}(f_i)$ where $f_i$ is any irreducible divisor of $f$ in $R$, and ${\rm deg}(f_i) = \frac{m}{l}$. Combining the expressions for ${\rm deg}(h)$ yields $p^e = k$.\\
Now by Theorem \ref{C7.3: Theorem 7}, $f$ is a generalised A-polynomial in $R$ if and only if  $f \vert_r (g(t)-\gamma)$ for some $\gamma \in \mathbb{F}$ which is true if and only if $h(t) = g(t)-\gamma$ for some $\gamma \in \mathbb{F}$, i.e. ${\rm deg}(\hat{h})=1$. Since ${\rm deg}(\hat{h})=\frac{m}{s}=\frac{m}{l}$ we have that $m=l$, that is $f$ is the product of $m$ irreducible polynomials in $R$ that are all similar to each other. Hence there exist elements $\alpha_1,\alpha_2,\dots,\alpha_m \in \mathbb{K}$ such that $$f(t)= \prod\limits_{i=1}^m (t-\alpha_i)$$ and elements $c_1,c_2,\dots,c_m \in \mathbb{K}^{\times}$ such that $\Omega_{c_i}(\alpha_m) = \alpha_i$ for each $i$ by Lemma \ref{C7.2: Lemma 1}. We use the notation $\alpha = \alpha_m$ and without loss of generality we may take $c_m=1$. Finally we note that $(t-\alpha)$ right divides $h(t)=g(t)-\gamma$ if and only if $V_{p^e}(\alpha) + \gamma_1 V_{p^{e-1}}(\alpha) + \cdots + \gamma_e \alpha - \gamma = 0$ by Lemma \ref{C7.3: Lemma 2}. 
\end{proof}.


\bibliography{ThesisDraftJan21Owen}
\bibliographystyle{alpha}

\end{document}